\documentclass[pdflatex,sn-mathphys-num]{sn-jnl}

\usepackage{graphicx}%
\usepackage{multirow}%
\usepackage{amsmath,amssymb,amsfonts}%
\usepackage{amsthm}%
\usepackage{mathrsfs}%
\usepackage[title]{appendix}%
\usepackage{xcolor}%
\usepackage{textcomp}%
\usepackage{manyfoot}%
\usepackage{booktabs}%
\usepackage{algorithm}%
\usepackage{algorithmicx}%
\usepackage{algpseudocode}%
\usepackage{listings}%

\theoremstyle{thmstyleone}%
\newtheorem{theorem}{Theorem}[section]
\newtheorem{proposition}{Proposition}[section]%
\newtheorem{corollary}{Corollary}[section]
\newtheorem{lemma}{Lemma}[section]

\theoremstyle{thmstyletwo}%
\newtheorem{example}{Example}[section]%
\newtheorem{remark}{Remark}[section]%

\theoremstyle{thmstylethree}%
\newtheorem{definition}{Definition}[section]%
\newtheorem{assumption}{Assumption}[section]

\usepackage{color}
\usepackage{hyperref}
\usepackage{multirow}
\usepackage{mathtools}
\usepackage{enumerate}
\usepackage{makecell}
\def\R{{\mathbb{R}}}

\numberwithin{equation}{section}

\begin{document}

\title[A smoothing extended sequential quadratic method for difference-of-convex optimization over a convex composite inequality constraint]
{A smoothing extended sequential quadratic method for difference-of-convex optimization over a convex composite inequality constraint}


\author[1]{\fnm{Jiefeng} \sur{Xu}}\email{jiefeng.xu@polyu.edu.hk}

\author[1]{\fnm{Ting Kei} \sur{Pong}}\email{tk.pong@polyu.edu.hk}

\author*[2]{\fnm{Yongle} \sur{Zhang}}\email{yongle-zhang@163.com}

\affil[1]{\orgdiv{Department of Applied Mathematics}, \orgname{the Hong Kong Polytechnic University}, \city{Hong Kong}, \country{People’s Republic of China}}

\affil*[2]{School of Mathematical Sciences, Sichuan Normal University, Chengdu, Sichuan, People’s Republic of China}


\abstract{We consider the problem of minimizing a difference-of-convex objective over a convex composite inequality constraint and a compact convex set constraint. To solve this problem, we extend the ESQM in \cite{13Auslender} via incorporating a variable smoothing scheme. In essence, in each iteration of our algorithm, we apply {\em one} proximal gradient step to a {\em smoothed} penalty function, constructed based on a smooth approximation of the convex composite constraint function; and we design explicit rules to update the smoothing and penalty parameters. Under suitable constraint qualifications, we establish an iteration complexity of $O(\epsilon^{-3})$ for obtaining an $(\epsilon,\epsilon)$-KKT point. Moreover, in the convex setting, we show that the whole sequence generated by our algorithm is convergent and derive its local convergence rate under a standard Hölderian growth condition.}




\maketitle

\section{Introduction}
	Let ${\mathbb X}$ and ${\mathbb Y}$ be two finite-dimensional real Hilbert spaces.
	In this paper, we consider the following convex-composite-constrained
	difference-of-convex (DC) minimization problem
	\begin{equation}\label{opt-dc-hc}
		\begin{array}{rl}
			\min\limits_{x\in \mathbb{X}}
			& \psi(x) \coloneqq f(x) + P_1(x) - P_2(x) + \delta_{\mathcal{C}}(x)
			\\ \text{s.t.}& g(x) \coloneqq h(c(x)) \le 0,
		\end{array}
	\end{equation}
	where $\mathcal{C}\subseteq \mathbb{X}$, $f$, $P_1$, $P_2$, $h$ and $c$ satisfy Assumption~\ref{ass-gen} below.\footnote{See Section~\ref{sec2} for notation.}
	\begin{assumption}\label{ass-gen} In problem \eqref{opt-dc-hc},
		\begin{enumerate}[{\rm (i)}]
			\item the set $\mathcal{C}\subseteq \mathbb{X}$ is compact and convex, and the set $\mathcal{F} \coloneqq \{x \in \mathcal{C} : g(x) \le 0 \}\neq \emptyset$;
			\item the function $f:{\mathbb X}\rightarrow \mathbb{R}$, and the mapping
			$c:{\mathbb X}\rightarrow {\mathbb Y}$ are continuously differentiable, and they are $L_{f}$-smooth and $L_c$-smooth on ${\cal C}$ with constants $L_f>0$ and $L_c\ge 0$, respectively; moreover, $c$ is $M_c$-Lipschitz continuous on ${\cal C}$ with constant $M_c > 0$.
			\item the functions $P_1, P_2 :{\mathbb X}\rightarrow \mathbb{R}$ are convex;\footnote{As is commonly assumed in the literature, we assume that the proximal mapping of $\gamma P_1 + \delta_{\cal C}$ can be computed efficiently for any $\gamma > 0$, and that one can compute an element of $\partial P_2(x)$ for any $x\in \mathbb{X}$.} 	
			\item the function $h :{\mathbb Y}\rightarrow \mathbb{R}$ is convex
			and $M_{h}$-Lipschitz continuous with constant $M_{h}>0$.
		\end{enumerate}
	\end{assumption}
	\noindent Note that the solution set of \eqref{opt-dc-hc} is nonempty under Assumption~\ref{ass-gen}.
	
	Problems of the form \eqref{opt-dc-hc} arise naturally in many contemporary applications \cite{09BDE,15YLHX,25LiYuYaoTongLiangLinYang,21VogelBelletClemencon}, where $f$ is typically a loss function for data fidelity, $P_1$ and $P_2$ are regularizers for inducing desirable structures in the solution, and the inequality constraint $g(x)\le 0$ is there to impose explicit restrictions on the solution. Notice that the Lipschitz continuous convex function $h$ in \eqref{opt-dc-hc} can be leveraged to model {\em any} convex constraints on $c(x)$, i.e., $c(x)\in {\cal D}$ for some closed convex set ${\cal D}$: indeed, one can simply choose $h$ to be the distance function ${\rm dist}(y,{\cal D}) \coloneqq \inf_{w\in {\cal D}}\|y - w\|$. Nevertheless, for algorithmic design, it can be desirable to choose $h$ such that $\inf h < 0$.\footnote{For example, the choice of $h(\cdot) = {\rm dist}(\cdot,{\cal D})$ violates Assumption~\ref{ass-cq} below.} Here, we present two such examples of $h$ and the kind of constraint they can model.
	\begin{example}[Solid convex sets]
		Let $\mathcal{D} \subsetneq {\mathbb Y}$ be a closed convex set with $0 \in \mathrm{int}(\mathcal{D})$.
		Then, from \cite[Theorem 14.5 and Corollary 14.5.1]{70Ro}, it follows that
		the polar ${\cal D}^{\circ}\coloneqq \{y\in \mathbb{Y}: \langle u,y\rangle\le 1\ \ \forall u\in {\cal D}\}$ is a compact convex set and contains the origin; moreover, ${\cal D}^\circ\neq \{0\}$ as ${\cal D}\neq \mathbb{Y}$.
		Let $h(y) \coloneqq \sigma_{\mathcal{D}^{\circ}}(y) - 1$ for any $y \in \mathbb{Y}$.
		Then, for every $y \in \mathbb{Y}$, it holds that
		\begin{equation*}
			y \in \mathcal{D}\quad \Longleftrightarrow \quad
			h(y) \le 0,
		\end{equation*}
		meaning that the constraint $c(x)\in {\cal D}$ can be reformulated as the constraint in \eqref{opt-dc-hc} with $h(\cdot) \coloneqq \sigma_{\mathcal{D}^{\circ}}(\cdot) - 1$.
		Furthermore, since $\mathcal{D}^{\circ}$ is nonempty and compact, the function $\sigma_{\mathcal{D}^{\circ}}$ is real-valued, convex, and $M_{\mathcal{D}^{\circ}}$-Lipschitz continuous with $M_{\mathcal{D}^{\circ}} \coloneqq \max_{u\in \mathcal{D}^{\circ}} \|u\| > 0$ (recall that ${\cal D}^\circ \neq \{0\}$). Therefore, Assumption~\ref{ass-gen}(iv) holds for $h(\cdot) = \sigma_{\mathcal{D}^{\circ}}(\cdot) - 1$, and we have $\inf h \le h(0) = -1 < 0$.
	\end{example}
	\begin{example}[Solid and pointed cones]\label{ex-conic constraint}
		Let $\mathcal{K} \subseteq {\mathbb Y}$ be a closed convex pointed cone with nonempty interior.
		Then, the constraint $c(x) \in \mathcal{K}$ can be reformulated as the constraint
		in \eqref{opt-dc-hc} with $h\coloneqq \sigma_{\mathcal{B}}$,
		where $\mathcal{B}$ is a compact base\footnote{A base for a cone $\mathcal{Q}\subseteq \mathbb{Y}$ is a convex set $\mathcal{D}$ with $0\notin \mathrm{cl}(\mathcal{D})$ and $\mathcal{Q} = \cup_{t\ge 0} t\mathcal{D}$.} of the polar cone $\mathcal{K}^{\circ}\coloneqq \{y\in \mathbb{Y}: \langle u,y\rangle \le 0\ \ \forall u\in {\cal K}\}$; see, e.g., \cite[Page~60, Exercise 20]{06BL} for the existence of base.
		Moreover, since $\mathcal{B}$ is compact and does not contain the origin, the support function $\sigma_{\mathcal{B}}$ is real-valued, convex, and $M_{\mathcal{B}}$-Lipschitz continuous with $M_{\mathcal{B}} \coloneqq \max_{u\in \mathcal{B}} \|u\| > 0$. Hence, Assumption~\ref{ass-gen}(iv) holds for $h = \sigma_{\mathcal{B}}$, and we have $\sigma_{\cal B}(y) < 0$ if and only if $y\in {\rm int}({\cal K})$.
	\end{example}
	Existing works on \eqref{opt-dc-hc} mainly adopted penalty-based or augmented Lagrangian-based approaches; see e.g., \cite{04CR, 22DJKM, 07KF, 05KNKF,13Auslender,15Auslender,18BST,22LMX,23KMM,23JKMW,26DGO}. One representative work is \cite{15Auslender}, which considered \eqref{opt-dc-hc} with $P_1 = P_2 = 0$ and $h$ chosen to model conic constraints. Specifically, the author in \cite{15Auslender} considered some nonnegative finite convex functions $h$ that are zero precisely inside the closed convex cone, and proposed the sequential linear cone method (SLCM). This method can be seen as a generalization of the extended sequential quadratic method (ESQM) in \cite{13Auslender} designed for the case when $\mathbb{Y} = \R^m$ and $h(y) = \max_{1\le i\le m}y_i$, and has close connections with various exact penalty methods for conic programs; see, e.g., \cite{04CR, 07KF, 05KNKF}. A basic version of SLCM can be described as follows: in each iteration, for some suitably chosen $\gamma_k > 0$ and $L_k > 0$, one minimizes
	\begin{equation}\label{subproblem}
		\gamma_k\left(\langle\nabla f(x^k),d\rangle + \frac{L_k}{2}\|d\|^2\right) + h^+(c(x^k) + Dc(x^k)d) + \delta_{\cal C}(x^k+d)
	\end{equation}
	to obtain $d^k$, where $h^+(y) \coloneqq \max\{h(y),0\}$, and then sets $x^{k+1} = x^k + \tau_k d^k$; here, $\tau_k$ is chosen via linesearch to induce a sufficient descent on the function $x\mapsto \gamma_k f(x) + h^+(c(x))$, and $\gamma_k$ is updated adaptively according to whether $h^+(c(x^k) + Dc(x^k)d^k) = 0$, see \cite[Section~3.2]{15Auslender}. It was shown in \cite[Section~3.3]{15Auslender} that, under suitable constraint qualifications, $\gamma_k$ stays constant for all sufficiently large $k$ and the sequence $\{x^k\}$ generated by SLCM clusters at a stationary point of the special case of \eqref{opt-dc-hc} studied therein.
	
	Note that one can readily extend the algorithmic framework of SLCM to cover the general problem \eqref{opt-dc-hc} under Assumption~\ref{ass-gen}, in view of recent studies of ESQM-type methods on optimization problems with DC objectives (see, e.g., \cite{25ZPX}). However, the subproblems that arise can be difficult to solve. Indeed, even when $P_1 = P_2 = 0$, the corresponding subproblem \eqref{subproblem} of SLCM requires an iterative solver in general and can only be solved {\em inexactly}. Difficult subproblems also arise naturally when one applies other penalty-based / augmented-Lagrangian-based approaches; see, e.g., \cite{16BP,18BST,22LMX,23KMM,23JKMW}. Recently, attempts have been made in \cite{25LiuXu,YangLiHuLinYang26} to avoid difficult subproblems,
	and their settings cover \eqref{opt-dc-hc} with $P_1 = P_2 = 0$,
	$\mathbb{Y} = \mathbb{R}^m$ and $h(y) = \sum_{i=1}^m\max\{y_i,0\}$.
	\footnote{The methods in \cite{25LiuXu,YangLiHuLinYang26} were proposed more generally to handle stochastic constraints and possibly nonsmooth weakly convex $f$ and $c$, and \cite{25LiuXu} also incorporated a SPIDER-type variance reduction technique. Their complexity results were derived under the more general setting.}
	Specifically, in \cite{YangLiHuLinYang26} where it was assumed that ${\cal C} = {\mathbb{R}}^n$ (instead of being compact), upon fixing a desired tolerance $\epsilon > 0$, the authors estimated an exact penalty parameter $\bar\beta$, and applied the subgradient method to the function $f(\cdot) + \bar\beta h(c(\cdot))$; see \cite[Algorithm~1]{YangLiHuLinYang26}.
	Their algorithm can find a nearly $\epsilon$--Karush-Kuhn-Tucker (KKT) point (in the sense of \cite[Definition~3.3]{YangLiHuLinYang26}) in expectation within $O(\epsilon^{-6})$ iterations; see \cite[Corollary~5.4]{YangLiHuLinYang26}.
	The authors in \cite{25LiuXu} 
	further incorporated a smoothing technique and applied the gradient projection method to the function\footnote{Notice that the gradient projection method can be applied efficiently if the gradients of $f$ and $H_{\bar \mu}\circ c$ can be computed efficiently and the projection onto ${\cal C}$ admits an efficient computational routine.}
	\[
	f(x) + \bar\beta H_{\bar \mu}(c(x))+ \delta_{\cal C}(x),
	\]
	where $H_{\bar \mu}$ is a Huber-type smoothing function for $h$; see \cite[Section~1.2]{25LiuXu}.
	They established in \cite[Section~3.3]{25LiuXu} an iteration complexity of $O(\epsilon^{-4})$ to obtain an $(\epsilon, \epsilon)$-KKT point in expectation (in the sense of \cite[Definition~1.3]{25LiuXu}).
	However, the methods in \cite{25LiuXu,YangLiHuLinYang26} require the target tolerance to be specified in advance as well as an estimation of $\bar \beta$ (and also $\bar \mu$ for \cite{25LiuXu}). Can we develop an exact-penalty-type method for \eqref{opt-dc-hc} that {\em does not require the tolerance} or estimation of the exact penalty parameter {\em as input}, and its subproblems {\em do not require iterative solvers}?
	
	In this paper, we develop such an algorithm by extending ESQM in \cite{13Auslender} via incorporating a {\em variable smoothing} scheme to solve \eqref{opt-dc-hc}. Specifically, in each iteration of our algorithm, for some suitably chosen $\gamma_k > 0$ and $\mu_k >0$, upon taking a $\xi^k\in \partial P_2(x^k)$, we apply {\em one} proximal-gradient step (with a suitable stepsize) to the function
	\begin{equation*}
		f(x)-\langle\xi_k,x - x^k\rangle + \gamma_k^{-1} H_{\mu_k}(c(x)) + (P_1 + \delta_{\cal C})(x)
	\end{equation*}
	to obtain $x^{k+1}$, where $H_\mu$ is a suitable smooth approximation of $h^+$. Then $\mu_k$ and $\gamma_k$ are updated according to some specific rules. In particular, we enforce $\lim_{k\to\infty}\mu_k = 0$ so that our approach is in line with recent works on variable smoothing methods (see, e.g., \cite{21BW,20BB,18TDFC,15BH} and references therein), and update $\gamma_k$ appropriately to induce feasibility asymptotically. Under a standard constraint qualification and suitably designed $\{\mu_k\}$, we show that $\gamma_k$ stays constant for all sufficiently large $k$ according to our update rule (a phenomenon also observed in other variants of ESQM \cite{13Auslender,15Auslender,25ZPX}), and establish an iteration complexity of $O(\max\{\epsilon_1^{-\frac{2}{1-r}}, \epsilon_2^{-\frac{1}{r}}\})$ when $\mu_k = \mu_0(1+k)^{-r}$ for all $k$ with some constants $\mu_0 > 0$ and $r \in (0,1)$, for obtaining an $(\epsilon_1,\epsilon_2)$-KKT point of \eqref{opt-dc-hc} (see Definition~\ref{def_AKKT}).
	In particular, by setting $r = \frac{1}{3}$ and $\epsilon_1=\epsilon_2$, the resulting complexity, namely $O(\epsilon_1^{-3})$, matches that of the variable smoothing method for unconstrained weakly convex composite minimization problems in \cite{21BW}. Next, in the convex setting (i.e., when $f$ is convex, $P_2 = 0$ and $c$ is $(-{\rm hzn}\,h)$-convex), we present explicit rules for choosing $\{\mu_k\}$ that guarantee asymptotic convergence of the sequence generated by our algorithm, and establish its local convergence rate under an additional H\"olderian error bound condition. 
	
	The rest of the paper is organized as follows. We present notation and preliminaries in Section~\ref{sec2}, including optimality conditions of \eqref{opt-dc-hc} under suitable constraint qualifications and properties of smoothing approximations. Our algorithm is presented in Section~\ref{sec3}, where we also prove its well-definedness and establish its iteration complexity for obtaining an $(\epsilon_1,\epsilon_2)$-KKT point. We study asymptotic convergence under the convex setting in Section~\ref{sec4}. Finally, we present numerical experiments in Section~\ref{sec5}.

\section{Notation and preliminaries}\label{sec2}
Throughout this paper, $\mathbb{X}$ and $\mathbb{Y}$ are finite-dimensional real Hilbert spaces, and we denote the standard inner product and the induced norm of the underlying Hilbert space by $\langle \cdot, \cdot \rangle$ and $\|\cdot\|$, respectively. For $p \in [1,\infty]$, we let $\|\cdot\|_{p}$ denote the standard $\ell_{p}$-norm on $\mathbb{R}^{n}$, i.e., $\|x\|_p \coloneqq (\sum_{i=1}^n|x_i|^p)^\frac1p$. The set of natural numbers is denoted by $\mathbb{N}$, and we write ${\mathbb{N}}_0\coloneqq \mathbb{N}\cup\{0\}$.
For a positive integer $m$, we write $[m]\coloneqq \{1,2,\ldots, m\}$.
	
	For a closed convex set $\mathcal{D} \subseteq \mathbb{X}$, its horizon (or recession) cone
	is the closed convex cone given by
	$$
	\mathcal{D}^{\infty}\coloneqq \{d\in \mathbb{X}: x + td\in {\cal D} \ \ \forall t > 0\},
	$$
	where $x$ is any point in ${\cal D}$, and the above definition is known to be independent of the choice of $x\in {\cal D}$.
	We let $\delta_{\cal D}$ denote the indicator function of ${\cal D}$, i.e., $\delta_{\cal D}$ is zero in ${\cal D}$ and infinity otherwise, and we denote the support function of ${\cal D}$ by $\sigma_{\cal D}$, which is defined as $\sigma_{\cal D}(x) \coloneqq \sup\{\langle x,u\rangle: u\in {\cal D}\}$ for all $x\in \mathbb{X}$. For a convex cone ${\cal Q}\subseteq \mathbb{X}$, its polar is defined as ${\cal Q}^\circ \coloneqq \{u\in \mathbb{X}: \langle u,x\rangle\le 0\ \ \forall x\in {\cal Q}\}$.
	
	For a function $\phi : \mathbb{X} \to \mathbb{R}$, we define its \emph{positive part} $\phi^{+} : \mathbb{X} \to \mathbb{R}_{+}$ by
	$
	\phi^{+}(x) \coloneqq \max\{\phi(x),0\}
	$
	for all $x \in \mathbb{X}$.
	A function $\phi : \mathbb{X} \to (-\infty,\infty]$ is said to be \emph{proper} if its effective domain
	$
	\mathrm{dom}\,\phi \coloneqq \{x \in \mathbb{X} : \phi(x) < \infty\}
	$
	is nonempty, and \emph{closed} if it is lower semicontinuous. Following \cite[Definition~8.3]{98RW}, for  a proper function
	$\phi : \mathbb{X} \to (-\infty,\infty]$ and an $x \in \mathrm{dom}\,\phi$,
	we call $\xi \in \mathbb{X}$ a \emph{regular subgradient} of $\phi$ at $x$ if
	\[
	\phi(z) \ge \phi(x) + \langle \xi, z - x \rangle + o(\|z - x\|)
	\quad \text{as } z \to x;
	\]
	the set of regular subgradients of $\phi$ at $x$ is denoted by  $\widehat{\partial}\phi(x)$. We call $\xi \in \mathbb{X}$ a (limiting) \emph{subgradient} of $\phi$ at an $x\in {\rm dom}\,\phi$ if there exist sequences $\{x^t\}$ and $\{\xi^t\}$ such that $x^{t} \to x$, $\phi(x^{t}) \to \phi(x)$, $\xi^{t} \to \xi$ and $\xi^{t} \in \widehat{\partial}\phi(x^{t})$ for all $t$, and we denote the set of subgradients of $\phi$ at $x$ by $\partial \phi(x)$; this set is called the subdifferential of $\phi$ at $x$. 
	For a proper convex function $\phi : \mathbb{X} \to (-\infty,\infty]$, it is known (see, e.g., \cite[Proposition~8.12]{98RW}) that its \emph{subdifferential} at $x \in \mathrm{dom}\, \phi$ is given by
	\[
	\partial \phi(x)
	= \left\{ \xi \in \mathbb{X} :
	\phi(z) \ge \phi(x) + \langle \xi, z - x \rangle
	\quad \forall z \in \mathbb{X} \right\};
	\]
	moreover, for $\epsilon \ge 0$, the \emph{$\epsilon$-subdifferential} of $\phi$ at $x$ is defined by
	\[
	\partial_{\epsilon} \phi(x)
	\coloneqq \left\{ \xi \in \mathbb{X} :
	\phi(z) \ge \phi(x) + \langle \xi, z - x \rangle - \epsilon
	\quad \forall z \in \mathbb{X} \right\}.
	\]
	Clearly, $\partial_{0}\phi(x) = \partial\phi(x)$.
	For a proper closed convex function $\phi : \mathbb{X} \to (-\infty,\infty]$, the \emph{horizon function} $\phi^{\infty}$ is defined via its epigraph by
	$
	\mathrm{epi}\,\phi^{\infty} \coloneqq (\mathrm{epi}\,\phi)^{\infty},
	$
	where $\mathrm{epi}\, \phi \coloneqq \{(x, \lambda) \in \mathbb{X} \times \mathbb{R} : \phi(x) \le \lambda \}$,
	and the \emph{horizon cone} of $\phi$ is defined as
	\[
	\mathrm{hzn}\,\phi
	\coloneqq \left\{ d \in \mathbb{X} :
	\phi(x + d) \le \phi(x)
	\quad \forall x \in \mathrm{dom}\,\phi \right\};
	\]
	moreover, the conjugate function $\phi^*$ of $\phi$ is defined by $\phi^*(\xi) \coloneqq \sup_{x \in \mathbb{X}} \{\langle \xi, x \rangle - \phi(x)\}$ for all $\xi\in \mathbb{X}$. Finally, the normal cone of a closed convex set $\mathcal{D} \subseteq \mathbb{X}$ at $x\in \mathcal{D}$ is defined as ${\cal N}_{\cal D}(x) \coloneqq \partial \delta_{\cal D}(x)$, and the tangent cone of $\mathcal{D}$ at $x\in \mathcal{D}$ is ${\cal T}_{\cal D}(x)\coloneqq({\cal N}_{\cal D}(x))^\circ$.

\subsection{Optimality conditions}

	We first derive optimality conditions for \eqref{opt-dc-hc} under the following constraint qualification (CQ). This CQ is a direct adaptation of the ECQC($x$) in \cite[Eq.~(2.44)]{15Auslender}, which was proposed for the case when $h \in {\cal G}_K$ defined in \cite[Eq.~(1.7)~\&~(1.8)]{15Auslender}.
	\begin{assumption}[CQ]
		\label{ass-cq}
		For problem \eqref{opt-dc-hc},
		\begin{equation}
			\label{eq-infeas-cq}
			0 \notin \partial g(x) + \mathcal{N}_{\mathcal{C}}(x)
			\quad \forall x\in {\cal C}\ {\rm with}\  g(x) \ge 0.
		\end{equation}
	\end{assumption}
	\begin{remark}\label{Remark_2}
		Note that under Assumption~\ref{ass-gen}, the condition \eqref{eq-infeas-cq} can be equivalently rewritten as
		\[
		0 \notin \partial (g + \delta_{\cal C})(x)
		\quad \forall x\in {\cal C}\ {\rm with}\  g(x) \ge 0.
		\]
		This is because $\partial (g + \delta_{\cal C})(x) = \partial g(x) + \mathcal{N}_{\mathcal{C}}(x)$, thanks to \cite[Exercise~10.26(a)~\&~(b)]{98RW}, and the fact that real-valued convex functions and closed convex sets are amenable in the sense of \cite[Definition~10.23]{98RW}.
	\end{remark}
	
	Suppose that Assumptions~\ref{ass-gen} and \ref{ass-cq} hold.
	Noting that real-valued convex functions are amenable in the sense of \cite[Definition~10.23]{98RW}, we see from \cite[Exercise~10.26(b)]{98RW} that
	\begin{equation}\label{eq-part ghc}
		\partial g(x) = D c(x)^{*} \partial h(c(x)) \quad \forall x\in \mathbb{X}.
	\end{equation}
	Hence, whenever $x\in {\cal C}$ and $g(x) = 0$, we have
	\[
	{\cal N}_{\cal F}(x)\overset{\rm (a)}\subseteq \bigcup_{\lambda > 0}\lambda \partial(g + \delta_{\cal C})(x)\cup {\cal N}_{\cal C}(x)
	\overset{\rm (b)}\subseteq \bigcup_{\lambda \ge 0}\lambda D c(x)^{*} \partial h(c(x)) + {\cal N}_{\cal C}(x),
	\]
	where (a) follows from Remark~\ref{Remark_2}, \cite[Proposition~10.3]{98RW} and \cite[Corollary~10.9]{98RW} (this corollary is used for showing that the horizon subdifferential of $g + \delta_C$ at $x$ belongs to ${\cal N}_{\cal C}(x)$), and (b) follows from the local Lipschitz continuity of $g$, \cite[Corollary~10.9]{98RW} and \eqref{eq-part ghc}. Using this observation, one can deduce that if $\bar x$ is a local minimizer of \eqref{opt-dc-hc}, then either $g(\bar x) < 0$ and $0\in \partial f(\bar x) + \partial P_1(\bar x) - \partial P_2(\bar x) + {\cal N}_{\cal C}(\bar x)$, or $g(\bar x) = 0$ and
	\[
	0 \in \partial f(\bar x) + \partial P_1(\bar x) - \partial P_2(\bar x) + \bigcup_{\lambda \ge 0}\lambda D c(\bar x)^{*} \partial h(c(\bar x)) + {\cal N}_{\cal C}(\bar x).
	\]
	In other words, under Assumptions~\ref{ass-gen} and \ref{ass-cq}, for any local minimizer $\bar x$ of \eqref{opt-dc-hc}, there exist $u\in \mathbb{Y}$ and $\lambda \ge 0$ such that
	\begin{align*}
		& 0 \in \nabla f(\bar x) + \partial P_1(\bar x) - \partial P_2(\bar x)
		+ \lambda Dc(\bar x)^* u
		+ \mathcal{N}_{\mathcal{C}}(\bar x),
		\\ & g(\bar x) \le 0, \quad u \in \partial h(c(\bar x)),
		\quad \lambda g(\bar x)= 0.
	\end{align*}
	This motivates the following notion of (approximate) Karush-Kuhn-Tucker (KKT) points.
	\begin{definition}[(Approximate) KKT points]\label{def_AKKT}
		Let $\epsilon_1\ge 0$ and $\epsilon_2 \ge 0$.
		A point $x \in \mathcal{C}$ is called an $(\epsilon_1, \epsilon_2)$-KKT point of \eqref{opt-dc-hc},
		if there exist $z\in \mathcal{C}$, $u \in \mathbb{Y}$, and $\lambda\ge 0$ such that
		\begin{align*}
			& \mathrm{dist}\left(0, \nabla f(x) + \partial P_1(x) - \partial P_2(z)
			+ \lambda Dc(x)^* u
			+ \mathcal{N}_{\mathcal{C}}(x)\right) \le \epsilon_1,
			\\ & g(x) \le \epsilon_2, \quad u \in \partial_{\epsilon_2}h(c(x)),
			\quad |\lambda g(x)| \le \epsilon_2
			\quad \text{and} \quad
			\|x-z\| \le \epsilon_1 \epsilon_2.
		\end{align*}
		When $\epsilon_1=\epsilon_2$, an $(\epsilon_1, \epsilon_2)$-KKT point
		is referred to as an $\epsilon_1$-KKT point;
		moreover, we refer to a $0$-KKT point simply as a KKT point.
	\end{definition}
	
	From the above discussion, we see that under Assumptions~\ref{ass-gen} and \ref{ass-cq}, any local minimizer $\bar x$ of \eqref{opt-dc-hc} is a KKT point. We would like to point out that the subdifferential of $P_2$ is computed at a $z$ close to $x$ (rather than at $x$); this flexibility is introduced to account for the fact that $\partial P_2(\cdot)$ is not continuous in general; see also \cite[Definition~2.8]{25XPS}.
	In this paper, we will analyze the iteration complexity of the algorithm we propose later for obtaining $(\epsilon_1, \epsilon_2)$-KKT points for \eqref{opt-dc-hc} in the sense of Definition~\ref{def_AKKT}, under Assumptions~\ref{ass-gen} and \ref{ass-cq}.
	
	We now take a closer look at Assumption~\ref{ass-cq}. The next remark explains the connection between Assumption~\ref{ass-cq} and a standard constraint qualification for conic programs (which is a special case of \eqref{opt-dc-hc} in view of Example~\ref{ex-conic constraint}). This relationship was proved in the discussions in \cite[Section~2.6]{15Auslender}, and we include a short proof for the convenience of readers.
	
	\begin{remark}[CQ for conic programming]
		Consider \eqref{opt-dc-hc} with $h$ given in Example~\ref{ex-conic constraint}.
		Suppose that Assumption~\ref{ass-gen}(i)--(iii) hold.
		We claim that, for every $x\in \mathcal{C}$ with $g(x) = 0$,
		the condition
		\begin{equation*}
			0 \notin \partial g(x) + \mathcal{N}_{\mathcal{C}}(x)
		\end{equation*}
		is satisfied if and only if the following Robinson's constraint qualification (RCQ) holds:
		\begin{equation*}
			0 \in \mathrm{int}\left( c(x) + Dc(x) \mathcal{T}_{\mathcal{C}}(x) - \mathcal{K}\right).
		\end{equation*}
		
		Indeed, let $x\in \mathcal{C}$ with
		$\sigma_{\mathcal{B}}(c(x)) = 0$.
		We have that
		\begin{align*}
			& 0 \in \mathrm{int}\left( c(x) + Dc(x) \mathcal{T}_{\mathcal{C}}(x) - \mathcal{K}\right)
			\Longleftrightarrow 0 \in \mathrm{int}\left( \mathcal{T}_{\mathcal{K}}(c(x)) - Dc(x) \mathcal{T}_{\mathcal{C}}(x)\right)
			\\ & \Longleftrightarrow \mathcal{T}_{\mathcal{K}}(c(x)) - Dc(x)\mathcal{T}_{\mathcal{C}}(x) = \mathbb{Y}
			\Longleftrightarrow \mathcal{N}_{\mathcal{K}}(c(x)) \cap (-Dc(x)\mathcal{T}_{\mathcal{C}}(x))^{\circ} =  \{0\}
			\\ & \Longleftrightarrow \partial \sigma_{\mathcal{B}}(c(x)) \cap (-Dc(x)\mathcal{T}_{\mathcal{C}}(x))^{\circ} =  \emptyset,
		\end{align*}
		where the last equivalence holds because it holds that $0 \notin \partial \sigma_{\mathcal{B}}(y)$ for all $y \in \mathcal{K}$,\footnote{This is true because $\sigma_{\cal B}(\hat y) < 0$ whenever $\hat y \in {\rm int}(K)$, which further implies that $\inf \sigma_{\cal B} = -\infty$.} and that
		$\mathcal{N}_{\mathcal{K}}(y) = \cup_{\lambda \ge 0} \lambda
		\partial \sigma_{\mathcal{B}}(y)$ whenever $\sigma_{\cal B}(y) = 0$.
		
		Moreover, $u \in (-Dc(x)\mathcal{T}_{\mathcal{C}}(x))^{\circ}$
		if and only if $\langle -Dc(x)^* u, d\rangle = \langle u, - Dc(x)d\rangle \le 0$ for all $d \in \mathcal{T}_{\mathcal{C}}(x)$,
		or equivalently, $0 \in Dc(x)^* u + \mathcal{N}_{\mathcal{C}}(x)$.
		Combining this fact with the above display and \eqref{eq-part ghc},
		we obtain the desired conclusion.
	\end{remark}
	
	In the next lemma, we derive an equivalent form of the CQ in Assumption~\ref{ass-cq}.
	The $\eta_1$ that arises in the lemma will be used for deriving bounds on some important numeric sequences generated by the algorithm
	we propose later; see Theorem~\ref{th-gamma-bd} below.
	\begin{lemma}[Reformulation of CQ]\label{lm-eta-cq}
		Consider \eqref{opt-dc-hc}.
		Suppose that Assumptions~\ref{ass-gen} and~\ref{ass-cq} hold.
		Then there exists a constant $\eta_{1}>0$ such that, for every $x\in \mathcal{C}$ with $g(x) \ge -\eta_1$,
		\begin{equation}\label{eq-cq-eta-lb}
			\mathrm{dist}(0, Dc(x)^*u + \mathcal{N}_{\mathcal{C}}(x)) \ge \eta_{1} \quad \forall u\in \partial_{\eta_{1}} h(c(x)).
		\end{equation}
	\end{lemma}
	
	\begin{proof}
		Suppose to the contrary that, for every $k\in \mathbb{N}_{0}$,
		there exist $x^{k} \in \mathcal{C}$ and
		$u^{k} \in \partial_{\tfrac{1}{k+1}} h(c(x^{k}))$
		such that
		\begin{equation*}
			g(x^{k}) \ge - \tfrac{1}{k+1} \quad \mbox{and} \quad \mathrm{dist}(0, Dc(x^{k})^*u^{k} + \mathcal{N}_{\mathcal{C}}(x^{k})) < \tfrac{1}{k+1}.
		\end{equation*}
		Notice that $\{x^{k}\} \subseteq \mathcal{C}$ is bounded. In addition, in view of Assumption~\ref{ass-gen}(iv) and \cite[Theorem~2.4.13]{02Zalinescu}, we see that the sequence $\{u^{k}\}$ is bounded because $u^k\in \partial_1h(c(x^k))$ for all $k$ and $\{x^{k}\}$ is bounded. Hence, there exist subsequences $\{x^{k_j}\}$ and $\{u^{k_j}\}$ such that $\lim_{j\to \infty} x^{k_j} = \bar{x}$ and $\lim_{j\to \infty} u^{k_j} = \bar{u}$ for some $\bar{x}\in \mathbb{X}$ and $\bar{u} \in \mathbb{Y}$.
		Then, using the closedness of $\mathcal{C}$ and $\mathcal{N}_{\mathcal C}$, together with the fact
		$$
		\limsup_{j \to \infty } \partial_{\tfrac{1}{k_{j}+1}} h(c(x^{k_j})) \subseteq \partial h(c(\bar{x})),
		$$
		we deduce that
		\begin{equation*}
			\bar{x} \in \mathcal{C},
			\quad g(\bar{x}) \ge 0,
			\quad \bar{u} \in \partial h(c(\bar{x})) \quad \text{and} \quad
			0 \in Dc(\bar{x})^*\bar{u} + \mathcal{N}_{\mathcal{C}}(\bar{x}),
		\end{equation*}
		which implies $0 \in \partial g(\bar{x})+ \mathcal{N}_{\mathcal{C}}(\bar{x})$ in view of \eqref{eq-part ghc},
		and hence contradicts Assumption~\ref{ass-cq}.
	\end{proof}
	
	Before ending this subsection, we present in the next proposition a concrete example that satisfies Assumption~\ref{ass-cq},
	and estimate the associated constant $\eta_1$ in Lemma~\ref{lm-eta-cq}.
	\begin{proposition}\label{prop-ex}
		Let $h:\mathbb{R}^{m} \to \mathbb{R}$ be defined by $h(y) \coloneqq \max_{i\in [m]} y_i$ for all $y\in \mathbb{R}^{m}$.
		Let $\mathcal{C}\subset \mathbb{R}^{n}$ be a compact and convex set containing the origin and ${\cal C}\neq \{0\}$.
		Let $H: \mathbb{R}^{n} \rightarrow \mathbb{R}^{m}$ be continuously differentiable on $\mathbb{R}^n$ and $L_{H}$-smooth on ${\cal C}$ with $L_{H}\ge 0$.
		Suppose that, for each $i \in [m]$,
		$H_i$ is positively homogeneous with degree $a_i\ge 1$, that is,
		$H_i(t x) = t^{a_i} H_i(x)$ for all $t \ge 0$ and $x \in \mathbb{R}^{n}$.
		Let $c(\cdot) \coloneq H(\cdot) - b$ with $b\in \mathbb{R}^{m}_{++}$.
		Then Assumption~\ref{ass-cq} holds.
		Moreover, the constant $\eta_1$ in Lemma~\ref{lm-eta-cq} can be chosen as
		\begin{equation}\label{eq-ex-eta-def}
			\eta_{1} = \min\, \left\{
			\frac{\min_{i\in [m]} a_i b_i}
			{2 \widetilde{m}_{\mathcal{C}}},
			\frac{\min_{i\in [m]} a_i b_i}
			{2(m+1)\max_{i\in [m]} a_i} \right\},
		\end{equation}
		where $\widetilde{m}_{\mathcal{C}} \coloneqq \sup_{x \in \mathcal{C}}\|x\| > 0$.
	\end{proposition}
	\begin{proof}
		Let $\eta_1$ be given in \eqref{eq-ex-eta-def}.
		Let $x\in \mathcal{C}$ with $h(c(x)) \ge - \eta_1$.
		We claim that \eqref{eq-cq-eta-lb} holds.
		
		Indeed, from the definition of $\epsilon$-subdifferential,
		we see that any $u \in \partial_{\eta_1} h(c(x))$ must belong to
		$
		\{u \in \mathbb{R}^{m}_{+} : \sum_{i=1}^{m}u_i = 1\}
		$
		and satisfy
		\begin{equation*}
			\eta_1 \ge h(c(x)) - \langle u, c(x)\rangle
			= \sum_{i=1}^{m}u_i (h(c(x)) -  c_i(x)).
		\end{equation*}
		Since $h(c(x)) -  c_i(x) \ge 0$ for all $i\in [m]$,
		the last display yields that
		\begin{equation}\label{eq-u-G-lb}
			u_i c_i(x) \ge u_i h(c(x)) - \eta_1 \ge - (1+u_i)\eta_1 \quad \forall i \in [m],
		\end{equation}
		where the last inequality holds because
		$h(c(x)) \ge -\eta_1$
		and $u\in \mathbb{R}^{m}_{+}$.
		
		Next, since $c(\cdot) = H(\cdot) - b$, we have $D c = D H$,
		which implies that
		\begin{equation}\label{eq-u-DG-x-lbd}
			\begin{split}
				& \langle u, Dc(x)x \rangle
				= \sum_{i \in [m]} u_i D H_i(x) x
				\overset{\mathrm{(a)}}{=} \sum_{i \in [m]} u_i a_i H_i(x)
				= \sum_{i \in [m]} u_i a_i (c_i(x) + b_i)
				\\ & \overset{\mathrm{(b)}}{\ge}  \sum_{i \in [m]}  a_i
				u_i b_i - \eta_1\sum_{i \in [m]} a_i (1+u_i)
				\overset{\mathrm{(c)}}{\ge} \min_{i\in [m]} a_i b_i -  (m+1)  \eta_1\max_{i\in [m]}a_i
				\overset{\mathrm{(d)}}{\ge} \tfrac{1}{2}\min_{i\in [m]}a_i b_i,
			\end{split}
		\end{equation}
		where (a) follows from Euler’s homogeneous function theorem (see, e.g., \cite[Theorem~16.5]{26Lukac}),
		(b) holds because of \eqref{eq-u-G-lb},
		(c) holds because $\sum_{i\in [m]} u_i = 1$ and $u \ge 0$,
		and (d) follows from the definition of $\eta_1$ in \eqref{eq-ex-eta-def}.
		
		Now, since $x\in \mathcal{C}$, we have, for any $\xi \in \mathcal{N}_{\mathcal{C}}(x)$,
		\begin{equation*}
			\begin{split}
				& \widetilde{m}_{\mathcal{C}} \|Dc(x)^*u + \xi\|   \ge \|Dc(x)^*u + \xi\| \|x\|  \ge \langle Dc(x)^*u + \xi, x\rangle
				\\ & = \langle u, Dc(x) x\rangle + \langle \xi, x\rangle \ge \langle u, Dc(x)x \rangle
				\ge  \tfrac{1}{2}\min_{i\in [m]} a_i b_i,
			\end{split}
		\end{equation*}
		where the third inequality
		holds because $\langle \xi, x\rangle \ge 0$ (since $\xi \in {\cal N}_{\cal C}(x)$ and $0\in {\cal C}$),
		and the last inequality follows from \eqref{eq-u-DG-x-lbd}.
		This implies that $\mathrm{dist}(0, Dc(x)^*u + \mathcal{N}_{\mathcal{C}}(x)) \ge \eta_{1}$, and hence \eqref{eq-cq-eta-lb} holds.
		
		In particular, we have also shown that \eqref{eq-infeas-cq} holds.
		This completes the proof.
	\end{proof}

\subsection{Smoothing approximation}
The following definition of smoothing approximation is taken from \cite[Definition~2.1]{12BT}, with minor changes in notation for the convenience of our analysis. Particularly, we explicitly include the possibly negative (sub)parameters $\alpha_{3, 1}^{\phi}$ and $\alpha_{3,2}^{\phi}$ in the specification of smoothing approximations, which will appear in our analysis in Sections~\ref{sec3} and~\ref{sec4}.
	
	\begin{definition}[Smoothing approximation]
		\label{def-sm}
		Let $\phi:{\mathbb Y}\to \mathbb{R}$ be a convex function and let $\alpha_1^{\phi}\ge 0$, $\alpha_2^{\phi} > 0$ and $\alpha_{3, 1}^{\phi}\in \mathbb{R}$ and $\alpha_{3,2}^{\phi}\in \mathbb{R}$ satisfy
		$\alpha_{3}^{\phi} \coloneqq \alpha_{3,1}^{\phi} + \alpha_{3,2}^{\phi}> 0$.
		The function $\phi$ is said to be $(\alpha_{1}^{\phi}, \alpha_{2}^{\phi}, \alpha_{3,1}^{\phi},\alpha_{3,2}^\phi)$-smoothable
		if for every $\mu>0$ there exists a convex function $\phi_{\mu} \in C^{1}({\mathbb Y})$
		such that
		\begin{enumerate}[{\rm(i)}]
			\item $\nabla \phi_{\mu}$ is $(\alpha_{1}^{\phi} + \alpha_{2}^{\phi} / \mu)$-Lipschitz continuous;
			
			\item $\phi(y) - \alpha_{3,1}^{\phi} \mu \leq \phi_{\mu}(y) \leq \phi(y) +  \alpha_{3,2}^{\phi} \mu$
			for every $y \in \mathbb{Y}$;
		\end{enumerate}
		Such a family $\{\phi_{\mu}\}_{\mu>0}$ is called a smoothing approximation
		(SA) of $\phi$ with parameters $(\alpha_1^{\phi}, \alpha_2^{\phi}, \alpha_{3, 1}^{\phi}, \alpha_{3, 2}^{\phi})$.
	\end{definition}
	
	The above notion of SA has been widely studied and adopted in the literature. For example, a special case of SA with $\alpha_{3,1}^{\phi} = 0$ was adopted in \cite[Definition~2.1]{25XPS}, and some notions of optimal smoothing were also proposed and studied in \cite{25SG} based on the above definition of SA. The next proposition discusses how an SA inherits certain properties of the original function.
	
	\begin{proposition}\label{prop-eps-subdiff}
		Let $\phi:{\mathbb Y}\rightarrow \mathbb{R}$ be a convex function.
		Let $\{\phi_{\mu}\}_{\mu>0}$ be an SA of
		$\phi$ with parameters $(\alpha_{1}^{\phi},\, \alpha_{2}^{\phi},\, \alpha_{3,1}^{\phi}, \alpha_{3, 2}^{\phi})$
		and $\alpha_3^{\phi} \coloneqq \alpha_{3,1}^{\phi} +\alpha_{3, 2}^{\phi}> 0$.
		Then
		\begin{align}\label{eq-nabla-h-mu-incl}
			& \nabla \phi_{\mu}(y) \in \partial_{\alpha_3^{\phi} \mu} \phi (y)
			\quad \forall y \in \mathbb{Y}, \mu>0,
			\\ &  \phi_{\mu}^{\infty} = \phi^{\infty}\quad \forall \mu>0. \label{eq-hmu-hzn}
		\end{align}
		Moreover, if $\phi$ is $M_{\phi}$-Lipschitz continuous for some $M_\phi\ge 0$, then
		$\phi_{\mu}$ is also $M_{\phi}$-Lipschitz continuous for every $\mu > 0$.
	\end{proposition}
	\begin{proof}
		For every $\mu>0$ and $y\in\mathbb{Y}$, using the convexity of $\phi_{\mu}$, we have
		\begin{equation}\label{eq-subdiff}
			\phi_{\mu}(w) \ge \phi_{\mu}(y) + \langle \nabla \phi_{\mu}(y), w - y\rangle\quad \forall\, w \in \mathbb{Y}.
		\end{equation}
		From Definition~\ref{def-sm}(ii), it follows that
		$$\phi_{\mu}(w) \le \phi(w) + 	\alpha_{3,2}^{\phi} \mu
		\quad \text{and} \quad
		\phi(y) - \alpha_{3,1}^{\phi} \mu \le \phi_{\mu}(y),
		$$
		which, together with \eqref{eq-subdiff}, gives
		\begin{align*}
			\phi(w) & \ge \phi(y) + \langle \nabla \phi_{\mu}(y), w - y \rangle - (\alpha_{3, 1}^{\phi} + \alpha_{3, 2}^{\phi}) \mu
			\quad \forall\, w \in \mathbb{Y}.
		\end{align*}
		This shows $\nabla \phi_{\mu}(y) \in \partial_{\alpha_3^{\phi} \mu} \phi (y)$.
		
		Next, using Definition~\ref{def-sm}(ii) again, we have for every $\mu> 0$ and $y\in \mathbb{Y}$,
		\begin{equation*}
			\frac{\phi(ty)- \alpha_{3,1}^{\phi}\mu}{t}  \leq \frac{ \phi_{\mu}(ty)}{t} \leq \frac{\phi(ty) +  \alpha_{3,2}^{\phi} \mu}{t} \quad \forall t>0.
		\end{equation*}
		Passing to the limit as $t\to \infty$ in the above display and invoking \cite[Theorem~3.21]{98RW},
		we get \eqref{eq-hmu-hzn}.
		
		Finally, according to \cite[Theorem~2.4.2(iv)]{02Zalinescu},
		we have
		\begin{equation*}
			\partial_{\epsilon} \phi(y) \subseteq \mathrm{dom}\, \phi^* \subseteq {\rm cl}(\mathrm{Im}\, \partial \phi)
			\quad \forall \epsilon\ge 0,\, y\in \mathbb{Y},
		\end{equation*}
		where the second inclusion follows from \cite[Theorem~3.1.2]{02Zalinescu}, with ${\rm Im}\,\partial \phi\coloneqq \bigcup_{x\in {\rm dom}\,\phi}\partial\phi(x)$.
		Therefore, if $\phi$ is $M_{\phi}$-Lipschitz continuous,
		then we have
		\begin{equation*}
			\sup_{u \in \partial_{\epsilon} \phi(y)} \| u \| \le
			\sup_{u \in {\rm cl}(\mathrm{Im}\, \partial \phi)} \| u \| \le
			M_{\phi} \quad \forall y \in \mathbb{Y},\, \epsilon \ge 0.
		\end{equation*}
		The claimed Lipschitz continuity of $\phi_\mu$ now follows from the above display and \eqref{eq-nabla-h-mu-incl}.
	\end{proof}

We next recall a way of constructing SAs for globally Lipschitz convex functions based on the Moreau envelope. Such construction also leads to the important inequality \eqref{Moreau_inequality} for our subsequent analysis; this inequality has been instrumental in the analysis of various variable smoothing methods in the literature, see, e.g., \cite{18TDFC,20BB}.
\begin{proposition}[The Moreau proximal smoothing]\label{cor-mps}
    Let $\phi:\mathbb{Y}\to \mathbb{R}$ be a convex function that is $M_\phi$-Lipschitz continuous with $M_\phi > 0$.
	For every $\mu>0$, define
	\begin{equation*}
		\phi_{\mu}(y)  \coloneqq \inf_{w\in \mathbb{Y}}
		\left\{ \phi(w) + \tfrac{1}{2\mu}\|y - w\|^2\right\}
		\quad \forall y \in \mathbb{Y}.
	\end{equation*}
	Then, $\{\phi_{\mu}\}_{\mu>0}$ is an SA of $\phi$ with parameters $(\alpha_1^{\phi}, \alpha_2^{\phi}, \alpha_{3,1}^{\phi},\alpha_{3,2}^{\phi})=(0,1,\tfrac{1}{2}M_{\phi}^2, 0)$ and
    \begin{equation}\label{Moreau_inequality}
			\phi_{\mu_{1}}(y) \le \phi_{\mu_{0}}(y) + \tfrac{1}{2}M_{\phi}^2 (\mu_{0}-\mu_{1}) \quad \forall \mu_{0}\ge \mu_{1}>0, y\in \mathbb{Y}.
	\end{equation}
\end{proposition}
\begin{proof}
  The fact that $\{\phi_\mu\}_{\mu > 0}$ is an SA with the prescribed parameters follows from \cite[Lemma~4.2]{12BT} and \cite[Corollary~4.1]{12BT}. The inequality \eqref{Moreau_inequality} was proved in \cite[Lemma~2.6]{20BB}; see also \cite[Lemma~10]{18TDFC}.
\end{proof}

In our algorithmic development, we need to make use of SAs of $h$ in \eqref{opt-dc-hc} (under Assumption~\ref{ass-gen}) and the sublinear function $[\cdot]_+$ that satisfy the additional inequality \eqref{Moreau_inequality}; the existence of such SAs is guaranteed by Proposition~\ref{cor-mps}. For ease of reference, we discuss the details in the next remark.
\begin{remark}[{SAs for $h$ in \eqref{opt-dc-hc} and $[\cdot]_+$}]\label{ass-sm}
For the rest of this paper, under Assumption~\ref{ass-gen}, we let $\{h_{\mu}\}_{\mu>0}$ be an SA of $h$ in \eqref{opt-dc-hc} with parameters $(\alpha_1^{h}, \alpha_2^{h}, \alpha_{3,1}^{h},\alpha_{3,2}^{h})$ and $\alpha_{3}^{h} \coloneqq \alpha_{3, 1}^{h}+\alpha_{3, 2}^{h}> 0$ such that there exists a constant $\alpha_{4}^{h}\ge 0$ satisfying
		\begin{equation}\label{eq-gmu-upper-Lip-cont}
			h_{\mu_{1}}(y) \le h_{\mu_{0}}(y) + \alpha_{4}^{h} (\mu_{0}-\mu_{1}) \quad \forall \mu_{0}\ge \mu_{1}>0, y\in \mathbb{Y};
		\end{equation}
we also let $\{\varphi_{\mu}\}_{\mu>0}$ be an SA of the function $[\cdot]_{+}$ with parameters
		$(\alpha_1^{\varphi}, \alpha_2^{\varphi}, \alpha_{3, 1}^{\varphi}, \alpha_{3,2}^{\varphi})$ and $\alpha_{3}^{\varphi} \coloneqq \alpha_{3, 1}^{\varphi}+\alpha_{3, 2}^{\varphi}> 0$ such that there exists a constant $\alpha_{4}^{\varphi}\ge 0$ satisfying
		\begin{equation}\label{eq-varphi-upper-Lip-cont}
			\varphi_{\mu_{1}}(s) \le \varphi_{\mu_{0}}(s) + \alpha_{4}^{\varphi} (\mu_{0}-\mu_{1}) \quad \forall \mu_{0}\ge \mu_{1}>0, s\in \mathbb{R}.
		\end{equation}
The existence of such SAs follows immediately from Proposition~\ref{cor-mps}. In addition, for the SA $\{\varphi_{\mu}\}_{\mu>0}$, according to Proposition~\ref{prop-eps-subdiff}, we have for every $\mu> 0$,
    \begin{equation}\label{eq-eps-subd-ReLU}
        \varphi_{\mu}^{\prime}(s) \in \partial_{\alpha_3^{\varphi}\mu} [s]_{+}
        = \begin{cases}
            [\max\{0, 1 - \alpha_3^{\varphi}\mu/s\}, 1] & \mbox{if}\ s>0,
            \\ [0, \min\{1, -\alpha_3^{\varphi}\mu/s\}] & \mbox{if}\ s<0,
            \\ [0, 1] & \mbox{if}\ s = 0.
        \end{cases}
    \end{equation}
\end{remark}

\begin{remark}[Computability of SAs]
  We briefly comment on the computability of the SAs mentioned in Remark~\ref{ass-sm}, i.e., whether one can construct SAs whose objective values and gradients are efficiently computable. In this regard, if the proximal mapping of $\mu h$ can be efficiently computed for all $\mu > 0$, then according to Proposition~\ref{cor-mps}, the required SAs $\{h_\mu\}_{\mu > 0}$ and $\{\varphi_\mu\}_{\mu > 0}$ can be constructed respectively as Moreau envelopes of $\mu h$ and $\mu [\cdot]_+$ for $\mu > 0$ and therefore can be efficiently computed.
\end{remark}

The next lemma shows that the composition of the SA $\{\varphi_\mu\}_{\mu > 0}$ with the SA $\{h_\mu\}_{\mu > 0}$ from Remark~\ref{ass-sm} gives an SA of $h^+$.

\begin{lemma}[An SA for $h^+$]\label{lm-Hmu}
	Consider \eqref{opt-dc-hc}, and suppose
	that Assumption~\ref{ass-gen} holds. Let $\{\varphi_\mu\}_{\mu > 0}$ and $\{h_\mu\}_{\mu > 0}$ be the SA given in Remark~\ref{ass-sm}.
	Define
	\begin{equation}\label{eq-H-mu-def}
		H_{\mu} \coloneqq \varphi_{\mu} \circ h_{\mu} \quad \forall \mu>0.
	\end{equation}
	Then, $\{H_{\mu}\}_{\mu>0}$ is an SA of $h^{+}$ with parameters $(\alpha_1, \alpha_2, \alpha_{3, 1}, \alpha_{3,2})$, and satisfies
    \begin{equation*}
		H_{\mu_{1}}(y) \le H_{\mu_{0}}(y) + \alpha_{4} (\mu_{0}-\mu_{1}) \quad \forall \mu_{0}\ge \mu_{1}>0, y\in \mathbb{Y},
	\end{equation*}
    where
	\begin{align}\label{eq-def-alpha}
		& \alpha_{1} = \alpha_{1}^{h} + (M_{h})^2 \alpha_{1}^{\varphi},
		\quad \alpha_2 = \alpha_2^{h} +  (M_{h})^2 \alpha_2^{\varphi},
		\quad \alpha_{3, 1} = [\alpha_{3, 1}^{h}]_{+} + \alpha_{3, 1}^{\varphi},
        \\ & \alpha_{3, 2} = [\alpha_{3, 2}^{h}]_{+} + \alpha_{3, 2}^{\varphi},
        \quad \alpha_3 = [\alpha_{3, 1}^{h}]_{+}
        +  [\alpha_{3, 2}^{h}]_{+} + \alpha_{3}^{\varphi} >0
        \quad \mbox{and} \quad
        \alpha_4 = \alpha_{4}^{h}  + \alpha_4^{\varphi}. \label{eq-def-alpha4}
	\end{align}
\end{lemma}
\begin{proof}
	It holds that, for every $y\in \mathbb{Y},\, w \in \mathbb{Y}$ and $\mu>0$,
	\begin{align*}
		& \| \nabla H_{\mu}(y) - \nabla H_{\mu}(w) \|
		= \| \varphi_{\mu}^{\prime} ( h_{\mu}(y) ) \nabla h_{\mu}(y) - \varphi_{\mu}^{\prime} ( h_{\mu}(w) ) \nabla h_{\mu}(w) \|
		\\& \le | \varphi_{\mu}^{\prime} ( h_{\mu}(y) )|  \|  \nabla h_{\mu}(y) -  \nabla h_{\mu}(w)\|
		+  \|\nabla h_{\mu}(w)\| |\varphi_{\mu}^{\prime} ( h_{\mu}(y) ) - \varphi_{\mu}^{\prime} ( h_{\mu}(w) ) |
		\\& \le  \|  \nabla h_{\mu}(y) -  \nabla h_{\mu}(w)\| + M_{h} |\varphi_{\mu}^{\prime} ( h_{\mu}(y) ) - \varphi_{\mu}^{\prime} ( h_{\mu}(w) ) |
		\\& \le
		(\alpha_{1}^{h} + \alpha_2^{h} /\mu ) \| y - w \| + M_{h} (\alpha_{1}^{\varphi} + \alpha_2^{\varphi} /\mu )
		\| h_{\mu}(y) - h_{\mu}(w) \|
		\\& \le
 \left(\alpha_{1}^{h} + (M_{h})^2 \alpha_{1}^{\varphi}  + \left(\alpha_2^{h} +  (M_{h})^2 \alpha_2^{\varphi}\right)/\mu \right)
		   \| y - w \|,
	\end{align*}
	where the second inequality follows from the Lipschitz continuity of $\varphi_{\mu}$ and $h_{\mu}$ (see Proposition~\ref{prop-eps-subdiff}),
	the third inequality holds because of the Lipschitz continuity of $\nabla \varphi_\mu$ and $\nabla h_{\mu}$,
	and the last inequality holds thanks to the Lipschitz continuity of $h_{\mu}$.
	
	Next, using Definition~\ref{def-sm}(ii), we have, for every $\mu>0$ and $y\in \mathbb{Y}$,
	\begin{align*}
		\varphi_{\mu}(h_\mu(y)) & \le [h_\mu(y)]_{+} + \alpha_{3, 2}^{\varphi} \mu
		\le [h(y) + \alpha_{3, 2}^{h} \mu]_{+} + \alpha_{3, 2}^{\varphi} \mu
		\\ & \le  [h(y)]_{+} + ([\alpha_{3, 2}^{h}]_{+} + \alpha_{3, 2}^{\varphi}) \mu
	\end{align*}
	and
	\begin{align*}
		\varphi_{\mu}(h_\mu(y)) & \ge [h_\mu(y)]_{+} - \alpha_{3, 1}^{\varphi} \mu
		\ge [h(y) - \alpha_{3, 1}^{h} \mu]_{+} - \alpha_{3, 1}^{\varphi} \mu
		\\ & \ge  [h(y)]_{+} - ([\alpha_{3, 1}^{h}]_{+} + \alpha_{3, 1}^{\varphi}) \mu.
	\end{align*}
	
	Finally, for every $y\in \mathbb{Y}$ and $\mu_0 \ge \mu_1 >0$, using the convexity of $\varphi_{\mu_1}$, we have
	\begin{align*}
		& \varphi_{\mu_1}(h_{\mu_1}(y)) \le \varphi_{\mu_1} (h_{\mu_0}(y)) + \varphi_{\mu_1}^{\prime} ( h_{\mu_1}(y) ) ( h_{\mu_1}(y) - h_{\mu_0}(y) )
		\\ & \le \varphi_{\mu_1} (h_{\mu_0}(y)) + \alpha_{4}^{h} (\mu_{0}-\mu_{1})
		\le \varphi_{\mu_0} (h_{\mu_0}(y)) +  (\alpha_{4}^{h}  + \alpha_4^{\varphi}) (\mu_0 - \mu_1),
	\end{align*}
	where the second inequality holds because of \eqref{eq-gmu-upper-Lip-cont} and \eqref{eq-eps-subd-ReLU},
    and the last inequality holds thanks to \eqref{eq-varphi-upper-Lip-cont}.
\end{proof}

Before ending this section, we present a lemma concerning properties of the composition of $H_\mu$ with $c$, where $\{H_\mu\}_{\mu > 0}$ is the SA of $h^+$ in Lemma~\ref{lm-Hmu}.
\begin{lemma}\label{lm-g-mu}
	Consider \eqref{opt-dc-hc}, and suppose
	that Assumption~\ref{ass-gen} holds.
	Let $\{H_{\mu}\}_{\mu>0}$ be defined as in \eqref{eq-H-mu-def}, and let $\alpha_1, \alpha_2, \alpha_{3, 1}, \alpha_{3,2}$, $\alpha_3$ and $\alpha_4$ be given in \eqref{eq-def-alpha} and \eqref{eq-def-alpha4}.
	Define
	\begin{equation}\label{eq-g-mu-def}
		g_{\mu} \coloneqq  H_{\mu} \circ c \quad \forall \mu>0.
	\end{equation}
	Then the following statements hold.
	\begin{enumerate}[{\rm (i)}]
		\item We have
		\begin{align}
			& g^{+}(x) - \alpha_{3, 1} \mu \le g_{\mu}(x)\le g^{+}(x) + \alpha_{3,2} \mu\quad \forall\, x\in \mathbb{X},\,
			\mu>0, \label{eq-gmu-appr}\\
			& g_{\mu_1}(x) \le g_{\mu_{0}}(x) + \alpha_{4} (\mu_{0} - \mu_1)\quad \forall\,
			x\in \mathbb{X},\, \mu_{0} \ge \mu_1>0. \label{eq-gmu-Lips}
		\end{align}
		\item For every $\mu>0$,
		\begin{equation*}
			g_{\mu}(x) \ge g_{\mu}(z) + \langle \nabla g_{\mu}(z), x - z\rangle - \frac{L_{c}M_{h}}{2} \|x - z\|^2
			\quad \forall\, x,z\in \mathcal{C}.
		\end{equation*}
		\item For every $\mu>0$, $\nabla g_{\mu}$ is $L_{\mu}$-Lipschitz continuous on $\mathcal{C}$, where
		\begin{equation}\label{eq-L-mu}
			L_{\mu} \coloneqq M_{h}L_{c} + \alpha_{1} ({M}_{c})^2 + \alpha_2 ({M}_{c})^2/\mu.
		\end{equation}
         \item If $c$ is $(-\mathrm{hzn}\, h)$-convex,
        then $g$, $g^{+}$ and $g_{\mu}$ are convex for any $\mu>0$.
	\end{enumerate}
\end{lemma}
\begin{proof}
	Item (i) follows from Definition~\ref{def-sm}(ii), Lemma~\ref{lm-Hmu}, and the definition of $g_{\mu}$ in \eqref{eq-g-mu-def}.
	
	Clearly, $h^{+}$ is $M_{h}$-Lipschitz continuous, which together with Proposition~\ref{prop-eps-subdiff} and Lemma~\ref{lm-Hmu} implies that $H_{\mu}$ is $M_{h}$-Lipschitz continuous for every $\mu>0$.
	Combining this fact with the $L_c$-smoothness of $c$ on ${\cal C}$ (see Assumption~\ref{ass-gen}(ii)) and proceeding as in the proof of \cite[Proposition~2.6(iii)]{25XPS}, we see that item (ii) holds.
	
	
	Item (iii) can be established in a way similar to \cite[Lemma~3.1]{25XPS}, using the compactness of $\mathcal{C}$ together with Definition~\ref{def-sm}(i) and Proposition~\ref{prop-eps-subdiff}.

    We now prove item (iv). For any $\mu>0$, from \eqref{eq-hmu-hzn}, it follows that $h^{\infty} = h_{\mu}^{\infty}$,
    which implies that $\mathrm{hzn}\, h = \mathrm{hzn}\, h_{\mu}$.
    Using this fact, and invoking \cite[Theorem~1]{21BTN}, we obtain that $h\circ c$ and $h_{\mu}\circ c$ are convex on $\mathbb{X}$ for every $\mu>0$.
    Moreover, for every $\mu>0$, we see from \eqref{eq-eps-subd-ReLU} that $\varphi_{\mu}$ is nondecreasing on $\mathbb{R}$,
    which, together with the convexity of $\varphi_{\mu}$ and $h_{\mu}\circ c$, implies that $g_{\mu} = \varphi_{\mu} \circ h_{\mu}\circ c$ is convex on $\mathbb{X}$. Similarly, $(h\circ c)^+$ is convex because $[\cdot]_+$ is convex and nondecreasing and $h\circ c$ is convex.
\end{proof}

\section{Smoothing ESQM}\label{sec3}
\subsection{Algorithm}
We now present our smoothing extended sequential quadratic method (\emph{s}ESQM)
for solving \eqref{opt-dc-hc},
summarized in Algorithm~\ref{alg-sl-sm-esqm} below.
\begin{algorithm}[h]
	\caption{{\em s}ESQM: a smoothing ESQM for \eqref{opt-dc-hc} under Assumption~\ref{ass-gen}} \label{alg-sl-sm-esqm}
	\begin{algorithmic}[0]
		\State {\bf Require.} Given $x^{0}\! \in\! \mathcal{C}$,
		$\mu_{0}>0$, $\overline{L} \ge \underline{L}>0$,
		${c_1}>0$, $c_2>0$, a decreasing positive
		sequence $\{{\widehat \gamma}_{k}\}$ (i.e., $\widehat \gamma_{k} > \widehat\gamma_{k+1} > 0$ for all $k$), and the SAs $\{h_{\mu}\}_{\mu>0}$ and $\{\varphi_{\mu}\}_{\mu>0}$ in Remark~\ref{ass-sm}.
		Set $t_0 = 0$ and $k = 0$.
		\State {\bf Step 1.}
		Take $\xi^k\in\partial P_2(x^k)$
		and $L_{k,0} \in [\underline{L}, \overline{L}]$.
		Set  $\gamma_{k} = {\widehat \gamma}_{t_{k}}$ and $\widetilde{L} = L_{k,0}$.
		\State {\bf Step 2.}
		Compute
			\begin{align}
				\quad\quad\qquad	\widetilde{x}
				=\underset{x\in \mathbb{X}}{\mathrm{arg\, min}}  \Big\{\psi_{k, \widetilde{L}}(x)
				& \coloneqq \gamma_{k}
				(\langle \nabla f(x^{k}) - \xi^k,\, x\rangle + P_1(x)) \notag
				\\ & \ \ \ \,
				+  \ell_{k}(x)
				+ \tfrac{\widetilde{L}}{2\mu_{k}}\|x-x^{k}\|^2 + \delta_{\mathcal{C}}(x)
				\Big\},\label{opt-sub}
			\end{align}
		\quad \quad \quad \quad where
		\begin{equation}\label{eq-ell-k}
			\ell_{k}(x) \coloneqq g_{\mu_{k}}(x^{k}) +  \langle \nabla g_{\mu_{k}}(x^{k}), x - x^{k} \rangle;
		\end{equation}
		\quad \quad \quad \quad here $\{g_{\mu}\}_{\mu>0}$ is defined as in \eqref{eq-g-mu-def}.
		\State {\bf Step 3.} If
		\begin{equation}\label{eq-ls-desc}
				\qquad \ \ \
				\gamma_{k} \psi(\widetilde{x})  +  g_{\mu_{k}}(\widetilde{x})
				\le \gamma_{k} \psi(x^{k})  +   g_{\mu_{k}}(x^{k})
				-  \tfrac{{c_1}}{2\mu_{k}}\|\widetilde{x} - x^{k}\|^2,
		\end{equation}
		\quad \quad \quad \quad then go to {\bf Step 4}; otherwise, let $\widetilde{L} \leftarrow 2 \widetilde{L}$ and go to {\bf Step 2.}
		
		\State {\bf Step 4.}
		Let $x^{k+1} = \widetilde{x}$ and $L_{k} = \widetilde{L}$.
		If the conditions
		\begin{equation}\label{eq-gamma-cond}
			h_{\mu_{k}}(c(x^{k+1}) )
			 > 2 \alpha_{3}^{\varphi} \mu_k  \quad \text{and}\quad
			\|x^{k+1} - x^{k}\| \le c_2\mu_{k} \gamma_{k}
		\end{equation}
		\quad \quad \quad \quad hold,
		then let $t_{k+1} = t_{k} + 1$; otherwise let $t_{k+1} = t_{k}$.
		
		\State {\bf Step 5.} Choose $\mu_{k+1}\in (0, \mu_{k}]$. Update $k\leftarrow k+1$ and go to {\bf Step 1.}
	\end{algorithmic}
\end{algorithm}

	In essence, our method is based on the following auxiliary function:
	$$
	f(x) + P_1(x) - P_2 (x) +\delta_{\mathcal{C}}(x) + \gamma^{-1} H_{\mu}(c(x)),
	$$
	where $\gamma>0$ and $\mu>0$,
	and the SA $\{H_{\mu}\}_{\mu>0}$ is given in \eqref{eq-H-mu-def}.
	
	In each iteration, the smooth components
	$f$ and $g_{\mu_{k}}=H_{\mu_k}\circ c$ are linearized
	at the current iterate $x^{k}$.
	In addition, following the standard DC algorithm,
	the $P_2$ is linearized using its subgradient at $x^{k}$,
	while the convex term $P_1+\delta_{\mathcal{C}}$ is kept.
	A quadratic proximal regularization term is then added to ensure strong convexity of the subproblem objective. Consequently, the subproblem \eqref{opt-sub} corresponds to a computation of the proximal mapping of $\frac{\gamma_k\mu_k}{\widetilde L}P_1 + \delta_{\cal C}$.
	
	A linesearch is performed in {\bf Step~3} to enforce sufficient descent of a certain merit function.
	Moreover, we decrease the parameters $\{\gamma_k\}$ when both conditions in \eqref{eq-gamma-cond} hold.
	The first condition explicitly involves $\alpha_3^{\varphi}$, a parameter of the SA $\{\varphi_{\mu}\}_{\mu>0}$ in Remark~\ref{ass-sm}.
	The second condition plays a key role in guaranteeing a positive lower bound of $\{\gamma_k\}$ (see Theorem~\ref{th-gamma-bd} below) and is essential for the convergence analysis. We next argue the well-definedness of the algorithm; more precisely, we need to show that at each iteration, {\bf Step 3} is only invoked finitely many times.

\begin{lemma}[Well-definedness of Algorithm~\ref{alg-sl-sm-esqm}]\label{lm-well-def}
	Consider \eqref{opt-dc-hc}, and suppose that Assumption~\ref{ass-gen} holds.
	Let $x^{k}$ be generated at the beginning of the $k$th iteration of Algorithm~\ref{alg-sl-sm-esqm} for some $k\ge 0$.
	Then the following statements hold.
	Particularly, item {\rm (ii)} implies that Algorithm~\ref{alg-sl-sm-esqm} is well defined.
	\begin{enumerate}[\rm (i)]
		\item It holds that $x^{k} \in \mathcal{C}$.
		\item For every $\widetilde{L}>0$, problem \eqref{opt-sub} has a unique solution $\widetilde{x} \in \mathcal{C}$.
		Moreover,  the inner loop terminates finitely with some $L_k$ satisfying
		\begin{equation}\label{eq-ML}
			\underline{L} \le L_{k} \le M_{L} \coloneqq \max\{\overline{L}, {c_1} + \gamma_{0} \mu_{0} L_{f} + \mu_{0} L_{\mu_{0}}\},
		\end{equation}
		where $L_{\mu}$ is given in \eqref{eq-L-mu}. In particular, an $x^{k+1}$ can be generated at the end of the $k$th iteration.
		
		\item For every  $x \in \mathcal{C}$,
		\begin{align}
			\gamma_{k} \psi(x^{k+1})
			+  g_{\mu_k}(x^{k+1})
			& \le  \gamma_{k} (f(x^{k}) + \langle \nabla f(x^{k}) - \xi^k, x - x^{k} \rangle + P_1(x) - P_2(x^k) ) \notag
			\\ & \ \ + \ell_{k}(x)
			+ \frac{L_{k}}{2\mu_{k}} (\|x^{k} - x\|^2 - \|x^{k+1} - x\|^2) \notag
			\\ & \ \ + \frac{\mu_k \gamma_k L_f + \mu_k L_{\mu_{k}} - L_{k}}{2\mu_{k}} \| x^{k+1} - x^{k}\|^2, \label{eq-des-xk1}
		\end{align}
        where $L_{\mu}$ is given in \eqref{eq-L-mu}.
		Moreover, it holds that
		\begin{align}
			0 &\in \gamma_{k} (\nabla f(x^{k}) - \xi^k + \partial P_1(x^{k+1})) 	
			+  \nabla g_{\mu_{k}}(x^{k})
			 + \tfrac{L_{k}}{\mu_{k}}(x^{k+1} - x^{k})  	
			+ \mathcal{N}_{\mathcal{C}}(x^{k+1}). \label{eq-sub-kkt-k1}
		\end{align}

		\item It holds that
		\begin{equation}\label{eq-des-Q}
			Q_{k+1}
			\le Q_{k}
			-  \frac{{c_1}}{2\mu_{k}}\|x^{k+1} - x^{k}\|^2,
		\end{equation}
		where
		\begin{equation}\label{eq-Q-def}
			Q_{k} \coloneqq \gamma_{k}(\psi(x^{k}) - \bar{m}) +  g_{\mu_{k}}(x^{k}) + \alpha_{4}  \mu_{k}
		\end{equation}
		with $\bar{m}\coloneqq \inf\{\psi(x) : x \in \mathbb{X}\}>-\infty$ and $\alpha_4$ given in \eqref{eq-def-alpha4}.
	\end{enumerate}
\end{lemma}
\begin{proof}
	Clearly, $x^{k} \in \mathcal{C}$. Moreover, notice that for every $\widetilde{L}>0$, the objective of \eqref{opt-sub} is a proper closed strongly convex function. Consequently, this problem has a unique solution, and hence the $\widetilde{x}\in \mathcal{C}$ in \eqref{opt-sub} is well-defined.
	
	We now prove that the inner loop terminates finitely and \eqref{eq-ML} holds.
	Since $\psi_{k, \widetilde{L}}$ in \eqref{opt-sub} is ${\widetilde{L}}/({2\mu_{k}})$-strongly convex, we have that
	$$
	\psi_{k, \widetilde{L}}(\widetilde{x})
	\le \psi_{k, \widetilde{L}}(x) -\tfrac{\widetilde{L}}{2\mu_{k}} \|x - \widetilde{x}\|^2
	\quad \forall x\in \mathbb{X},
	$$
	which implies that, for every $x\in \mathcal{C}$,
		\begin{align}
			& \gamma_{k} (\langle \nabla f(x^{k}) - \xi^k, \widetilde{x} - x \rangle + P_1(\widetilde{x}))
			+  \ell_{k}(\widetilde{x}) \notag
			\\ & \le \gamma_k P_1(x) + \ell_{k}(x)
			+ \tfrac{\widetilde{L}}{2\mu_{k}}( \|x - x^{k}\|^2
			-  \|x - \widetilde{x}\|^2
			-  \| \widetilde{x} - x^{k}\|^2).
            \label{eq-str-conv}
		\end{align}
	Next, since $g_{\mu_{k}}$ is $L_{\mu_{k}}$-smooth on $\mathcal{C}$ (see Lemma~\ref{lm-g-mu}(iii)), and noting that both $\widetilde{x}\in \mathcal{C}$ and $ x^{k}\in \mathcal{C}$, we obtain
	\begin{equation}\label{eq-g-desc-lm}
		g_{\mu_{k}}(\widetilde{x}) \le \ell_{k}(\widetilde{x}) + \tfrac{L_{\mu_{k}}}{2} \|\widetilde{x} - x^{k}\|^2.
	\end{equation}
    Now, using the $L_f$-smoothness of $f$ on ${\cal C}$ and the convexity of $P_2$, we have
		\begin{align*}
			 \gamma_{k} \left(f(\widetilde{x}) + P_1(\widetilde{x}) - P_2(\widetilde{x})\right) +  g_{\mu_k}(\widetilde{x})
			&  \le
			\gamma_{k} (f(x^{k}) +  \langle \nabla f(x^{k}), \widetilde{x} - x^{k} \rangle + \tfrac{L_f}{2} \|\widetilde{x} - x^{k}\|^2)
			\\ & ~~~~ +  \gamma_{k} (P_1(\widetilde{x}) - P_2(x^{k}) - \langle\xi^k, \widetilde{x} - x^k\rangle) +  g_{\mu_k}(\widetilde{x})
			\\ & \le  \gamma_{k} (f(x^{k}) +  \langle \nabla f(x^{k}) - \xi^k, \widetilde{x} - x^{k} \rangle + P_1(\widetilde{x}) - P_2(x^k))
			\\ & ~~~~ + \ell_{k}(\widetilde{x}) + \frac{\gamma_k L_f + L_{\mu_k}}{2} \|\widetilde{x} - x^{k}\|^2,
		\end{align*}
	where the last inequality holds because of \eqref{eq-g-desc-lm}.
	
	From the last display, \eqref{eq-str-conv} and the fact that $\widetilde{x} \in \mathcal{C}$,
	we see that for every $x\in \mathcal{C}$,
		\begin{align}
			&\gamma_{k} \psi(\widetilde{x})
			+  g_{\mu_k}(\widetilde{x}) \notag
			\\& \le  \gamma_{k} (f(x^{k}) + \langle \nabla f(x^{k}) - \xi^k, x - x^{k} \rangle + P_1(x) - P_2(x^k) ) + \ell_{k}(x) \notag
			\\ & \ \ + \frac{\widetilde{L}}{2\mu_{k}} (\|x - x^{k}\|^2 - \|x - \widetilde{x}\|^2) + \frac{\mu_k \gamma_k L_f + \mu_k L_{\mu_{k}} - \widetilde{L}}{2\mu_{k}} \| \widetilde{x} - x^{k}\|^2. \label{eq-des-xk-tilde}
		\end{align}
	Upon setting $x = x^{k}$ in \eqref{eq-des-xk-tilde}, and using \eqref{eq-ell-k} and the fact that $x^{k}\in \mathcal{C}$,
	we have
	\begin{align}\label{eq-des}
		\gamma_{k} \psi(\widetilde{x}) +  g_{\mu_{k}}(\widetilde{x})
		\le \gamma_{k} \psi(x^{k}) +  g_{\mu_{k}}(x^{k}) - \frac{2\widetilde{L} - \gamma_{k} \mu_{k} L_f - \mu_{k} L_{\mu_{k}} }{2\mu_{k}} \|\widetilde{x}
		- x^{k}\|^2.
	\end{align}
	
	Now, using $\mu_{k} \le \mu_{0}$, $\mu_{k} L_{\mu_{k}} \le \mu_{0} L_{\mu_{0}}$ (see \eqref{eq-L-mu}) and $\gamma_{k}\le \gamma_{0}$, we have that for every $\widetilde{L} \ge M_{L}/2$,
	\begin{equation*}
		2\widetilde{L} - \gamma_{k} \mu_{k} L_f - \mu_{k} L_{\mu_{k}} \ge M_{L} - ( \gamma_{0} \mu_{0} L_{f} + \mu_{0} L_{\mu_{0}})
		\ge c_1,
	\end{equation*}
	where the last inequality follows from
	the definition of $M_{L}$ in \eqref{eq-ML}.
	Combining this with \eqref{eq-des},
	we obtain that
	condition \eqref{eq-ls-desc} holds for all $\widetilde{L}\ge M_{L}/2$.
	Consequently, at the $k$th iteration, the inner loop (i.e., {\bf Step 2} and {\bf Step 3}) of Algorithm~\ref{alg-sl-sm-esqm} is finitely terminated,
	i.e., there exists an $l_{k}\in \mathbb{N}$
	such that $L_{k} = 2^{l_{k}}L_{k,0}$. The bound \eqref{eq-ML} now follows immediately from the fact that \eqref{eq-ls-desc} holds for all $\widetilde{L}\ge M_{L}/2$ and the rule for updating $\widetilde L$ in {\bf Step 3} of Algorithm~\ref{alg-sl-sm-esqm}.
	
	Next, we consider item (iii).
	Replacing $(\widetilde{x}, \widetilde{L})$
	by $(x^{k+1}, L_k)$ in \eqref{eq-des-xk-tilde}, we obtain \eqref{eq-des-xk1}.
	In addition, \eqref{eq-sub-kkt-k1} follows immediately upon noting that $x^{k+1}$ is a solution of \eqref{opt-sub}
	with $\widetilde{L} = L_{k}$ and invoking the first-order optimality condition.
	
	Finally, we turn to item (iv).
	Using \eqref{eq-gmu-Lips} and $\mu_{k+1}\le \mu_k$, together with $0<\gamma_{k+1}\le \gamma_{k}$ and $\psi(x^{k+1}) - \bar{m} \ge 0$, we have
	\begin{align*}
		& \gamma_{k+1}(\psi(x^{k+1}) - \bar{m}) +
		g_{\mu_{k+1}}(x^{k+1}) +  \alpha_{4}\mu_{k+1} \notag
		\\ & \le \gamma_{k} (\psi(x^{k+1}) - \bar{m}) +  g_{\mu_{k}}(x^{k+1}) +  \alpha_{4}\mu_{k} \notag
		\\ & \le \gamma_{k}(\psi(x^{k}) - \bar{m}) +
		g_{\mu_{k}}(x^{k}) + \alpha_{4}  \mu_{k} - \tfrac{{c_1}}{2\mu_{k}}\|x^{k+1} - x^{k}\|^2,
	\end{align*}
	where the last inequality holds because of \eqref{eq-ls-desc}
	with $\widetilde{x} = x^{k+1}$.
	This completes the proof.
\end{proof}

\subsection{Complexity analysis}

In this subsection, we analyze the complexity of Algorithm~\ref{alg-sl-sm-esqm} for obtaining an $(\epsilon_1,\epsilon_2)$-KKT point of \eqref{opt-dc-hc} under the following assumption on the sequences $\{\mu_k\}$ and $\{\widehat\gamma_k\}$ in Algorithm~\ref{alg-sl-sm-esqm}. Intuitively, we require these sequences to vanish, but do not allow them to go to zero too fast. An explicit choice of $\{\mu_k\}$ can be found in \cite[Proposition~3.5]{25XPS}.
\begin{assumption}\label{ass-mu}
	\begin{enumerate}[{\rm (i)}]
		\item The sequence $\{\mu_{k}\}$ is positive and nonincreasing, and satisfies
		\begin{equation*}
			\textstyle \lim_{k\to \infty} \mu_{k} = 0
			\quad \mbox{and} \quad
			S_{K} \coloneqq \sum_{k = \lceil K/2 \rceil}^{K} \mu_{k}
			\rightarrow \infty \quad \mbox{as}\ K \to \infty.
		\end{equation*}
		\item $\lim_{k\to\infty}\widehat{\gamma}_{k} = 0$, and there exists $M_{\widehat\gamma} > 1$ such that
		\begin{equation}
			\label{eq-gamma-hat}
			0< \widehat{\gamma}_{k+1} < \widehat{\gamma}_{k} \le M_{\widehat\gamma} \widehat{\gamma}_{k+1}
			\quad \forall k\in \mathbb{N}_{0}.
		\end{equation}
	\end{enumerate}
\end{assumption}

We now present our complexity analysis for Algorithm~\ref{alg-sl-sm-esqm}.
We start with an auxiliary lemma and will need to make use of the following quantities, which are all finite under Assumption~\ref{ass-gen}:
\begin{equation}\label{eq-def-M-fPi}
	M_f \coloneqq \sup_{x\in \mathcal{C}} \|\nabla f(x)\|
	\quad \mbox{and} \quad
	M_{P_i} \coloneqq \sup\{\|\xi\| : \xi \in \partial P_i(x), x\in \mathcal{C}\} \quad \forall i=1,2.
\end{equation}
\begin{lemma}\label{lm-properties}
	Consider \eqref{opt-dc-hc}, and suppose that Assumption~\ref{ass-gen} holds.
	Let $\{x^{k}\}$ and $\{\gamma_{k}\}$ be generated by Algorithm~\ref{alg-sl-sm-esqm}.
    Let $\alpha_3$, $L_{\mu}$ and $M_{L}$ be given in \eqref{eq-def-alpha4}, \eqref{eq-L-mu} and \eqref{eq-ML}, respectively.
	Let $M_{f}$, $M_{P_1}$, and $M_{P_2}$ be defined in \eqref{eq-def-M-fPi}.
    Then the following statements hold for every $k\in \mathbb{N}_{0}$.
	\begin{enumerate}[{\rm (i)}]
        \item It holds that
        \begin{align}
        & \lambda_{k+1} \coloneqq  \gamma_{k}^{-1} \varphi_{\mu_k}^{\prime} \left( h_{\mu_k}(c(x^{k+1})) \right) \in [0,  \gamma_{k}^{-1}],	
        \label{eq-lambda-def}
        \\ & u^{k+1} \coloneqq \nabla h_{\mu_k}(c(x^{k+1})) \in \partial_{\alpha_{3}^{h}\mu_k} h (c(x^{k+1})), \label{eq-nabla-h-mu-part}
        \\ & \label{eq-patial-barB-k1}
		v^{k+1} \coloneqq \nabla H_{\mu_k}(c(x^{k+1})) \in
		\partial_{\alpha_{3}\mu_k} h^{+}(c(x^{k+1})),
		\\ & \delta_{k+1} \coloneqq
		\mathrm{dist}(0,
		\nabla f(x^{k+1})  - \partial P_{2}(x^{k})
		+ \partial P_1(x^{k+1}) \notag
		\\ & \qquad\qquad\qquad \ \
		+ \lambda_{k+1} Dc(x^{k+1})^* u^{k+1} + \mathcal{N}_{\mathcal{C}}(x^{k+1}) )
		\le M_1 \gamma_{k}^{-1}\omega_{k+1},
		\label{eq-delta-upbd-omega}
	\end{align}
    where
        \begin{align}
        M_1 \coloneqq M_L + \mu_{0} {\widehat{\gamma}_{0}} L_{f} + \mu_{0} L_{\mu_{0}}
		\quad \mbox{and} \quad
        \omega_{k+1} \coloneqq \mu_k^{-1} \|x^{k+1} - x^{k}\|
		\quad \forall k\in \mathbb{N}_0. \label{eq-omega}
        \end{align}

		\item If $h_{\mu_k}(c(x^{k+1})) > 2 \alpha_{3}^{\varphi} \mu_k$, then
		\begin{align}
			\label{eq-gB-res}
			& \mathrm{dist}(0,  Dc(x^{k+1})^* u^{k+1} + \mathcal{N}_{\mathcal{C}}(x^{k+1})) \notag
			\\ & \qquad \le
			2(M_{f} + M_{P_1} + M_{P_2}) \gamma_{k}
			+ 2(M_{L}
			+\mu_{0}L_{\mu_{0}}) \omega_{k+1}.
		\end{align}

		\item If $h_{\mu_k}(c(x^{k+1})) \le 2 \alpha_{3}^{\varphi} \mu_k$, then
		\begin{align}
			& g_{\mu_{k}}(x^{k+1}) \le (2 \alpha_{3}^{\varphi} + \alpha_{3,2}^{\varphi})\mu_k,
			\quad g^{+}(x^{k+1})
			 \le (3\alpha_{3}^{\varphi} + [\alpha_{3,1}^h]_{+}) \mu_k, \label{eq-feas-upbd}
			\\&   |\lambda_{k+1} g(x^{k+1})|
			\le \gamma_{k}^{-1}(3 \alpha_{3}^{\varphi} + [\alpha_{3,1}^{h}]_+ + [\alpha_{3,2}^{h}]_+) \mu_k. \label{eq-slack-upb}
		\end{align}
	\end{enumerate}
\end{lemma}
\begin{proof}
    Let $k\in \mathbb{N}_0$.
    First, we consider item (i).
    The relation \eqref{eq-lambda-def} follows from \eqref{eq-eps-subd-ReLU}, while the relations \eqref{eq-nabla-h-mu-part} and \eqref{eq-patial-barB-k1} follow directly from
    Remark~\ref{ass-sm}, Lemma~\ref{lm-Hmu}, together with \eqref{eq-nabla-h-mu-incl}.

    Now, we turn to establish \eqref{eq-delta-upbd-omega}.
	From \eqref{eq-sub-kkt-k1} and the fact $\xi^{k} \in \partial P_2(x^{k})$, we have
		\begin{align*}
			& -L_{k}\mu_k^{-1}(x^{k+1} - x^{k})
			+  \gamma_{k} (\nabla f(x^{k+1}) - \nabla f(x^{k})) 	
			+   \nabla g_{\mu_{k}}(x^{k+1}) -  \nabla g_{\mu_{k}}(x^{k})
			\\ & \in
			\gamma_{k} \left(\nabla f(x^{k+1}) - \partial P_2(x^{k}) + \partial P_1(x^{k+1})
			+ \gamma_{k}^{-1}\nabla g_{\mu_{k}}(x^{k+1})
			+ \mathcal{N}_{\mathcal{C}}(x^{k+1}) \right),
		\end{align*}
	which, together with $\nabla g_{\mu_{k}}(x^{k+1}) = \varphi_{\mu_k}^{\prime} \left( h_{\mu_k}(c(x^{k+1})) \right) Dc(x^{k+1})^* u^{k+1}$,
    implies that
		\begin{align*}
			\gamma_{k}\delta_{k+1} & \le L_{k}\mu_{k}^{-1} \| x^{k+1} - x^{k} \|
			+ \gamma_{k} \| \nabla f(x^{k+1}) - \nabla f(x^{k}) \|
			\\ & \ \ +    \| \nabla g_{\mu_{k}}(x^{k+1}) -  \nabla g_{\mu_{k}}(x^{k}) \|
			\\ & \le \left(L_{k} + \mu_{k} \gamma_{k} L_{f} + \mu_{k} L_{\mu_{k}} \right) \mu_{k}^{-1} \| x^{k+1} - x^{k} \|
			\\ & \le (M_L + \mu_{0} \widehat{\gamma}_{0} L_{f} + \mu_{0} L_{\mu_{0}} ) \mu_{k}^{-1} \| x^{k+1} - x^{k} \|,
		\end{align*}
	where the second inequality holds thanks to the $L_f$- and $L_{\mu_{k}}$-smoothness of $f$ and $g_{\mu_{k}}$ on $\mathcal{C}$ and the fact that $\{x^{k}\}\subseteq \mathcal{C}$,
	and the last inequality holds because of \eqref{eq-ML}, $0<\gamma_{k} \le \gamma_0 = \widehat{\gamma}_0$, $0<\mu_{k} \le \mu_{0}$
	and \eqref{eq-L-mu}.
	This proves \eqref{eq-delta-upbd-omega}.

	Next, we consider item (ii).
	Assume that $s_k \coloneqq h_{\mu_k}(c(x^{k+1})) > 2 \alpha_{3}^{\varphi} \mu_k$.
    It then follows from \eqref{eq-eps-subd-ReLU}, with $s_k>0$, that
	\begin{equation}\label{eq-lambda-low-bd}
		\varphi_{\mu_k}^{\prime} ( s_k ) \in \partial_{\alpha_3^{\varphi} \mu_k }[\cdot]_{+}( s_k )
		= [\max\{0, 1 - \tfrac{\alpha_3^{\varphi} \mu_k}{s_k} \}, 1]
		\subseteq [\tfrac{1}{2}, 1].
	\end{equation}
	Moreover, by \eqref{eq-sub-kkt-k1},
	there exists a $\zeta^{k+1} \in \partial P_1(x^{k+1})$
	such that
	\begin{align*}
			& -\gamma_{k} (\nabla f(x^{k}) - \xi^{k} + \zeta^{k+1})	
			- L_{k}\mu_{k}^{-1}(x^{k+1} - x^{k})  	
			+  \nabla g_{\mu_{k}}(x^{k+1}) - \nabla g_{\mu_{k}}(x^{k})
			\\ &\in \nabla g_{\mu_{k}}(x^{k+1}) + \mathcal{N}_{\mathcal{C}}(x^{k+1})
			= \varphi_{\mu_k}^{\prime} ( s_k ) Dc(x^{k+1})^* u^{k+1}
			+ \mathcal{N}_{\mathcal{C}}(x^{k+1}),
	\end{align*}
	which, together with \eqref{eq-lambda-low-bd}, implies
		\begin{align*}
			& \frac{1}{2}\mathrm{dist}(0,  Dc(x^{k+1})^* u^{k+1} + \mathcal{N}_{\mathcal{C}}(x^{k+1}))
			\\ & \le \gamma_{k} (\|\nabla f(x^{k}) \|
			+ \|\xi^{k}\| + \|\zeta^{k+1}\| )
			+ L_{k}\mu_{k}^{-1} \| x^{k+1} - x^{k}\|	
			+ \| \nabla g_{\mu_{k}}(x^{k+1}) - \nabla g_{\mu_{k}}(x^{k})\|
			\\& \le \gamma_{k} (\|\nabla f(x^{k}) \|
			+ \|\xi^{k}\| + \|\zeta^{k+1}\| )
			+ (L_{k}+\mu_{k}L_{\mu_{k}}) \omega_{k+1}	
			\\ & \le
			(M_{f} + M_{P_1} + M_{P_2}) \gamma_{k} + (M_{L}+\mu_{0}L_{\mu_{0}} ) \omega_{k+1},
		\end{align*}
	where the second inequality follows from Lemma~\ref{lm-g-mu}(iii)
	and the definition of $\omega_{k+1}$ in \eqref{eq-omega},
	and the last inequality holds thanks to
	\eqref{eq-ML}, \eqref{eq-def-M-fPi}
	and $\mu_{k} L_{\mu_{k}} \le \mu_{0} L_{\mu_{0}}$ (by \eqref{eq-L-mu} and $\mu_k \le \mu_0$).
	This proves \eqref{eq-gB-res}.
	
	Finally, we prove item (iii).
	Suppose that $s_k \coloneqq h_{\mu_k}(c(x^{k+1})) \le 2 \alpha_{3}^{\varphi} \mu_k$.
    By definition, we have
	\begin{align}\label{eq-gmu-xk1}
		g_{\mu_{k}}(x^{k+1}) = \varphi_{\mu_k}( s_k )
		\le [s_k]_{+} + \alpha_{3,2}^{\varphi} \mu_k
		\le (2 \alpha_{3}^{\varphi} + \alpha_{3,2}^{\varphi}) \mu_k,
	\end{align}
    where the first inequality follows from Definition~\ref{def-sm}(ii).
	Furthermore, by \eqref{eq-gmu-appr}, we have
		\begin{align}\label{eq-gB-up}
			g^{+}(x^{k+1})
			\le g_{\mu_{k}}(x^{k+1}) + \alpha_{3, 1} \mu_{k}
			\le (\alpha_{3, 1} + 2 \alpha_{3}^{\varphi} + \alpha_{3,2}^{\varphi}) \mu_k
            = ([\alpha_{3, 1}^{h}]_{+} + 3 \alpha_{3}^{\varphi} ) \mu_k,
		\end{align}
	where the second inequality follows from \eqref{eq-gmu-xk1}, and the equality holds thanks to \eqref{eq-def-alpha}.
	The above two displays give \eqref{eq-feas-upbd}.

    Now, we prove \eqref{eq-slack-upb}. From the definition of $\lambda_{k+1}$ in \eqref{eq-lambda-def}, it follows that
	\begin{equation}\label{eq-lam-g-updb}
		\lambda_{k+1} g(x^{k+1}) \!=\! \gamma_{k}^{-1} \varphi_{\mu_k}^{\prime} ( s_k)
		g(x^{k+1})\!\le\! \gamma_{k}^{-1} \varphi_{\mu_k}^{\prime} ( s_k)
		g^+(x^{k+1}) \!\le\!
		 \gamma_{k}^{-1}([\alpha_{3, 1}^{h}]_{+} + 3 \alpha_{3}^{\varphi}) \mu_k,
	\end{equation}
	where the two inequalities hold thanks to \eqref{eq-eps-subd-ReLU} and \eqref{eq-gB-up}.
    Moreover, it holds that
	\begin{align}\label{eq-lam-g-lbd}
		& \lambda_{k+1} g(x^{k+1})
		= \gamma_{k}^{-1} \varphi_{\mu_k}^{\prime} ( s_k) h(c(x^{k+1}))
		 \ge \gamma_{k}^{-1} \varphi_{\mu_k}^{\prime} ( s_k) \left( s_k - \alpha_{3, 2}^{h} \mu_k\right),
	\end{align}
    where the inequality holds in view of Remark~\ref{ass-sm} and Definition~\ref{def-sm}(ii).
	
	If $s_k < 0$, then using \eqref{eq-eps-subd-ReLU}, we obtain
	\begin{equation*}
		\varphi_{\mu_k}^{\prime} ( s_k ) \in \partial_{\alpha_3^{\varphi} \mu_k }[\cdot]_{+}( s_k )
		= \left[0, \min\{1, \tfrac{\alpha_3^{\varphi} \mu_k}{-s_k} \}\right],
	\end{equation*}
	which implies $-\varphi_{\mu_k}^{\prime} ( s_k )s_k \le \alpha_3^{\varphi} \mu_k.$
	Clearly, this inequality also holds when $s_k \ge 0$.
	Using these facts and \eqref{eq-lam-g-lbd}, we obtain that
	\[
    \lambda_{k+1} g(x^{k+1}) \ge
    -\gamma_k^{-1}\alpha_3^{\varphi}\mu_k- \gamma_{k}^{-1} \varphi_{\mu_k}^{\prime} ( s_k)\alpha_{3, 2}^{h} \mu_k
    \ge - \gamma_{k}^{-1}(\alpha_3^{\varphi} + [\alpha_{3, 2}^{h}]_+ ) \mu_k.
    \]
	Combining this with \eqref{eq-lam-g-updb},
    we conclude that \eqref{eq-slack-upb} holds.
	This completes the proof.
\end{proof}

We show in the next theorem that there is a positive lower bound on the $\{\gamma_k\}$ generated by Algorithm~\ref{alg-sl-sm-esqm}. This observation is consistent with the behavior of the corresponding sequence generated by ESQM and its variants in \cite{13Auslender,15Auslender,25ZPX}; see \cite[Theorem~3.1(b)]{13Auslender}, \cite[Corollary~1]{15Auslender} and \cite[Theorem~4.2]{25ZPX}.

\begin{theorem}[Lower bound for $\{\gamma_k\}$]\label{th-gamma-bd}
	Consider \eqref{opt-dc-hc}.
	Suppose that Assumptions~\ref{ass-gen}, \ref{ass-cq} and~\ref{ass-mu} hold.
	Let $\{\gamma_{k}\}$ and $\{t_k\}$ be generated by Algorithm~\ref{alg-sl-sm-esqm}.
	Let $\eta_{1}$ be given in Lemma~\ref{lm-eta-cq}.
	Let $L_{\mu}$ and $M_{L}$ be given in \eqref{eq-L-mu} and \eqref{eq-ML}, respectively.
    Let $M_{f}$, $M_{P_1}$, and $M_{P_2}$ be defined by \eqref{eq-def-M-fPi}.
	Then it holds that, for all $k \in \mathbb{N}_{0}$,
	\begin{equation}\label{eq-gamma-pos-low-bd}
		\gamma_{k} \ge \eta_{\gamma} \coloneqq \min\left\{ \frac{\eta_{1}}{2 M_{\widehat\gamma}(M_{f}
			+ M_{P_1} + M_{P_2} + c_2(M_{L}+\mu_{0}L_{\mu_{0}}))},
		{\widehat \gamma}_{k_0+1} \right\} > 0,
	\end{equation}
	where
	\begin{align}
		k_{0}  \coloneqq \mathrm{min}\left\{k \in \mathbb{N}_{0} :  \mu_{k} \le \eta_{1} /  \max\{
		\alpha_{3}^{h}, [\alpha_{3,2}^{h}]_+
		\}  \right\}\in [0,\infty). \label{eq-k0-def}
	\end{align}
	Moreover, the index set $\mathfrak{I} \coloneqq \{k \in \mathbb{N}_{0} : t_{k+1} \neq t_{k}\}$ is finite and satisfies
	\begin{equation}\label{eq-I-bd}
		|\mathfrak{I}| \le \max\{t\in \mathbb{N}_0 : \widehat{\gamma}_{t} \ge \eta_{\gamma}\} =: M_{\bar t} \in [k_0+1,\infty).
	\end{equation}
\end{theorem}
\begin{remark}\label{rem:mugamma}
If we choose the sequences $\{\mu_k\}$ and $\{\widehat{\gamma}_k\}$ in Algorithm~\ref{alg-sl-sm-esqm} a priori according to Assumption~\ref{ass-mu}, then the constants
$k_0$ in \eqref{eq-k0-def}, $\eta_{\gamma}$ in \eqref{eq-gamma-pos-low-bd}, and $M_{\bar t}$ in \eqref{eq-I-bd}
are independent of the sequences $\{x^{k}\}$ and $\{\gamma_k\}$ generated by Algorithm~\ref{alg-sl-sm-esqm}, and can be computed explicitly via their defining formulae prior to running the algorithm.
\end{remark}
\begin{proof}[Proof of Theorem~\ref{th-gamma-bd}]
	Since $\lim_{k\to\infty}\mu_k = 0$, we see that $k_{0}$ is well defined (i.e., finite). In addition, since $\lim_{k\to\infty}\widehat{\gamma}_k = 0$ and $\widehat \gamma_{k_0 + 1}\ge \eta_\gamma$, we see that $M_{\bar t}$ is well defined and $M_{\bar t}\ge k_0 + 1$.
	If $\mathfrak{I}= \emptyset$, then $\gamma_{k} \equiv {\widehat \gamma}_{0} \ge {\widehat \gamma}_{k_0+1}$ for all $k\in \mathbb{N}_{0}$ and hence \eqref{eq-gamma-pos-low-bd} and \eqref{eq-I-bd} hold.
	
	Now, suppose that $\mathfrak{I} \neq \emptyset$.
	We claim that $\gamma_{k+1} \ge \eta_{\gamma}$ for all $k\in \mathfrak{I}$.
	Granting this, we see that $\sup_{k\in \mathfrak{I}} t_{k+1} < \infty$ (otherwise, we have
	$\gamma_{k+1} = \widehat{\gamma}_{t_{k+1}} \to 0$ as $k\to \infty$, a contradiction).
	Combining this with {\bf Step 4} in Algorithm~\ref{alg-sl-sm-esqm},
	we obtain that $\mathfrak{I}$ is finite and
	$|\mathfrak{I}| = \sup_{k\in \mathfrak{I}} t_{k+1} = t_{\bar{k}+1}$,
	where $\bar{k} = \max \mathfrak{I}$.
	Since $\widehat{\gamma}_{t_{\bar{k}+1}} = \gamma_{{\bar k}+1} \ge \eta_{\gamma}$,
	it follows from the definition of $M_{\bar t}$ in \eqref{eq-I-bd}
	that $t_{\bar{k}+1} \le M_{\bar t}$,
    which implies \eqref{eq-I-bd}.
    Moreover, from the definition of $\mathfrak{I}$,
	we have $t_{k+1} \equiv t_{\bar{k}+1}$ for all $k\ge \bar{k}$,
	which implies that
	$\gamma_{k+1} = \widehat{\gamma}_{t_{k+1}}
	= \widehat{\gamma}_{t_{\bar{k} + 1}} \ge \eta_{\gamma}$
	for all $k\ge \bar{k}$.
	Since $\{\gamma_{k}\}$ is nonincreasing,
	we have that
	$\gamma_{k} \ge \gamma_{\bar{k}+1} \ge \eta_{\gamma}$ for all $k\le \bar{k}$,
    which implies \eqref{eq-gamma-pos-low-bd}.
	
	We now show that  $\gamma_{k+1} \ge \eta_{\gamma}$ for all $k\in \mathfrak{I}$. To this end, let $k$ be any integer in $\mathfrak{I}$.
	If $k\le k_{0}$, then we have
	\begin{equation}\label{eq-gamma-k0}
		\gamma_{k+1} = {\widehat \gamma}_{t_{k+1}}
		\ge {\widehat \gamma}_{t_{k_0+1}}
		\ge {\widehat \gamma}_{k_0+1},
	\end{equation}
    where the last inequality holds thanks to $t_{k_0+1}\leq k_0+1$ and the fact that $\{{\widehat \gamma}_k\}$ is decreasing.
	
	Now, we suppose that $k>k_{0}$.	
	From {\bf Step 4} in Algorithm~\ref{alg-sl-sm-esqm}
	and the fact $k\in \mathfrak{I}$, it follows that
	\eqref{eq-gamma-cond} holds, which, together with the definition of $\omega_{k+1}$ in \eqref{eq-omega}, implies that
	\begin{equation}\label{eq-omega-upbd}
		h_{\mu_{k}}(c(x^{k+1}) )
		> 2 \alpha_{3}^{\varphi} \mu_k
		\quad \mbox{and} \quad
		\omega_{k+1} = \mu_{k}^{-1}\|x^{k+1} - x^{k}\| \le c_2 \gamma_{k}.
	\end{equation}

    We have from \eqref{eq-nabla-h-mu-part} that
	\begin{equation}\label{eq-v-eta2}
		u^{k+1} = \nabla h_{\mu_k}(c(x^{k+1}))\in \partial_{\alpha_{3}^{h}\mu_k} h(c(x^{k+1}))
		\subseteq \partial_{\eta_{1}} h(c(x^{k+1})),
	\end{equation}
	where the last inclusion holds because of $\alpha_{3}^{h} \mu_k \le \alpha_{3}^{h} \mu_{k_0} \le \eta_{1}$ in view of the definition of $k_0$ in \eqref{eq-k0-def}.
	In addition, we have
	\begin{align*}
		g(x^{k+1}) & = h(c(x^{k+1}))  \ge h_{\mu_k}(c(x^{k+1})) - \alpha_{3, 2}^{h} \mu_k
        \ge  - [\alpha_{3, 2}^{h}]_+ \mu_{k}
		 \ge  - [\alpha_{3, 2}^{h}]_+ \mu_{k_0} \ge - \eta_{1},
	\end{align*}
	where the first inequality follows from Definition~\ref{def-sm}(ii), the second inequality holds because of the first relation in \eqref{eq-omega-upbd}, the third inequality holds because $\mu_k \le \mu_{k_0}$,  and
	the last inequality follows from the definition of $k_0$ in \eqref{eq-k0-def}.
	
	Combining the last display with \eqref{eq-v-eta2}, and invoking Lemma~\ref{lm-eta-cq},
	we conclude that
	\begin{align}
		\eta_1 & \le \mathrm{dist}(0,  Dc(x^{k+1})^* u^{k+1} + \mathcal{N}_{\mathcal{C}}(x^{k+1}))\notag
		\\ & \overset{\mathrm{(a)}}{\le} 2(M_{f} + M_{P_1} + M_{P_2}) \gamma_{k} + 2(M_{L}+\mu_{0}L_{\mu_{0}}) \omega_{k+1}\notag
		\\ & \overset{\mathrm{(b)}}{\le}
		2(M_{f} + M_{P_1} + M_{P_2} + c_2(M_{L}+\mu_{0}L_{\mu_{0}})) \gamma_k\notag
		\\ & \overset{\mathrm{(c)}}{\le}
		2 M_{\widehat\gamma} (M_{f} + M_{P_1} + M_{P_2}
		+ c_2(M_{L}+\mu_{0}L_{\mu_{0}})) \gamma_{k+1},\label{eq-gamma-lb->k0}
	\end{align}
	where (a) follows from \eqref{eq-gB-res},
	(b) holds thanks to
	the second relation in \eqref{eq-omega-upbd},
	and (c) holds by the fact that
	$$
	\gamma_{k} = \widehat{\gamma}_{t_k}
	\le M_{\widehat\gamma} \widehat{\gamma}_{t_k + 1}
	\le M_{\widehat\gamma} \widehat{\gamma}_{t_{k + 1}}
	=  M_{\widehat\gamma} \gamma_{k+1};
	$$
	here the first inequality is a consequence of \eqref{eq-gamma-hat}.
	
	Finally, combining \eqref{eq-gamma-k0} and \eqref{eq-gamma-lb->k0} with
	the definition of $\eta_{\gamma}$ in \eqref{eq-gamma-pos-low-bd}, we conclude that
	$\gamma_{k+1} \ge \eta_{\gamma}$ for all $k\in \mathfrak{I}$.
	This completes the proof.
\end{proof}

We next present our complexity bound for the iterates generated by Algorithm~\ref{alg-sl-sm-esqm}.
Note that the guarantee we establish is only valid when $K$ exceeds some threshold involving (upper bounds of) $M_{\bar t}$ in \eqref{eq-I-bd} and $Q_0$ in \eqref{eq-Q-def}. Notice that (an upper bound of) $M_{\bar t}$ can be computed based on problem data and algorithmic parameters explicitly {\em prior to} running the algorithm if $\{\mu_k\}$ is chosen a priori as discussed in Remark~\ref{rem:mugamma}, while (an upper bound of) $Q_0$ can be computed based on the initial point and an estimate of a lower bound for $\inf_{x\in \mathbb{X}}\psi(x)$, which are also obtainable {\em without} running the algorithm.

\begin{theorem}[Complexity bound in nonconvex setting]\label{th-complexity-S-mu}
	Consider \eqref{opt-dc-hc}.
	Suppose that Assumptions~\ref{ass-gen}, \ref{ass-cq} and~\ref{ass-mu} hold.
	Let $\{x^{k}\}$ and $\{\gamma_k\}$ be generated by Algorithm~\ref{alg-sl-sm-esqm}.
	Let $M_1$, $\eta_{\gamma}$ and $M_{\bar t}$ be given in \eqref{eq-omega}, \eqref{eq-gamma-pos-low-bd} and \eqref{eq-I-bd}, respectively.
    Then
    \begin{equation}\label{definition_M2}
    M_2 \coloneqq 1 + 2 (M_{\bar t} + 1) (Q_{0} + [\alpha_{3,1}]_+\mu_0)/c_1 > 0,
    \end{equation}
    where $Q_0$ is defined in \eqref{eq-Q-def} and $\alpha_{3,1}$ is given in \eqref{eq-def-alpha}.
	Moreover, for any $K\in \mathbb{N}_0$ satisfying
	\begin{equation}\label{eq-K-cond}
		K\ge 2 M_{\bar t}
		\quad \mbox{and} \quad
		S_{K}^{-1} \le
		(\eta_{\gamma}c_2)^{2}/ M_2,
	\end{equation}
	there exists an integer $\hat{k} \in [\lceil K/2 \rceil - M_{\bar t}, K]$ such that
	\begin{align}
		& x^{\hat{k}+1} \in \mathcal{C}, \quad \lambda_{{\hat k}+1} \in  [0, \eta_{\gamma}^{-1}],\quad
        u^{\hat{k}+1} \in
		\partial_{\alpha_3^{h}
			\mu_{\lceil K/2\rceil - M_{\bar t}}} h(c(x^{\hat{k}+1})),
		\label{eq-compl-x-v}
		\\ & \|x^{\hat{k}+1} - x^{\hat{k}}\| \le \sqrt{ M_2}\mu_{\lceil K/2\rceil - M_{\bar t}} S_{K}^{-\frac{1}{2}},
		\quad \delta_{\hat{k}+1} \le \eta_{\gamma}^{-1}M_1 \sqrt{ M_2}S_{K}^{-\frac{1}{2}},
		\label{eq-compl-delta}
		\\ &
		g^{+}(x^{\hat{k} + 1})
		\le (3 \alpha_{3}^{\varphi} + [\alpha_{3, 1}^{h}]_{+}) \mu_{\lceil K/2\rceil - M_{\bar t}}, \quad
		 |\lambda_{\hat{k} +1}g(x^{\hat{k} + 1})|
		\le M_3 \mu_{\lceil K/2\rceil - M_{\bar t}},
		\label{eq-compl-slack}
	\end{align}
	where $\{\lambda_{k+1}\}$, $\{u^{k+1}\}$ and $\{\delta_{k+1}\}$ are given respectively in \eqref{eq-lambda-def},
    \eqref{eq-nabla-h-mu-part} and \eqref{eq-delta-upbd-omega}, and
    \begin{equation}\label{eq-def-M3}
        M_3 \coloneqq(3 \alpha_{3}^{\varphi} + [\alpha_{3, 1}^{h}]_{+} + [\alpha_{3,2}^{h}]_+)/\eta_{\gamma}.
    \end{equation}
\end{theorem}
\begin{proof}
    Notice that for all $k\in \mathbb{N}_0$, we have
    \begin{align}
    Q_0\overset{\rm (a)}\ge Q_k &= \gamma_{k}(\psi(x^{k}) - \bar{m}) +  g_{\mu_{k}}(x^{k}) + \alpha_{4}  \mu_{k} \overset{\rm (b)}\ge g_{\mu_{k}}(x^{k}) + \alpha_{4}  \mu_k\notag\\
    &\overset{\rm (c)}\ge g^+(x^k) + (\alpha_4 - \alpha_{3,1}) \mu_k
    \ge -[\alpha_{3,1}]_+\mu_k
    \overset{\rm (d)}\ge -[\alpha_{3,1}]_+\mu_0,\label{Qk_relation}
    \end{align}
    where (a) follows from \eqref{eq-des-Q}, (b) follows from the definition of $\bar m$ and the fact that $\{x^k\}\subset {\cal C}$, (c) follows from \eqref{eq-gmu-appr}, and (d) holds because $\mu_k\le \mu_0$. Hence, $M_2$ defined in \eqref{definition_M2} is positive.

	We fix any $K\in \mathbb{N}_0$ satisfying \eqref{eq-K-cond} from now on. From Lemma~\ref{lm-well-def}(i), we have $x^{k} \in \mathcal{C}$ for all $k\in \mathbb{N}_{0}$.
	Moreover, from \eqref{eq-gamma-pos-low-bd} and \eqref{eq-lambda-def}, it follows that $\lambda_{k+1} \in [0, \eta_{\gamma}^{-1}]$ for all $k\in \mathbb{N}_{0}$.
	
	Now, let $\{Q_{k}\}$ and $\{\omega_{k+1}\}$ be defined as in \eqref{eq-Q-def} and \eqref{eq-omega}, respectively.
	Since $\{\mu_k\}$ is nonincreasing, we have, for each $k \ge M_{\bar t}$,
	\begin{equation*}
		\mu_{k}  \sum_{j=k-M_{\bar t}}^{k} \omega_{j+1}^2
		\le \sum_{j=k-M_{\bar t}}^{k} \mu_j \omega_{j+1}^2
		\le \frac{2}{c_1}\sum_{j=k-M_{\bar t}}^{k} (Q_{j} - Q_{j+1})
		= \frac{2}{c_1}(Q_{k-M_{\bar t}} - Q_{k+1}),
	\end{equation*}
	where the second inequality is deduced from \eqref{eq-des-Q}.	
	Since $K\ge 2M_{\bar t}$, we can sum the last display for $k=\lceil K/2 \rceil, \ldots, K$ and rearrange terms to obtain		
		\begin{align}\label{eq-sum-res-dx}
			\min_{\lceil K/2\rceil \le k \le K}  \sum_{j=k-M_{\bar t}}^{k} \omega_{j+1}^2&
            \le \frac{2 \sum_{k=\lceil K/2\rceil}^K (Q_{k-M_{\bar t}} - Q_{k+1})}
			{c_1 S_{K}}\notag\\
            &= \frac{2 \sum_{j=0}^{M_{\bar t}}
				(Q_{\lceil K/2 \rceil - M_{\bar t} + j} - Q_{K - M_{\bar t} + j+1})}
			{c_1 S_{K}}
			\le \frac{ M_2}
			{S_{K}},
		\end{align}
	where the last inequality follows from \eqref{Qk_relation} and the definition of $M_2$ in \eqref{definition_M2}.

	Consequently, we see from \eqref{eq-sum-res-dx} that there exists an integer
	$\tilde{k}\in [\lceil K/2\rceil, K]$ such that
		\begin{align*}
			\textstyle \omega_{l+1}^2 \le \sum_{j=\tilde{k}-M_{\bar t}}^{\tilde{k}} \omega_{j+1}^2
			\le  M_2 / S_{K}
			\quad \forall l \in \{ \tilde{k} - M_{\bar t}, \tilde{k} - M_{\bar t}+1, \ldots, \tilde{k}\}.
		\end{align*}
	
	Now, let $\mathfrak{I} = \{k\in \mathbb{N}_{0} : t_{k+1} \neq t_{k}\}$.
	From \eqref{eq-I-bd}, we have $|\mathfrak{I}| \le M_{\bar t}$.
	Thus, there exists a $\hat{k} \in \{\tilde{k}-M_{\bar t}, \ldots, \tilde{k}\}\setminus \mathfrak{I} \neq \emptyset$.
	Combining this with the last display and the nonincreasing property of $\{\mu_k\}$,
	we have
	\begin{equation}
		\label{eq-omega-mu-khat}
		\omega_{\hat{k} +1} \le \sqrt{ M_2}S_{K}^{-\frac{1}{2}}
		\quad \mbox{and} \quad
		\mu_{\hat{k}} \le \mu_{\lceil K/2\rceil - M_{\bar t}}.
	\end{equation}
	From \eqref{eq-delta-upbd-omega}, we deduce that
	\begin{align*}
		& \delta_{\hat{k}+1}
		\le M_1 \omega_{\hat{k} + 1} \gamma_{\hat{k}}^{-1}
		\le \eta_{\gamma}^{-1}M_1 \sqrt{ M_2}S_{K}^{-\frac{1}{2}},		
	\end{align*}
	where the last inequality holds thanks to \eqref{eq-gamma-pos-low-bd} and \eqref{eq-omega-mu-khat}. This proves the second relation in \eqref{eq-compl-delta}.
	Furthermore, using \eqref{eq-nabla-h-mu-part} and the second relation in \eqref{eq-omega-mu-khat}, we obtain the third relation in \eqref{eq-compl-x-v}.
	
	Moreover, from the definition of $\{\omega_{k+1}\}$ in \eqref{eq-omega},
	we have
	$$
	\|x^{\hat{k} + 1} - x^{\hat k}\| = \mu_{\hat{k}} \omega_{\hat{k}+1}
	\le \sqrt{ M_2}\mu_{\lceil K/2\rceil - M_{\bar t}} S_{K}^{-\frac{1}{2}},
	$$
	where the last inequality holds because of \eqref{eq-omega-mu-khat}. This proves the first relation in \eqref{eq-compl-delta}.
	
	Finally, using \eqref{eq-omega-mu-khat} again,
	we have
	\begin{equation*}
		\|x^{\hat{k} + 1} - x^{\hat k}\| = \mu_{\hat{k}} \omega_{\hat{k}+1}
		\le \mu_{\hat{k}} \sqrt{ M_2}S_{K}^{-\frac{1}{2}}
		\le c_2 \eta_{\gamma} \mu_{\hat{k}} \le
		c_2 \gamma_{\hat{k}} \mu_{\hat{k}},
	\end{equation*}
	where the second inequality holds thanks to \eqref{eq-K-cond}, and the
	last inequality holds because of \eqref{eq-gamma-pos-low-bd}.
	Since $\hat{k} \notin \mathfrak{I}$, from
	{\bf Step 4} in Algorithm~\ref{alg-sl-sm-esqm}
	and the last display,
	we see that $h_{\mu_{\hat k}}(c(x^{\hat k+1}) )
	\le 2 \alpha_{3}^{\varphi} \mu_{\hat k}$,
	which, together with Lemma~\ref{lm-properties}(iii), \eqref{eq-gamma-pos-low-bd} and
	\eqref{eq-omega-mu-khat}, proves \eqref{eq-compl-slack}.
    This completes the proof.
\end{proof}

\begin{remark}[Limit points of $\{x^k\}$ from Algorithm~\ref{alg-sl-sm-esqm}]\label{rm-stationary-pen}
    Consider \eqref{opt-dc-hc}.
	Suppose that Assumptions~\ref{ass-gen}, \ref{ass-cq} and~\ref{ass-mu} hold.
	Let $\{x^{k}\}$ and $\{\gamma_{k}\}$ be generated by Algorithm~\ref{alg-sl-sm-esqm}.
    From \eqref{eq-gamma-pos-low-bd} and the nonincreasing properties of $\{\gamma_k\}$, it follows that
    $\bar \gamma \coloneqq \lim_{k\to\infty} \gamma_k \in [\eta_{\gamma}, \infty).$
    We claim that there exists an accumulation point $\bar{x}$ of $\{x^{k}\}$ being a stationary point of \eqref{opt-dc-hc}, and moreover, it holds that
    \begin{equation*}
         0 \in \nabla f(\bar x) + \partial P_1(\bar x) - \partial P_2(\bar x)
         + \bar \gamma^{-1} Dc(\bar x)^* \bar v + \mathcal{N}_{\mathcal{C}}(\bar x)
    \end{equation*}
    for some $\bar v \in \partial h^{+} (c(\bar x))$.

    Indeed, we have that $\lambda_{k+1} Dc(x^{k+1})^* u^{k+1} = \gamma_{k}^{-1} Dc(x^{k+1})^* v^{k+1}$ for all $k \in \mathbb{N}_{0}$,
    where $\lambda_{k+1}$, $u^{k+1}$ and $v^{k+1}$ are given in
    \eqref{eq-lambda-def}, \eqref{eq-nabla-h-mu-part} and \eqref{eq-patial-barB-k1}, respectively.
    Combining this fact with \eqref{eq-patial-barB-k1} and the definition of
    $\{\delta_{k+1}\}$ in \eqref{eq-delta-upbd-omega},
    and invoking Theorem~\ref{th-complexity-S-mu}, the desired conclusion follows immediately.
\end{remark}

\begin{corollary}[Complexity for obtaining $(\epsilon_1, \epsilon_2)$-KKT points]\label{corol-compl-nconv}
	Consider \eqref{opt-dc-hc}.
	Suppose that Assumptions~\ref{ass-gen}, \ref{ass-cq} and~\ref{ass-mu} hold.
	Let $\{x^{k}\}$ be generated by Algorithm~\ref{alg-sl-sm-esqm}.
    Let $M_1$, $\eta_{\gamma}$, $M_{\bar t}$, $M_2$ and $M_3$ be given in \eqref{eq-omega},
    \eqref{eq-gamma-pos-low-bd}, \eqref{eq-I-bd}, \eqref{definition_M2} and \eqref{eq-def-M3}, respectively.
	Let $\epsilon_1>0$, $\epsilon_2>0$ and
	integer $K\ge 2 M_{\bar t}$ satisfy
	\begin{align}\label{eq-SK-cond}
		& S_{K}^{-1} \le M_2^{-1} \eta_{\gamma}^2  \min\left\{ c_2^{2}, M_1^{-2} \epsilon^2_{1}\right\},
    \\ & \mu_{\lceil K/2\rceil - M_{\bar t}}
    \le \min\left\{ \frac{1}{\alpha_3^{h}}, \frac{1}{M_3}, \frac{M_1}{\eta_{\gamma}},
    \frac{1}{3\alpha_{3}^{\varphi} + [\alpha_{3, 1}^{h}]_{+}}
    \right\} \epsilon_2 \eqcolon M_4 \epsilon_2.\label{eq-muK-cond}
	\end{align}
	Then there exists an integer $\hat{k} \in [\lceil K/2 \rceil - M_{\bar t}, K]$ such that
	$x^{\hat{k}+1}$ is an $(\epsilon_1, \epsilon_2)$-KKT point of \eqref{opt-dc-hc}.
\end{corollary}
\begin{proof}
Note that \eqref{eq-SK-cond} implies the second relation in \eqref{eq-K-cond}. Hence, we can invoke
	Theorem~\ref{th-complexity-S-mu} to conclude that
	there exists an integer $\hat{k} \in [\lceil K/2 \rceil - M_{\bar t}, K]$ such that
	\eqref{eq-compl-x-v}, \eqref{eq-compl-delta} and \eqref{eq-compl-slack} hold.
	Applying this result together with \eqref{eq-SK-cond}  and \eqref{eq-muK-cond},
	one readily sees that
    	\begin{align*}
		 & \|x^{\hat{k}+1} - x^{\hat{k}}\| \le \sqrt{ M_2}\mu_{\lceil K/2\rceil - M_{\bar t}} S_{K}^{-\frac{1}{2}}
        \le \eta_{\gamma}M_1^{-1} \epsilon_1 \mu_{\lceil K/2\rceil - M_{\bar t}} \le \epsilon_1 \epsilon_2,
		\\	& \delta_{\hat{k}+1} \le \eta_{\gamma}^{-1}M_1 \sqrt{ M_2}S_{K}^{-\frac{1}{2}} \le \epsilon_1,\quad
        u^{\hat{k}+1} \in
		\partial_{\epsilon_2} h(c(x^{\hat{k}+1})),
		\\ &
		g^{+}(x^{\hat{k} + 1})
		\le (3 \alpha_{3}^{\varphi} + [\alpha_{3, 1}^{h}]_{+}) \mu_{\lceil K/2\rceil - M_{\bar t}}  \le \epsilon_2,
		\\ & |\lambda_{\hat{k} +1}g(x^{\hat{k} + 1})|
		\le M_3 \mu_{\lceil K/2\rceil - M_{\bar t}} \le \epsilon_2.
	\end{align*}
	This completes the proof.
\end{proof}

\begin{remark}[Complexity under an explicit choice of $\{\mu_k\}$]
	Consider \eqref{opt-dc-hc}.
	Suppose that Assumptions~\ref{ass-gen}, \ref{ass-cq} and~\ref{ass-mu}(ii) hold.
	Let $\mu_0 > 0$ and $\mu_k = \mu_0 (1+k)^{-r}$
	for all $k\in \mathbb{N}$ with some constant $r\in (0, 1)$.
	Let $\{x^{k}\}$ be generated by Algorithm~\ref{alg-sl-sm-esqm}.
	Let $M_1$, $\eta_{\gamma}$, $M_{\bar t}$, $M_2$, $M_4$ be given in \eqref{eq-omega},
    \eqref{eq-gamma-pos-low-bd}, \eqref{eq-I-bd}, \eqref{definition_M2} and \eqref{eq-muK-cond}, respectively.
	Let $\epsilon_1 \in (0, 1)$, $\epsilon_2\in(0, 1)$ and
	$K\in \mathbb{N}_0$ satisfy
	\begin{align*}
		& K\ge \max\left\{
		\left(\frac{\mu_{0}}{2^{2r+1}}\right)^{-\frac{1}{1-r}}
		\min\left\{ \frac{\eta_{\gamma}^2 \epsilon^2_{1}}{M_1^2 M_2}, \frac{\eta_{\gamma}^2 c_2^{2}}{M_2}
		\right\}^{-\frac{1}{1-r}},\, 2M_{\bar t}+
		2 M_4^{-\frac{1}{r}}
		\left(\frac{\epsilon_2}{\mu_0}\right)^{-\frac{1}{r}} \right\}.
	\end{align*}
    It follows directly from \cite[Proposition~3.5]{25XPS} and Corollary~\ref{corol-compl-nconv} that there exists an integer $\hat{k} \in [\lceil K/2 \rceil - M_{\bar t}, K]$ such that
	$x^{\hat{k}+1}$ is an $(\epsilon_1, \epsilon_2)$-KKT point of \eqref{opt-dc-hc}.
    In other words, within $O(\max\{\epsilon_1^{-\frac{2}{1-r}}, \epsilon_2^{-\frac{1}{r}}\})$ iterations,
	Algorithm~\ref{alg-sl-sm-esqm} obtains an $(\epsilon_1, \epsilon_2)$-KKT point of \eqref{opt-dc-hc}.

    In particular, by setting $r=\frac{1}{3}$,
		Algorithm~\ref{alg-sl-sm-esqm} gets
		an $\epsilon_1$-KKT point of \eqref{opt-dc-hc}
		within $O(\epsilon^{-3}_1)$ iterations.
		This result matches the complexity of variable smoothing methods for unconstrained weakly convex composite minimization problems in \cite{21BW}.
		We refer the reader to \cite{11CGT,14CGT,22LMX,23KMM} and the references therein for related work on the complexity analysis of first-order methods that rely on subproblem solvers for constrained minimization problems.
\end{remark}

\section{Convex setting}\label{sec4}
In this section, we make the following assumption,
under which problem \eqref{opt-dc-hc} reduces to a convex optimization problem, thanks to Lemma~\ref{lm-g-mu}(iv).
\begin{assumption}\label{ass-conv}
	In \eqref{opt-dc-hc}, $f$ is convex, $P_2 = 0$, and $c$ is $(- \mathrm{hzn}\, h)$-convex.
\end{assumption}
Our aim is to establish the convergence of the whole sequence generated by Algorithm~\ref{alg-sl-sm-esqm} and study its local convergence rate under additional assumptions. Since we are interested in the asymptotic behavior, it suffices to study the behavior of the sequences generated by Algorithm~\ref{alg-sl-sm-esqm} for all {\em sufficiently large} $k$. Specifically, we consider the {\em last} iteration at which $\{t_k\}$ is increased, i.e.,
\begin{equation}\label{bark}
    \bar k\coloneqq \min \{ k \in \mathbb{N}_0 : t_j = t_k, \forall j\ge k\} - 1,
\end{equation}
and look at $x^k$ for all $k > \bar k$.
Observe from Theorem~\ref{th-gamma-bd} that $\bar{k} < \infty$ under Assumptions~\ref{ass-gen},
	\ref{ass-cq} and
	\ref{ass-mu}. 
    Note that unlike the $M_{\bar t}$ in \eqref{eq-I-bd} that arises in \eqref{eq-K-cond} as a lower bound for $K/2$, the $\bar k$ in \eqref{bark} cannot be computed without running the algorithm in general. Fortunately, we do not need the precise value of $\bar k$ here as we are focusing on the asymptotic behavior of our algorithm.

\begin{lemma}\label{lm-kbar-conv}
	Consider \eqref{opt-dc-hc}.
	Suppose that Assumptions~\ref{ass-gen},
	\ref{ass-cq},
	\ref{ass-mu} and \ref{ass-conv} hold.
	Let $\psi_{*}$ and $\Omega^*$ be the optimal value and the solution set of \eqref{opt-dc-hc}, respectively.
	Let $\{x^{k}\}$ and $\{\gamma_k\}$
	be generated by Algorithm~\ref{alg-sl-sm-esqm}.
    Let $\alpha_{3,1}$, $\alpha_{3,2}$, $\alpha_3$ and $\alpha_4$ be given in \eqref{eq-def-alpha} and \eqref{eq-def-alpha4}.
    Let $L_{\mu}$ and $M_{L}$
	be given in \eqref{eq-L-mu} and \eqref{eq-ML}, respectively.
	Let  $\bar{\gamma} \coloneqq \gamma_{\bar{k}+1}$, where $\bar k$ is defined in \eqref{bark}.
	Then the following statements hold.
\begin{enumerate}[{\rm (i)}]
    \item There exists an accumulation point of $\{x^{k}\}$ that solves both \eqref{opt-dc-hc}
    and the following minimization problem:
    \begin{equation}\label{opt-pen}
        \min_{x\in \mathbb{X}}\ \ \  \psi(x) + \bar{\gamma}^{-1} g^{+}(x).
    \end{equation}

	\item For every integer $k> \bar{k}$,
	\begin{equation}\label{eq-des-P-2}
		{\widetilde P}_{k+1}
		\le {\widetilde P}_{k}
		-  \frac{{c_1}}{2 \bar{\gamma} \mu_{k}}\|x^{k+1} - x^{k}\|^2
	\end{equation}
	where
	\begin{equation}\label{eq-P-def}
		{\widetilde P}_{k} \coloneqq \psi(x^{k})
		+ \bar{\gamma}^{-1} ( g_{\mu_{k}}(x^{k}) + \alpha_{4}  \mu_{k} ).
	\end{equation}
	Moreover, $\lim_{k\to\infty} {\widetilde P}_{k} = \psi_{*}$.
	
	\item For every integer $k> \bar{k}$ and every $x^*\in \Omega^*$,
	\begin{align} \label{eq-key-conv-1}
		&  C_0 \mu_{k} {\widehat Q}_{k+1}
        \le  C_1 \mu_{k}^2 +   C_2 (\mu_k {\widehat Q}_{k} - \mu_{k+1}{\widehat Q}_{k+1})
		+  \| x^{k} - x^*\|^2 - \|  x^{k+1} - x^*\|^2,
	\end{align}
	where
	\begin{align}\label{eq-def-psi-k}
        & {\widehat Q}_{k} \coloneqq \psi(x^{k})
		+ \bar{\gamma}^{-1}
		g^{+}(x^{k}) - \psi_{*} \ge 0,
		\\&
		C_0 = \frac{2\bar{\gamma}}{M_{L}}, \quad
        C_1 =  \frac{2\alpha_{3} c_1 + 2 a_1 ( [\alpha_{3,1}]_{+} + \alpha_{3,2} + \alpha_{4})}{c_1 \underline{L}},
         \quad
		C_2 = \frac{2  \bar{\gamma} a_1}{c_1 \underline{L}} \label{eq-def-C012}
	\end{align}
    with $a_1 \coloneqq [\mu_0 \widehat{\gamma}_{0} L_f + \mu_0 L_{\mu_{0}} - \underline{L}]_{+}$.	
	\item For every integer $k > \bar{k}$ and every $x^*\in \Omega^*$,	
	\begin{align}
		& \mu_{k}  ( {\widetilde P}_{k+1} - \psi_{*}
		+ \bar{\gamma}^{-1} \omega_{k+1}^2) \notag
		\\ & \le C_3 \mu_k^2
		+  C_4  ({\widetilde P}_{k} - {\widetilde P}_{k+1})
		+ \frac{M_{L}}{2 \bar{\gamma}} (\| x^{k} - x^*\|^2 - \|  x^{k+1} - x^*\|^2),
	\label{eq-key-conv-2}
	\end{align}
	where ${\widetilde P}_k$ is defined in \eqref{eq-P-def}, $\omega_{k+1}$ is given in \eqref{eq-omega}, and
	\begin{equation}\label{eq-C-3-4}
		C_3 =  \frac{([\alpha_{3,2}]_{+} + \alpha_4) M_L}{\eta_{\gamma} \underline{L}}, \quad
		C_4 = \frac{2 \underline{L} + \mu_{0} a_1 M_L }{c_1 \underline{L}},
	\end{equation}
    with $a_1$ defined in item {\rm (iii)}.
\end{enumerate}
\end{lemma}
\begin{proof}
    Using Theorem~\ref{th-gamma-bd}, we obtain that $\bar{k} < \infty$ is well-defined,
    and that $\gamma_{k} \equiv \bar{\gamma}$ for all $k> \bar{k}$.
    Item (i) follows directly from Remark~\ref{rm-stationary-pen}, together with Assumption~\ref{ass-conv}.

    We now turn to proving items (ii)-(iv).
    In what follows, let $k$ be an integer greater than $\bar{k}$.

	From \eqref{eq-gmu-Lips} and $\mu_{k+1}\le \mu_k$,
	it follows that
	\begin{align*}
		{\widetilde P}_{k+1}
		& = \psi(x^{k+1})  +  \bar{\gamma}^{-1}  (g_{\mu_{k+1}}(x^{k+1})
		+ \alpha_4 \mu_{k+1})
		\\ & \le \psi(x^{k+1})  +   \bar{\gamma}^{-1}
		(g_{\mu_{k}}(x^{k+1}) + \alpha_4 \mu_k)
		\\ & \le \psi(x^{k})
		+ \bar{\gamma}^{-1} ( g_{\mu_{k}}(x^{k}) + \alpha_{4}  \mu_{k} )
		 - \frac{{c_1}}{2 \bar{\gamma} \mu_{k}}\|x^{k+1} - x^{k}\|^2,
	\end{align*}
	where we use \eqref{eq-ls-desc} with $\widetilde{x} = x^{k+1}$ to get the last inequality.
    This implies \eqref{eq-des-P-2}.
	
	Furthermore, by Theorem~\ref{th-complexity-S-mu}
	and Assumption~\ref{ass-conv},
	there exists an infinite subsequence
	$\{x^{k_j}\}$ such that $\lim_{j\to \infty}x^{k_j} = \bar{x}$
	is a solution of \eqref{opt-dc-hc}.
	Since $\{{\widetilde P}_{k}\}$ is nonincreasing, we have
	\begin{equation*}
		\lim_{k\to \infty} 	{\widetilde P}_{k}
		= \lim_{j\to \infty}  (\psi(x^{k_j})
		+  \bar{\gamma}^{-1} ( g_{\mu_{k_j}}(x^{k_j}) + \alpha_{4}  \mu_{k_j} ))
		= \psi_{*},
	\end{equation*}
	where the last equality holds because $\lim_{j\to \infty}g_{\mu_{k_j}}(x^{k_j})
	= g^{+}(\bar{x}) = 0$ (see \eqref{eq-gmu-appr}).
    This shows item (ii).

    Before proceeding to prove items (iii) and (iv), we first derive the auxiliary inequality in \eqref{eq-gammak-fk1} below.
    To this end, note from the convexity of $g_{\mu_k}$ (see Lemma~\ref{lm-g-mu}(iv)) that,
    for every $x\in \mathcal{F}$,
		\begin{align*}
			\ell_{k}(x)
			& = g_{\mu_{k}}(x^{k}) +  \langle \nabla g_{\mu_{k}}(x^{k}), x - x^{k} \rangle
			 \le g_{\mu_{k}}(x)
			\le g^{+}(x)  + \alpha_{3,2} \mu_k = \alpha_{3,2} \mu_k,
		\end{align*}
	where the second inequality holds thanks to
	\eqref{eq-gmu-appr}, and the last equality follows from the fact that $x\in \mathcal{F}$.
	Moreover,
	using the convexity of $f$, we have that
	\begin{equation*}
		f(x^{k}) + \langle \nabla f(x^{k}), x - x^{k} \rangle
		\le f(x) \quad \forall x\in \mathcal{F}.
	\end{equation*}
	The last two displays, together with \eqref{eq-des-xk1}
	and $P_2 \equiv 0$, yield that, for every $x\in \mathcal{F}$,
    		\begin{align}
			& \psi(x^{k+1})
			+  \bar{\gamma}^{-1}g_{\mu_k}(x^{k+1})
			-  \psi(x) - \alpha_{3, 2} \bar{\gamma}^{-1}  \mu_{k}
            - \frac{L_k}{2\mu_{k}\bar{\gamma}} ( \| x^{k} - x\|^2 - \|  x^{k+1} - x\|^2)
            \notag
			\\ & \le \frac{\mu_k \bar{\gamma} L_f + \mu_k L_{\mu_{k}} - L_{k}}{2\mu_{k}\bar{\gamma}} \| x^{k+1} - x^{k}\|^2
			 \le   \frac{a_1}{2 \mu_k \bar{\gamma}}
			\| x^{k+1} - x^{k}\|^2  \le \frac{a_1}{c_1} ({\widetilde P}_{k} - {\widetilde P}_{k+1}),
            \label{eq-gammak-fk1}
		\end{align}
	where $a_1 = [\mu_0 \widehat{\gamma}_{0} L_f + \mu_0 L_{\mu_{0}} - \underline{L}]_{+}$,
	the second inequality holds thanks to
	$\mu_k \le \mu_{0}$, $\bar{\gamma} \le \widehat{\gamma}_{0}$,
    $\mu_k L_{\mu_k} \le \mu_{0} L_{\mu_{0}}$ (see \eqref{eq-L-mu}) and $L_k\ge \underline{L}$ (see \eqref{eq-ML}),
	and the last inequality follows from \eqref{eq-des-P-2}.

	Now, we are ready to prove item (iii).
    According to item (i), $\psi_{*}$ is also the optimal value of \eqref{opt-pen},
    which implies the nonnegativity claim in \eqref{eq-def-psi-k}.
    Combining this with $L_k\le M_L$ (see \eqref{eq-ML}), we have, for every $x^*\in \Omega^*$,
	\begin{align}
		& \frac{2\mu_{k}\bar{\gamma}}{M_{L}} {\widehat Q}_{k+1} \le
        \frac{2\mu_{k}\bar{\gamma}}{L_k} \left(\psi(x^{k+1})
		+  \bar{\gamma}^{-1}
		g^{+}(x^{k+1})
		-  \psi(x^*) \right) \notag
		\\ &  \le \frac{2\mu_{k}\bar{\gamma}}{L_k}
		\left(\psi(x^{k+1})
		+  \bar{\gamma}^{-1}
		(g_{\mu_{k}}(x^{k+1}) + \alpha_{3,1} \mu_{k})
		-  \psi(x^*) \right) \notag
		\\ & \le  \frac{2(\alpha_{3,1}  + \alpha_{3, 2})}{L_k}
		\mu_{k}^2 +
		\frac{2 a_1 \bar{\gamma}}{c_1 L_k} \mu_{k}({\widetilde P}_{k} - {\widetilde P}_{k+1})
		+  \| x^{k} - x^*\|^2 - \|  x^{k+1} - x^*\|^2 \notag
        \\  & \le  \frac{2\alpha_{3}}{\underline{L}}
		\mu_{k}^2 +
		\frac{2 a_1 \bar{\gamma}}{c_1 \underline{L}} \mu_{k}({\widetilde P}_{k} - {\widetilde P}_{k+1})
		+  \| x^{k} - x^*\|^2 - \|  x^{k+1} - x^*\|^2 \notag
        \\ & \le \left( \frac{2\alpha_{3}}{\underline{L}}
        + \frac{2 a_1 ( \alpha_{3,2} + [\alpha_{3,1}]_{+} + \alpha_{4})}{c_1 \underline{L}}
        \right)
		\mu_{k}^2
       + \frac{2 a_1 \bar{\gamma}}{c_1 \underline{L}} \mu_{k}({\widehat Q}_{k} - {\widehat Q}_{k+1}) \notag
		\\ & \qquad +  \| x^{k} - x^*\|^2 - \|  x^{k+1} - x^*\|^2, \label{Add_number}
	\end{align}
	where the second inequality follows from \eqref{eq-gmu-appr},
	the third inequality follows from \eqref{eq-gammak-fk1} upon setting $x=x^{*}$,
    and the fourth inequality holds because $\alpha_{3,1}  + \alpha_{3, 2} = \alpha_3 > 0$, $L_k\ge \underline{L}$ (see \eqref{eq-ML}) and
    ${\widetilde P}_k - {\widetilde P}_{k+1}\ge 0$ (see \eqref{eq-des-P-2}), and the last inequality follows from the facts that
    \begin{align*}
        {\widetilde P}_{k} & = \psi(x^{k})
		+ \bar{\gamma}^{-1} ( g_{\mu_{k}}(x^{k}) + \alpha_{4}  \mu_{k})
         \overset{\rm (a)}\le \psi(x^{k})
		+ \bar{\gamma}^{-1} g^{+}(x^{k})
        + \bar{\gamma}^{-1} (\alpha_{3,2} + \alpha_{4})  \mu_{k}
        \\ &  = {\widehat{Q}_{k}} + \psi_{*} + \bar{\gamma}^{-1} (\alpha_{3,2} + \alpha_{4})  \mu_{k},
        \\ {\widetilde P}_{k+1} & = \psi(x^{k+1})
		+ \bar{\gamma}^{-1} ( g_{\mu_{k+1}}(x^{k+1}) + \alpha_{4}  \mu_{k+1})
         \overset{\rm (a)}\ge \psi(x^{k+1})
		+ \bar{\gamma}^{-1} ( g^{+}(x^{k+1})
        -   \alpha_{3,1}\mu_{k+1} )
        \\ & \ge \psi(x^{k+1})
		+ \bar{\gamma}^{-1} ( g^{+}(x^{k+1})
        -   [\alpha_{3,1}]_{+}\mu_{k+1} )
        \overset{\rm (b)}\ge {\widehat{Q}_{k+1}} + \psi_{*} - \bar{\gamma}^{-1} [\alpha_{3,1}]_{+}\mu_{k};
    \end{align*}
    here, we have used \eqref{eq-gmu-appr} in (a) and the fact $\mu_{k+1} \le \mu_{k}$ in (b).
    Combining \eqref{Add_number} with the fact that $\mu_{k+1} \le \mu_{k}$
    and ${\widehat Q}_{k+1} \ge 0$, we conclude that \eqref{eq-key-conv-1} holds.
	
	Finally, we prove item (iv).
    Using item (ii), we have ${\widetilde P}_{k+1} - \psi_{*}\ge 0$,
    which, together with \eqref{eq-ML} and the definition of ${\widetilde P}_{k+1}$ in \eqref{eq-P-def}, implies that, for every $x^*\in \Omega^*$,
	\begin{align*}
		& \mu_{k}( {\widetilde P}_{k+1} - \psi_{*})
		\le \frac{M_L \mu_{k}}{L_k}
		\left(\psi(x^{k+1})  +  \bar{\gamma}^{-1}  (g_{\mu_{k+1}}(x^{k+1})
		+ \alpha_4 \mu_{k+1}) - \psi(x^*)\right) \notag
		\\ & \le \frac{M_L \mu_{k}}{L_k}
		\left( \psi(x^{k+1})  +   \bar{\gamma}^{-1}
		(g_{\mu_{k}}(x^{k+1}) + \alpha_4 \mu_k)
		  - \psi(x^*)\right)   \notag
\\ & \le  \frac{(\alpha_{3,2}  + \alpha_4) M_L}{\bar{\gamma} L_k} \mu_k^2
+  \frac{a_1 M_L \mu_{k}}{c_1 L_k}  ({\widetilde P}_{k} - {\widetilde P}_{k+1})
+ \frac{M_L}{2 \bar{\gamma}} (\| x^{k} - x^*\|^2 - \|  x^{k+1} - x^*\|^2)
\notag
\\ & \le  \frac{([\alpha_{3,2}]_{+} + \alpha_4) M_L}{\eta_{\gamma} \underline{L}} \mu_k^2
\!+\!  \frac{ \mu_{0} a_1 M_L}{c_1 \underline{L}}  ({\widetilde P}_{k} - {\widetilde P}_{k+1}) \!
+\! \frac{M_L}{2 \bar{\gamma}} (\| x^{k} - x^*\|^2 \!-\! \|  x^{k+1} - x^*\|^2),
	\end{align*}
	the second inequality follows from \eqref{eq-gmu-Lips} and the fact that $\mu_{k+1}\le \mu_k$,
	the third inequality holds because of \eqref{eq-gammak-fk1} (with $x = x^*$),
	and the last inequality holds thanks to \eqref{eq-ML}, \eqref{eq-gamma-pos-low-bd},
	$\mu_{k}\le \mu_{0}$ and ${\widetilde P}_{k} - {\widetilde P}_{k+1}\ge 0$ (see \eqref{eq-des-P-2}).
	
	Furthermore, using \eqref{eq-des-P-2} and the definition of $\omega_{k+1}$ in \eqref{eq-omega},
	we obtain
	$$
	\bar{\gamma}^{-1} \mu_k \omega_{k+1}^2 \le \tfrac{2}{c_1} ({\widetilde P}_{k} - {\widetilde P}_{k+1}).
	$$
	Summing the last two displays, we conclude that \eqref{eq-key-conv-2} holds.
\end{proof}

We are now ready to establish asymptotic convergence and the (asymptotic) convergence rate in the next theorem.
\begin{theorem}[Asymptotic convergence rate in convex setting]\label{th-compl-conv}
	Consider \eqref{opt-dc-hc}.
	Suppose that Assumptions~\ref{ass-gen},
	\ref{ass-cq}, \ref{ass-mu} and \ref{ass-conv} hold.
	Let $\psi_{*}$ be the optimal value of \eqref{opt-dc-hc}.
	Let $\{x^{k}\}$ be generated by Algorithm~\ref{alg-sl-sm-esqm}.
	Let $M_L$ and $\eta_{\gamma}$ be given in \eqref{eq-ML} and \eqref{eq-gamma-pos-low-bd},
	respectively.
	Let $C_3$ and $C_4$ be given in \eqref{eq-C-3-4}, $\alpha_{3,1}$, $\alpha_{3,2}$ and $\alpha_4$ be given in \eqref{eq-def-alpha} and \eqref{eq-def-alpha4}, and $\bar{k}$ be defined in \eqref{bark}.
	Let $\epsilon>0$ and integer $K \ge 2(\bar{k}+1)$ satisfy the following conditions
	\begin{align}
		S_{K}^{-1} & \le  \frac{ \min \left\{ \eta^2_{\gamma} c_2^2,  \eta_{\gamma} \epsilon /2 \right\} }{2C_4\eta_\gamma B_1 + M_{L}m_{\mathcal{C}}^2} , \label{eq-cond-conv-SK}
		\\ \mu_{\lceil K/2\rceil} & \le \min \left\{ \frac{\eta_{\gamma} c_2^2}{2C_3},\
        \frac{\eta_{\gamma}\epsilon}{2 + 2[\alpha_{3,1}]_{+}},\
		\frac{\epsilon}{4 C_3}, \
		\frac{ \epsilon}{3 \alpha_{3}^{\varphi} + [\alpha_{3, 1}^{h}]_{+}}, \
		\frac{ \eta_{\gamma} \epsilon}{\alpha_{4} + 2 \alpha_{3}^{\varphi} + [\alpha_{3,2}^{\varphi}]_{+}}
		\right\}, \label{eq-cond-conv-muK/2}
	\end{align}
	where
	\begin{align*}
			& B_1 \coloneqq m_{\psi} + \eta_{\gamma}^{-1}
		(m_{g} +  ([\alpha_{3,2}]_{+} + \alpha_{4})  \mu_{0} ) -\psi_*,
		\\ & m_{\mathcal{C}} \coloneqq \sup_{x, z\in \mathcal{C}} \| x - z\|,\
		m_{\psi} \coloneqq \sup_{x\in \mathcal{C}} \psi(x), \
		m_{g} \coloneqq \sup_{x\in \mathcal{C}} g^{+}(x) \ge 0.
	\end{align*}	
	Then there exists $\widehat{k} \in [\lceil K/2 \rceil, K]$ such that
    \begin{equation}\label{eq-conv-eps-solution}
        | \psi(x^{\widehat{k}+1}) - \psi_{*} | \le \epsilon
        \quad \mbox{and} \quad
        g(x^{\widehat{k}+1}) \le \epsilon.
    \end{equation}
	If we assume, in addition, that $\sum_{k=0}^{\infty} \mu_k^2 < \infty$,
	then $\{x^{k}\}$ converges to a solution of both \eqref{opt-dc-hc} and \eqref{opt-pen}.
\end{theorem}
\begin{proof}
    Let $\Omega^*$ be the solution set of \eqref{opt-dc-hc}.
	Since $\lceil K/2\rceil \ge \bar{k}+1$,
	we have that $\gamma_{k} \equiv \gamma_{\bar{k}+1} \eqqcolon \bar{\gamma}$ for all $k \ge \lceil K/2\rceil$.
	Invoking Lemma~\ref{lm-kbar-conv}(ii),
	we obtain that ${\widetilde P}_{k+1}-\psi_{*}\ge 0$ for all $k \ge \lceil K/2\rceil$.
	Moreover, with $\omega_k \coloneqq \|x^{k+1} -x^k\|/\mu_k$ defined as in \eqref{eq-omega}, we see that
	\begin{align*}
		&  \min_{\lceil K/2\rceil \le k \le K}
		 \{{\widetilde P}_{k+1} -  \psi_{*}
		+  \bar{\gamma}^{-1} \omega_{k+1}^2 \}
		S_K \le \sum_{k=\lceil K/2\rceil}^{K}
		    \left( {\widetilde P}_{k+1} - \psi_{*}
		 + \bar{\gamma}^{-1} \omega_{k+1}^2\right)\mu_{k}
		\\ & \le
		C_3 \sum_{k=\lceil K/2\rceil}^{K}   \mu_k^2
		+  C_4  \left({\widetilde P}_{\lceil K/2\rceil} - {\widetilde P}_{K+1}\right)
		 +  \frac{M_{L}}{2 \bar{\gamma}} \left(\| x^{\lceil K/2\rceil} - x^*\|^2\! -\! \|  x^{K+1} - x^*\|^2 \right)
		\\ & \le
		 C_3 \mu_{\lceil K/2\rceil} S_{K}
		+  C_4  B_1
		+  \frac{M_{L}m_{\mathcal{C}}^2}{2 \eta_{\gamma}},
	\end{align*}
	where the second inequality follows from \eqref{eq-key-conv-2},
	and the last inequality holds thanks to the nonincreasing property of $\{\mu_k\}$,
	\eqref{eq-gamma-pos-low-bd}, $x^{\lceil K/2\rceil}, x^{k+1}, x^*\in \mathcal{C}$,
	 and the fact that
	\begin{align*}
		\psi_{*} &\le {\widetilde P}_{k}  = \psi(x^{k})
		+ \bar{\gamma}^{-1} ( g_{\mu_{k}}(x^{k}) + \alpha_{4}  \mu_{k} )
		\le  \psi(x^{k})
		+ \bar{\gamma}^{-1} ( g^{+}(x^{k}) + (\alpha_{3,2} + \alpha_{4})  \mu_{k} )
		\\ & \le m_{\psi} + \eta_{\gamma}^{-1}
		(m_{g} +  ([\alpha_{3,2}]_{+} + \alpha_{4})  \mu_{0} ) = B_1 + \psi_* \quad \forall k>\bar{k};
	\end{align*}
	here we use \eqref{eq-gmu-appr} to obtain the second inequality, and use $\mu_k \le \mu_{0}$
	and the definitions of $m_{\psi}$ and $m_{g}$ to get the last inequality.
	Therefore, there exists an integer $\widehat{k} \in [\lceil K/2\rceil, K]$ such that
	\begin{align}\label{eq-comple-penalty}
		{\widetilde P}_{\widehat{k}+1} -  \psi_{*} + \bar{\gamma}^{-1}\omega_{\widehat{k}+1}^2
		& \le C_3 \mu_{\lceil K/2\rceil}
		 +  \frac{2C_4\eta_\gamma B_1 + M_{L}m_{\mathcal{C}}^2}{2 \eta_{\gamma} S_{K}}
		 \le \min\{ \epsilon/2, \eta_{\gamma} c_2^2\},
	\end{align}
	where the last inequality holds thanks to \eqref{eq-cond-conv-SK} and \eqref{eq-cond-conv-muK/2}.
	In addition, by applying the definition of ${\widetilde P}_{\widehat{k}+1}$
	in \eqref{eq-P-def}, we have
	\begin{align}\label{eq-phi-up}
		& \psi(x^{\widehat{k}+1}) - \psi_{*} = {\widetilde P}_{\widehat{k}+1} -  \psi_{*} - \bar{\gamma}^{-1} \left(g_{\mu_{\widehat{k}+1}}(x^{\widehat{k}+1}) + \alpha_4 \mu_{\widehat{k}+1} \right)
        \notag
        \\ & \le {\widetilde P}_{\widehat{k}+1} -  \psi_{*}
        - \bar{\gamma}^{-1} \left(g^{+}(x^{\widehat{k}+1}) + (\alpha_4 - \alpha_{3,1}) \mu_{\widehat{k}+1} \right)
        \notag
        \\ & \le \frac{\epsilon}{2}
        + [\alpha_{3,1}]_{+} \bar{\gamma}^{-1}   \mu_{\widehat{k}+1}
        \le \frac{\epsilon}{2} + [\alpha_{3,1}]_{+} \eta_{\gamma}^{-1}   \mu_{\lceil K/2 \rceil}
        \le \epsilon,
	\end{align}
    where the first inequality follows from \eqref{eq-gmu-appr},
    the second inequality follows from \eqref{eq-comple-penalty},
    the third inequality holds thanks to \eqref{eq-gamma-pos-low-bd} and
    $\mu_{\widehat{k}+1} \le \mu_{\lceil K/2 \rceil}$, and the last inequality holds because of \eqref{eq-cond-conv-muK/2}.
	
	On the other hand, it follows from \eqref{eq-comple-penalty} that
	\begin{equation}
		\omega_{\widehat{k}+1}^2
		\le \min\{\bar{\gamma} \epsilon/2, \bar{\gamma} \eta_{\gamma} c_2^2\}
		\le \min \{\widehat{\gamma}_{0}\epsilon/2,  c_2^2 \gamma_{\widehat{k}}^2 \},
        \notag
	\end{equation}
	where the last inequality holds thanks to $\bar{\gamma} = \gamma_{\widehat{k}} \le \widehat{\gamma}_{0}$
    (since we have $\widehat{k} > \bar k$) and $\eta_{\gamma} \le \gamma_{\widehat{k}}$ (see \eqref{eq-gamma-pos-low-bd}).
	Combining the definition of $\omega_{\widehat{k}+1}$ with the last display, we obtain that $\|x^{\widehat{k}+1} - x^{\widehat{k}}\|
	\le c_2 \gamma_{\widehat{k}} \mu_{\widehat{k}}$.
	Since $\widehat{k} \ge \lceil K/2\rceil \ge \bar{k}+1$,
	invoking the definition of $\bar{k}$ and {\bf Step~4} of Algorithm~\ref{alg-sl-sm-esqm},
	this further yields that $h_{\mu_{\widehat{k}}}(c(x^{\widehat{k}+1}) )
			 \le 2 \alpha_{3}^{\varphi} \mu_{\widehat{k}}$.
	Invoking this and Lemma~\ref{lm-properties}(iii), we obtain that
	\begin{align}
		& g^{+}(x^{\widehat{k}+1})
		 \le ( 3 \alpha_{3}^{\varphi} + [\alpha_{3, 1}^{h}]_{+}
		) \mu_{\widehat{k}}  \le \epsilon,
		\label{eq-feas}
	\end{align}
	where the last inequality follows from the fact that $\mu_{\widehat{k}} \le \mu_{\lceil K/2\rceil}$ and \eqref{eq-cond-conv-muK/2}.
	
	Moreover, since ${\widetilde P}_{\widehat{k}+1} - \psi_{*} \ge 0$,
	by the definition of ${\widetilde P}_{\widehat{k}+1}$ in \eqref{eq-P-def},
	we have
	\begin{align*}
		& \psi_{*} - \psi(x^{\widehat{k}+1})
		 \le \bar{\gamma}^{-1} ( g_{\mu_{\widehat{k}+1}}(x^{\widehat{k}+1}) + \alpha_{4}  \mu_{\widehat{k}+1} )
		 \overset{\mathrm{(a)}}{\le} \bar{\gamma}^{-1} ( g_{\mu_{\widehat{k}}}(x^{\widehat{k}+1}) + \alpha_{4}  \mu_{\widehat{k}} )
         \\ &
        \overset{\mathrm{(b)}}{\le} \bar{\gamma}^{-1} (\alpha_{4} + 2 \alpha_{3}^{\varphi} + [\alpha_{3,2}^{\varphi}]_{+})
		\mu_{\widehat{k}}
		\overset{\mathrm{(c)}}{\le}   \eta_{\gamma}^{-1} (\alpha_{4} + 2 \alpha_{3}^{\varphi} + [\alpha_{3,2}^{\varphi}]_{+})
		\mu_{\widehat{k}}
		 \overset{\mathrm{(d)}}{\le}  \epsilon,
	\end{align*}
	where (a) holds because of \eqref{eq-gmu-Lips}
	and $\mu_{\widehat{k}+1} \le \mu_{\widehat{k}}$,
	(b) follows from the first relation in \eqref{eq-feas-upbd}, (c) follows from \eqref{eq-gamma-pos-low-bd}, and
	(d) holds thanks to \eqref{eq-cond-conv-muK/2} and the fact that $\mu_{\hat k}\le \mu_{\lceil K/2\rceil}$.
	Combining the last display with \eqref{eq-phi-up} and \eqref{eq-feas},
	we conclude that \eqref{eq-conv-eps-solution} holds.

	Finally, suppose further that $\sum_{k=0}^{\infty} \mu_k^2 < \infty$.
	Since ${\widetilde P}_{k+1} \ge \psi_{*}$ for all $k>\bar{k}$,
	we have from \eqref{eq-key-conv-2} that, for every $x^*\in \Omega^*$,
	\begin{align}
		\|  x^{k+1} - x^*\|^2
		& \le \| x^{k} - x^*\|^2 	+  \frac{2\bar \gamma C_3}{ M_{L}} \mu_{k}^2 +
		\frac{2 \bar\gamma C_4}{M_L} ({\widetilde P}_{k} - {\widetilde P}_{k+1}) \notag
		\quad \forall k > \bar{k}.
	\end{align}
	Note that, for each $k> \bar{k}$, ${\widetilde P}_k - {\widetilde P}_{k+1}$ is nonnegative (see \eqref{eq-des-P-2}).
	Moreover, since $\{{\widetilde P}_k\}_{k > \bar{k}}$ is bounded below by $\psi_{*}$ (see Lemma~\ref{lm-kbar-conv}(ii)),
	we have that $\{{\widetilde P}_k - {\widetilde P}_{k+1}\}_{k> \bar{k}}$ is summable.
	Combining this fact with Remark~\ref{rm-stationary-pen} and \cite[Proposition~1]{03Iusem},
	we deduce that
	$\{x^{k}\}$ converges to $\bar{x} \in \Omega^*$,
    which is also a solution of \eqref{opt-pen} in view of Lemma~\ref{lm-kbar-conv}(i).
\end{proof}

We next turn to the local convergence rate. We first present an auxiliary lemma concerning a particular difference inequality.
The induction argument in the proof follows the same approach used in \cite[Lemma 4.2]{24LMX} and \cite[Theorem 4.5]{25XPS}.

\begin{lemma}\label{lm-rate}
Let $r_1 \in (0, 1)$, $r_2 > r_1$, $p\ge 1$, $a_0 > 0$, $a_1 > 0$ and $N_1 \in \mathbb{N}_0$.
    Let $\{\Delta_{k}\}$ be a nonnegative sequence such that
    \begin{equation}\label{eq-rec}
         \Delta_{k+1} + a_0 (1+k)^{-r_1} \Delta_{k+1}^{p}
         \le \Delta_{k} + a_1 (1+k)^{-r_2} \quad \forall k \ge N_1.
    \end{equation}
    Then there exist $a_2> 0$ and $N_2\in \mathbb{N}_0$ such that
    \begin{equation}\label{eq-Delta-rate-lm}
        \Delta_{k} \le a_2 (1+k)^{-s} \quad \forall k\ge N_2,
    \end{equation}
    where
    \begin{equation}\label{eq-s-aux}
        s = \begin{cases}
            r_2 - r_1 & \mbox{if}\ p = 1,
            \\ \min\left\{ \frac{r_2 - r_1}{p}, \frac{1-r_1}{p - 1} \right\}
            & \mbox{if}\ p > 1.
        \end{cases}
    \end{equation}
\end{lemma}
\begin{proof}
    Let
    \begin{equation}\label{eq-def-N2-aux}
        N_2 = \max\{ N_1,\ \lceil  (2{b_1})^{\frac{1}{1 - r_1}} \rceil \}
    \end{equation}
    and
    \begin{equation} \label{eq-def-a2}
        a_2 = \begin{cases}
             \max \left\{ \frac{\Delta_{N_2}}{(1+N_2)^{-s}},\ 2b_2 \right\} & \mbox{if}\ p = 1,
            \\
            \max \left\{ \frac{\Delta_{N_2}}{(1+N_2)^{-s}},
            \ (2b_2)^{\frac{1}{p}},
            \ (2 b_1)^{\frac{1}{p-1}} \right\} &  \mbox{if}\ p > 1,
        \end{cases}
    \end{equation}
    where
    \begin{equation}\label{eq-def-b12}
        b_1 = s a_0^{-1}  2^{sp}
        \quad \mbox{and} \quad
        b_2 = a_1 a_0^{-1} 2^{sp}.
    \end{equation}
    We prove by induction that \eqref{eq-Delta-rate-lm} holds
    with $a_2$ and $N_2$ given above.

    Clearly, $\Delta_{N_2} \le a_2 (1+N_2)^{-s}$ by the definition of $a_2$.

    Suppose now that $\Delta_{k} \le a_2 (1+k)^{-s}$ for some $k\ge N_2$.
    We aim to show that it also holds for $k+1$.
    To this end, define
    \begin{equation}\label{eq-def-phi}
        \phi(t) \coloneqq t + a_0 (1+k)^{-r_1} t^{p} \quad \forall t\ge 0.
    \end{equation}

    We claim that $\phi( \Delta_{k+1} ) - \phi( a_2 (2+k)^{-s} ) \le 0$,
    which, together with the (strictly) increasing property of $\phi$ on $[0, \infty)$,
    completes the proof.

    Indeed, we have from \eqref{eq-def-phi} and \eqref{eq-rec} that
    \begin{align}
       &  \phi( \Delta_{k+1} ) - \phi( a_2 (2+k)^{-s} ) \notag
        \\ & \le \Delta_{k} + a_1 (1+k)^{-r_2} - a_2 (2+k)^{-s} -  a_0 a_2^{p} (1+k)^{-r_1}(2+k)^{-sp} \notag
        \\ & \le a_2 (1+k)^{-s}  + a_1 (1+k)^{-r_2} - a_2 (2+k)^{-s}
        -  a_0 a_2^{p} 2^{-sp} (1+k)^{-(sp + r_1)} \notag
        \\ & \le s a_2 (1+k)^{-(1+s)} + a_1 (1+k)^{-r_2}
        -  a_0 a_2^{p} 2^{-sp} (1+k)^{-(sp + r_1)} \notag
        \\ & = b_0 (1+k)^{-\widetilde{r}_0} \left(
        b_1 a_2^{1-p} (1+k)^{- \widetilde{r}_1} + b_2 a_2^{-p} (1+k)^{-\widetilde{r}_2}
        -  1
        \right), \label{eq-phi-k1}
    \end{align}
    where the second inequality follows from the induction hypothesis and the fact that $k+2 \le 2(k+1)$,
    the third inequality holds thanks to the convexity of $t\mapsto t^{-s}$ for $t \in (0, \infty)$,
    and the equality holds because of the definitions of $b_1$ and $b_2$ in \eqref{eq-def-b12}, upon defining
    \begin{align}
     b_0 \coloneqq a_0 2^{-sp} a_2^{p}, \quad \widetilde{r}_0 \coloneqq sp + r_1,
        \quad \widetilde{r}_1 \coloneqq s(1-p) + 1 - r_1,
        \quad \widetilde{r}_2 \coloneqq - sp + r_2 -r_1. \label{eq-def-td-r}
    \end{align}

    Now, we consider separately the cases $p=1$ and $p>1$.

    \noindent {\bf Case (i)}. $p = 1$. In this case, we have
    from \eqref{eq-s-aux} that $s = r_2 - r_1$, and hence
    $\widetilde{r}_1 = 1 - r_1$ and
    $\widetilde{r}_2 = 0$.
    Combining this with \eqref{eq-phi-k1}, we obtain
    \begin{align*}
        \phi( \Delta_{k+1} ) - \phi( a_2 (2+k)^{-s} )
       &   \le b_0  (1+k)^{-\widetilde{r}_0} (
        b_1 (1+k)^{- (1-r_1)} + b_2 a_2^{-1} -  1
        )
        \\ & \le b_0 (1+k)^{-\widetilde{r}_0} (
        b_1 (1+N_2)^{- (1-r_1)} -  \tfrac{1}{2}
        ) \le 0,
    \end{align*}
    where the second and the final inequalities hold, respectively, because
    $a_2 \ge 2 b_2$ (see \eqref{eq-def-a2}) and $N_2 \ge (2{b_1})^{\frac{1}{1 - r_1}}$ (see \eqref{eq-def-N2-aux}).

    \noindent {\bf Case (ii)}. $p > 1$. In this case, using \eqref{eq-s-aux} and \eqref{eq-def-td-r}, we have that
    $\widetilde{r}_1 \ge 0$ and
    $\widetilde{r}_2 \ge 0$.
    This, together with \eqref{eq-phi-k1}, shows that
        \begin{align}
       &  \phi( \Delta_{k+1} ) - \phi( a_2 (2+k)^{-s} ) \notag
         \le b_0  (1+k)^{-\widetilde{r}_0} \left(
        b_1 a_2^{1-p} + b_2 a_2^{-p} -  1
        \right) \le 0,
    \end{align}
    where the last inequality holds thanks to
    $b_1 a_2^{1-p} \le \frac{1}{2}$ and $b_2 a_2^{-p} \le \frac{1}{2}$,
    in view of \eqref{eq-def-a2}.
\end{proof}

Our local convergence rate result is based on the following stronger assumption on $\{\mu_k\}$ (compared with Assumption~\ref{ass-mu}(i)) and the H\"olderian error bound condition in \eqref{eq-KL-theta} concerning the objective function of \eqref{opt-pen}. The latter condition was also used in \cite[Theorem~4.5(iii)]{25ZPX} for studying local convergence rate of the extended sequential quadratic method with extrapolation; see \cite[Section~5]{25ZPX} for further discussions on this error bound condition.
\begin{assumption}\label{ass-mu-conv}
    There exist $\bar r \in (\frac{1}{2}, 1)$, $b_1>0$, $b_2>0$ and $k_1\in \mathbb{N}_{0}$ such that
    \begin{equation}\label{eq-mu-bd}
        b_1 (1 + k)^{-\bar r} \le \mu_{k} \le b_2 (1 + k)^{-\bar r} \quad \forall k \ge k_1.
    \end{equation}
\end{assumption}
\begin{theorem}\label{thm_local_rate}
    Consider \eqref{opt-dc-hc}.
	Suppose that Assumptions~\ref{ass-gen},
	\ref{ass-cq}, \ref{ass-mu}, \ref{ass-conv} and \ref{ass-mu-conv} hold.
    Let $\psi_{*}$ and $\Omega^*$ be the optimal value and the solution set of \eqref{opt-dc-hc}, respectively.
    Let $\{x^{k}\}$ and $\{\gamma_k\}$ be generated by Algorithm~\ref{alg-sl-sm-esqm}.
    Let $\bar{\gamma}$ be specified as in Lemma~\ref{lm-kbar-conv}.
    Assume there exist $\kappa>0$, $\theta\in [0, 1)$, $\varepsilon_0>0$ and $\varepsilon_1\in (0,1)$
	such that
	\begin{equation}\label{eq-KL-theta}
		\mathrm{dist}(x, \Omega^*) \le \kappa (\psi(x) + \bar{\gamma}^{-1} g^{+}(x) - \psi_{*})^{1-\theta}
	\end{equation}
	for all $x\in \mathbb{X}$ with $\mathrm{dist}(x, \Omega^*)\le
	\varepsilon_0$
	and $\psi(x) + \bar{\gamma}^{-1} g^{+}(x) \le \psi_{*} + \varepsilon_1$.
    Then $\{x^{k}\}$ converges to an $\bar{x} \in \Omega^*$, and there exist
    $\widetilde{\kappa} > 0$ and $\widetilde{k}\in \mathbb{N}_0$ such that
    \begin{equation}\label{eq-conv-rate}
        \| x^{k} - \bar{x} \| \le \widetilde{\kappa} (1 + k)^{-s} \quad \forall k \ge \widetilde{k},
    \end{equation}
    where
    		\begin{equation}\label{eq-rate-s}
			s \coloneqq s(\bar r) \coloneqq \begin{cases}
				{\bar r} - \frac{1}{2}
				& \mbox{if } \theta \in [0, \frac{1}{2}],\\
				\min\left\{{\bar r} - \frac{1}{2}, \frac{(1-{\bar r})(1-\theta)}{2\theta-1}\right\}
				& \mbox{if } \theta\in (\frac{1}{2}, 1).
			\end{cases}
		\end{equation}
\end{theorem}

	\begin{remark}
		\begin{enumerate}[{\rm (i)}]
			\item The local convergence rate in \eqref{eq-conv-rate} for {\em s}ESQM coincides with that established in \cite[Theorem~4.5]{25XPS} for {\em s}MBA.
			Recall that {\em s}MBA is a feasible method for \eqref{opt-dc-hc} with ${\cal C} = \mathbb{R}^n$ and $h$ chosen as in Example~\ref{ex-conic constraint}, and relies on {\em strictly feasible} initial point and the efficient evaluation of the prox-broximal operators as discussed in \cite[Section~6]{25XPS}, whereas {\em s}ESQM is an infeasible method and requires efficient evaluation of the proximal operator of $\tau P_1 + \delta_{\mathcal{C}}$ for any $\tau > 0$.
			
			\item We also note that, in the setting of $h(y) \coloneqq \max_{i\in [m]} y_i$ for $y \in \mathbb{Y} = \mathbb{R}^{m}$,
			the convergence rate in \eqref{eq-conv-rate} appears to be worse than that of ESQM and its variant; see, e.g., \cite[Theorem~4.5]{25ZPX}.
			However, notice that each subproblem in {\em s}ESQM only requires the evaluation of the proximal operator of $\tau P_1 + \delta_{\mathcal{C}}$ for some $\tau > 0$, while the subproblems considered in \cite{25ZPX} typically require an iterative solver and can only be solved inexactly in general. Thus, the costs of each outer iteration of the two methods are not directly comparable.
			
		\end{enumerate}
	\end{remark}

\begin{proof}[Proof of Theorem~\ref{thm_local_rate}]
Under Assumption~\ref{ass-mu-conv},
one can see that $\sum_{k=0}^{\infty} \mu_k^2 < \infty$. This, together with Theorem~\ref{th-compl-conv},
yields that the sequence $\{x^{k}\}$ converges
to some $\bar{x}$ that solves both \eqref{opt-dc-hc} and \eqref{opt-pen}.
Therefore, there exists an integer $k_{2} \in \mathbb{N}_{0}$ such that
\begin{equation}\label{eq-k2}
    \mathrm{dist}(x^{k}, \Omega^*)\le
	\varepsilon_0 \quad \mbox{and} \ \quad
   {\widehat Q}_{k} \coloneqq \psi(x^{k}) + \bar{\gamma}^{-1} g^{+}(x^{k}) - \psi_{*} \le \varepsilon_1 < 1
    \quad \forall k\ge k_{2},
\end{equation}
which implies that $x^{k}$ satisfies \eqref{eq-KL-theta}, that is,
\begin{equation}\label{eq-dist-psi}
    \mathrm{dist}(x^{k}, \Omega^*) \le \kappa {\widehat Q}_{k}^{1-\theta}
    \quad \forall k\ge k_2.
\end{equation}

For notational simplicity, we denote
\begin{equation} \label{eq-def-tau-Delta}
 \tau_{k} \coloneq \sum_{i=k}^{\infty} \mu_i^2
 \quad \mbox{and} \quad    \Delta_{k} \coloneqq \| x^{k} - \bar x\|^2 +
    C_1 \tau_{k} + C_2 \mu_k {\widehat Q}_{k} \quad \forall k\in \mathbb{N}_0,
\end{equation}
where $C_1$ and $C_2$ are given in \eqref{eq-def-C012}.

First, we note from \eqref{eq-mu-bd} that, for every $k\ge k_1$,
\begin{align}\label{eq-tau-up-bd}
     & \tau_{k+1}
    \le b_2^2 \sum_{i=k+1}^{\infty} (1 + i)^{-2\bar r}
     \le b_2^2 \int_{k}^{\infty} (1+t)^{-2 \bar{r}} dt
     = b_2^2 (2 \bar{r} - 1)^{-1} (1+k)^{-(2 \bar{r} - 1)}.
\end{align}
Moreover, we have, for every $k\ge k_1$,
\begin{align}
     & \tau_{k} \ge b_1^2 \sum_{i=k}^{\infty} (1 + i)^{-2\bar r}
    \ge b_1^2 \int_{k}^{\infty} (1+t)^{-2 \bar{r}} dt \notag
     = b_1^2 (2 \bar{r} - 1)^{-1} (1+k)^{-(2 \bar{r} - 1)} \notag
     \\ & \ge b_1^2 (2 \bar{r} - 1)^{-1} (1+k)^{-\bar{r}}
    \ge b_1^2 b_2^{-1} (2 \bar{r} - 1)^{-1} \mu_{k}, \notag
\end{align}
where the third inequality follows from the relation $0< 2 \bar{r} - 1 < \bar r$ (since $\bar{r} \in (\frac{1}{2}, 1)$), and the last inequality follows from \eqref{eq-mu-bd}.
This implies that
\begin{equation*}
    \mu_{k} \le b_1^{-2} b_2 (2 \bar{r} - 1) \tau_{k} \quad\forall k\ge k_1,
\end{equation*}
which, together with \eqref{eq-k2}, yields
\begin{equation} \label{eq-def-C5}
     C_1 \tau_{k}
    + C_2  \mu_{k} {\widehat Q}_{k} \le \left( C_1 + C_2 b_1^{-2} b_2 (2 \bar{r} - 1)\right) \tau_k
    \eqcolon C_5 \tau_k \quad\forall k\ge \max\{k_1, k_2\}.
\end{equation}

Now, let $\bar k$ be defined as in Lemma~\ref{lm-kbar-conv}.
Setting $x^* = \bar{x}$ in \eqref{eq-key-conv-1} yields that
\begin{align}\label{eq-mu psi-Delta}
    C_0 \mu_{k} {\widehat Q}_{k+1}
    & \le  C_1 \mu_{k}^2 +   C_2 (\mu_k {\widehat Q}_{k} - \mu_{k+1}{\widehat Q}_{k+1})
    +  \| x^{k} - \bar x \|^2 - \|  x^{k+1} - \bar x\|^2 \notag
    \\ & = \Delta_{k} - \Delta_{k+1} \quad \forall k> \bar{k}.
\end{align}

Furthermore, since ${\widehat Q}_{k+1} \ge 0$ for any $k>\bar{k}$ (see \eqref{eq-def-psi-k}),
we obtain from \eqref{eq-key-conv-1} that,
for every $\widetilde{x} \in \Omega^*$,
\begin{align*}
    & \|  x^{k+1} - \widetilde{x}\|^2
    \le  \| x^{k} - \widetilde{x}\|^2 +
    C_1 \mu_{k}^2 +   C_2 (\mu_k {\widehat Q}_{k} - \mu_{k+1}{\widehat Q}_{k+1})  \quad \forall k > \bar{k}.
\end{align*}
This implies that, for every $k>\max\{\bar{k}, k_1, k_2\}$, $k^{\prime} \in \mathbb{N}$ and $\widetilde{x} \in \Omega^*$,
\begin{align}
    \|  x^{k+k^{\prime}} - \widetilde{x}\|^2
   & \le  \| x^{k + k^{\prime} -1} - \widetilde{x}\|^2 +
    C_1 \mu_{k + k^{\prime} -1}^2 +   C_2 (\mu_{k + k^{\prime} -1}
    {\widehat Q}_{k + k^{\prime} -1} - \mu_{k + k^{\prime}}{\widehat Q}_{k + k^{\prime}})
    \notag
    \\ & \le \| x^{k} - \widetilde{x} \|^2
    + C_1 \sum_{i=k}^{k+k^{\prime}-1} \mu_i^2
    + C_2 ( \mu_{k} {\widehat Q}_{k} - \mu_{k + k^{\prime}}{\widehat Q}_{k + k^{\prime}}) \notag
    \\ & \le \| x^{k} - \widetilde{x} \|^2
    + C_1 \tau_{k}
    + C_2  \mu_{k} {\widehat Q}_{k}
    \le \| x^{k} - \widetilde{x} \|^2  + C_5 \tau_{k}, \notag
\end{align}
where the third inequality holds since ${\widehat Q}_{k+k^{\prime}} \ge 0$ (see \eqref{eq-def-psi-k}),
and the last inequality follows from \eqref{eq-def-C5}.

Now, setting $\widetilde{x} = {\bar x}^{k}$ with
$\| \bar{x}^{k} - x^{k}\| = \mathrm{dist}(x^{k}, \Omega^*) $ in the last display,
we obtain that
\begin{align}\label{eq-xk-kprime}
   \|  x^{k+k^{\prime}} - \bar{x}^{k}\|^2
    \le \mathrm{dist}(x^{k}, \Omega^*)^2
    + C_5 \tau_{k} \quad
    \forall k>\max\{\bar{k}, k_1, k_2\}, k^{\prime} \ge 1.
\end{align}
Using the triangle inequality, we have
\begin{align*}
    \| x^{k} - x^{k+k^{\prime}} \|^2
   & \le (\| x^{k} - \bar{x}^{k} \| + \| x^{k+k^{\prime}} - \bar{x}^{k} \|)^2
    \le 2 (\| x^{k} - \bar{x}^{k} \|^2 +  \| x^{k+k^{\prime}} - \bar{x}^{k}\|^2)
   \\ & \le 2( 2\mathrm{dist}(x^{k}, \Omega^*)^2
    + C_5 \tau_{k}) \quad
    \forall k>\max\{\bar{k}, k_1, k_2\}, k^{\prime} \ge 1,
\end{align*}
where the last inequality follows from \eqref{eq-xk-kprime}.
For fixed $k> \max\{\bar{k}, k_1, k_2\}$, by passing to the limit as $k^{\prime} \to \infty$ in the last display, we obtain that
\begin{align*}
    \| x^{k} - \bar{x} \|^2  \le 4 \mathrm{dist}(x^{k}, \Omega^*)^2
    + 2 C_5 \tau_{k}
    \quad \forall k> \max\{\bar{k}, k_1, k_2\}.
\end{align*}
Combining the last display with the definition of $\Delta_{k}$ in \eqref{eq-def-tau-Delta},
we deduce that
\begin{align}
      \Delta_{k+1} & \le 4 \mathrm{dist}(x^{k+1}, \Omega^*)^2
    + (C_1 +2 C_5) \tau_{k+1}  +  C_2 \mu_{k+1} {\widehat Q}_{k+1}
    \notag
    \\& \le  4 \kappa^2 {\widehat Q}_{k+1}^{2(1-\theta)} + 3 C_5 \tau_{k+1}
     \le 4 \kappa^2 {\widehat Q}_{k+1}^{p^{-1}} + 3 C_5 \tau_{k+1}
     \quad \forall k> \max\{\bar{k}, k_1, k_2\},
     \label{eq-Delta-up-psi-tau}
\end{align}
where the second inequality follows from \eqref{eq-dist-psi} and \eqref{eq-def-C5},
and the last inequality holds thanks to $0 \le {\widehat Q}_{k+1} \le 1$ (see \eqref{eq-def-psi-k} and \eqref{eq-k2}) and upon defining $p\coloneqq \max\{\frac{1}{2(1-\theta)}, 1 \} \in [1, \infty)$.
Furthermore, it follows from \eqref{eq-mu-bd} that, for any $k> \max\{\bar{k}, k_1, k_2\}$,
\begin{align}
     & C_0 b_1 (1 + k)^{-\bar r} \Delta_{k+1}^{p}
     \le C_0 \mu_{k} \Delta_{k+1}^{p}
      \overset{\mathrm{(a)}}{\le}  C_0 \mu_{k} \left(4 \kappa^2 {\widehat Q}_{k+1}^{p^{-1}} + 3 C_5 \tau_{k+1}\right)^{p}
     \notag
     \\& \overset{\mathrm{(b)}}{\le}  2^{p - 1}  C_0 \mu_{k} \left(
     (4 \kappa^2)^{ p } {\widehat Q}_{k+1}
     + (3 C_5 \tau_{k+1})^{ p }
     \right) \notag
     \\ & \overset{\mathrm{(c)}}{\le}  2^{p - 1} (4 \kappa^2)^{ p }
     (\Delta_{k} - \Delta_{k+1}) + 2^{p - 1}  C_0  (3 C_5)^{ p }  \mu_{k}\tau_{k+1}^{ p }
     \notag
     \\ & \overset{\mathrm{(d)}}{\le}  2^{p - 1} (4 \kappa^2)^{ p }
     (\Delta_{k} - \Delta_{k+1}) +  2^{p - 1}  C_0 (3 C_5)^{ p }  b_2^{2p+1} (2 \bar{r} - 1)^{-p}
     (1+k)^{-(2 \bar{r} - 1)p-\bar{r}},
     \notag
\end{align}
where (a) follows from \eqref{eq-Delta-up-psi-tau}, (b) follows from Jensen's inequality (since $p\ge 1$),
(c) holds because of \eqref{eq-mu psi-Delta},
and (d) follows from \eqref{eq-mu-bd} and \eqref{eq-tau-up-bd}.

Rearranging terms in the above display, we obtain that
\begin{equation*}
    \Delta_{k+1} + a_0 (1 + k)^{-\bar r} \Delta_{k+1}^{p} \le \Delta_{k}  +  a_1 (1 + k)^{-(2 \bar{r} - 1)p-\bar{r}}
    \quad \forall k \ge N_1
\end{equation*}
with
$$
N_1 \coloneqq \max\{\bar{k}, k_1, k_2\}+1, \quad a_0 \coloneqq
\frac{b_1 C_0}{2^{p-1}
(4 \kappa^2)^{ p }}
\quad \mbox{and} \quad
a_1 \coloneqq
b_2 C_0 \left(  \frac{
3 C_5 b_2^2 }
{4 \kappa^2(2 \bar{r} - 1)}
\right)^{ p }.
$$
By this and Lemma~\ref{lm-rate},
there exist $a_2 > 0$ and $N_2 \in \mathbb{N}_{0}$ such that
 \begin{equation*}
    \Delta_{k} \le a_2 (1+k)^{-\widetilde{s}} \quad \forall k\ge N_2
\end{equation*}
with
\begin{equation*}
    \widetilde{s} = \begin{cases}
    2 \bar{r} - 1 & \mbox{if}\ p = 1,
    \\ \min\left\{ 2 \bar{r} - 1, \frac{1- \bar r}{p - 1}   \right\}
    & \mbox{if}\ p > 1.
\end{cases}
\end{equation*}

Since $p = \max\{\frac{1}{2(1-\theta)}, 1 \}$,
one can see from the last display and the definition of $s$ in \eqref{eq-rate-s} that $\widetilde{s} = 2s$.
Combining this with the definition of $\Delta_{k}$ in \eqref{eq-def-tau-Delta},
we obtain that \eqref{eq-conv-rate} holds with $\widetilde{\kappa} = \sqrt{a_2}$ and $\widetilde{k} = N_2$.
\end{proof}

\section{Numerical experiments}\label{sec5}
In this section, we test the performance of Algorithm~\ref{alg-sl-sm-esqm} on a synthetic problem.\footnote{All codes are written and executed in Matlab  R2025b on a machine equipped with a 12th Gen Intel(R) Core(TM) i7-12700 CPU at 2.10 GHz, 32 GB of RAM, and running Windows 11 Enterprise.
		The codes are available at \href{https://github.com/XuJiefeng-CN/sESQM-GL}{https://github.com/XuJiefeng-CN/sESQM-GL}.}
	Specifically, we consider the following group Lasso regularized minimization problem with polynomial constraints:
	\begin{align}\label{opt-group-ell-one-homo}
		& \min_{x\in \mathbb{R}^{n}}\ \frac{1}{2}x^{\top} Q x + q^{\top}x + \tau \sum_{J \in \mathcal{J}} \|x_{J}\|_{2} \notag 
		\\ &\ \mbox{s.t.}\quad
		c_{i}(x) \coloneqq \sum_{j_1\in [n], \ldots, j_d \in [n]} a_{j_1, \ldots, j_d}^{(i)}
		x_{j_1}\cdots x_{j_d} + x^{\top} q^{i} -  b_i \le 0 \quad \forall i \in [m], \notag
		\\ & \qquad \ \ x \in \mathcal{C} \coloneqq \{ z \in \mathbb{R}^{n} : \| z_{J} \|_{2} \le M \ \ \ \forall J\in \mathcal{J} \},
	\end{align}
	where $Q\in \mathbb{R}^{n \times n}$, $q \in \mathbb{R}^{n}$, $\tau > 0$,
	$\mathcal{J}$ is a partition of $[n]$,
	$x_{J}$ is the subvector of $x\in \mathbb{R}^{n}$ indexed by $J \in \mathcal{J}$,
	$q^{i} \in \mathbb{R}^{n}$, $A^{(i)} = (a_{j_1, \ldots, j_d}^{(i)})$ is a real-valued array with size $n \times \cdots \times n$ of order $d \ge 1$ for each $i\in [m]$, $b \in \mathbb{R}^{m}_{++}$,
    and $M>0$.
	Let
	\begin{align}\label{eq-numerical-ex}
		f(x) \coloneqq \frac{1}{2}x^{\top} Q x + q^{\top}x,\
		P_1(x) \coloneqq \tau \sum_{J \in \mathcal{J}} \|x_{J}\|_2,\  P_2(x) \coloneqq 0, \
		h(y) \coloneqq \max_{i\in [m]} y_i.
	\end{align}
	Then, one can see that \eqref{opt-group-ell-one-homo} is an instance of \eqref{opt-dc-hc}, and
	satisfies Assumption~\ref{ass-gen} since the origin is a strictly feasible point for \eqref{opt-group-ell-one-homo}.
	
	We perform numerical experiments on random instances of \eqref{opt-group-ell-one-homo}, which are generated as follows.
	Let $(n, m)$ be the problem dimensions.
	We set $\tau = 1$ and $M=20$.
	We consider the case where $n$ is even, and set $\mathcal{J} = \{ \{1, 2\}, \{3, 4\}, \ldots \{n-1, n \} \}$.
	We generate $b$ using the Matlab command \verb|0.1+0.9*rand(m,1)|.
	We consider the following convex and nonconvex settings.
	\begin{enumerate}[{\rm \bf (a)}]
		\item {\bf Convex setting.}
		Let $d = 2$, $p=100$, $n = 1000$ and $m = 500$.
		We let $U \in \mathbb{R}^{p\times n}$ be the column-normalized matrix obtained from a random matrix $B\in \mathbb{R}^{p\times n}$ whose entries are i.i.d. standard Gaussian.
		We set $Q=U^{\top} U$ and $q = - U^{\top}\widetilde{q}$,
		where $\widetilde{q}$ is generated by the Matlab command
		\verb|10+randn(p,1)|.
		Then $f$ in \eqref{eq-numerical-ex} is convex and can be reformulated as $f(x) = \frac{1}{2}\| U x - \widetilde q\|_2^{2}  -\frac12\|\widetilde q\|_2^{2}$ for all $x\in \mathbb{R}^{n}$.\footnote{In our numerical test, we compute $f$ via this formula, which takes $O(np)$ flops.}
		Moreover, for each $i\in [m]$, we let $A^{(i)} = U_i \mathrm{Diag}(a^{i}) U_i^{\top}$, where $U_i\in \mathbb{R}^{n\times n}$ is an orthogonal matrix generated by the Matlab command \verb|qr(randn(n))| and $a^i \in \mathbb{R}^{n}$ is generated by the Matlab command \verb|100*sprand(n,1,0.1)|.\footnote{ Here, for a vector $y$, we let ${\rm Diag}(y)$ denote the diagonal matrix with the $i$th diagonal entry being $y_i$ for all $i$.}
		We let $q^{i}\in \mathbb{R}^{n}$ be a random vector that has i.i.d. Gaussian entries with mean $10$ and variance $1$, for each $i\in [m]$.
		Then, the function $c_i$ in \eqref{opt-group-ell-one-homo} is convex for each $i\in [m]$.
		Therefore, the instance of \eqref{opt-group-ell-one-homo} generated as above is a convex program and satisfies Assumption~\ref{ass-cq} as the origin is strictly feasible.
		In this setting, we let $x^{0} = {\rm Proj}_{\mathcal{C}}(U^{\dagger}\widetilde{q})$ be the initial point of Algorithm~\ref{alg-sl-sm-esqm}.
		
		\item {\bf Nonconvex setting.}
		Let $d = 3$, $n = 180$ and $m = 100$.
		We set $Q = U \mathrm{Diag}(a) U^{\top}$, where $U\in \mathbb{R}^{n\times n}$ is an orthogonal matrix generated by the Matlab command \verb|qr(randn(n))|, and $a \in \mathbb{R}^{n}$ is generated by the Matlab command \verb|20*(2*rand(n,1)-1)|.
		We let $q\in \mathbb{R}^{n}$ be a random vector that has i.i.d. Gaussian entries with mean $10$ and variance $1$.
		We generate $A^{(i)}$ by the Matlab command \verb|2*rand([n*ones(1,d)])-1| followed by a symmetrization, and let $q^i = 0$, for each $i\in [m]$.
		By Proposition~\ref{prop-ex}, the instance of \eqref{opt-group-ell-one-homo} generated as above satisfies Assumption~\ref{ass-cq}.
		Hence, Algorithm~\ref{alg-sl-sm-esqm} can be applied to solving \eqref{opt-group-ell-one-homo}.
		In this setting, we let $x^{0} = {\rm Proj}_{\mathcal{C}}(-Q^{-1}q)$ be the initial point of Algorithm~\ref{alg-sl-sm-esqm}.
	\end{enumerate}
	
	We now present the parameter settings of Algorithm~\ref{alg-sl-sm-esqm}.
	First, we define, for all $\mu>0$, $t \in \mathbb{R}$ and $y\in \mathbb{R}^{m}$,
	\begin{align*}
    \varphi_{\mu} (t)  \coloneqq \begin{cases}
        [t]_{+} & \mbox{if}\ |t| > \frac{\mu}{2},
        \\ \frac{t^2}{2 \mu} + \frac{t}{2} + \frac{\mu}{8}
        & \mbox{if}\ |t| \le \frac{\mu}{2},
        \end{cases}
		\quad \mbox{and} \quad h_{\mu}(y)
		\coloneqq  \inf_{u \in \mathbb{R}^{m}} \{ h(u)  + \tfrac{1}{2\mu} \| u - y \|^2_2 \} + \tfrac{\mu}{2}.
	\end{align*}
    Then, according to Proposition~\ref{cor-mps}, \cite[Theorem~4.3]{25SG} and the discussions in \cite[Sections~4.2.1 \& 4.2.3]{25SG}, we see that
    $\{\varphi_{\mu}\}_{\mu>0}$ and $\{h_{\mu}\}_{\mu>0}$ satisfy the conditions in Remark~\ref{ass-sm}
    with $\alpha_3^{\varphi} = \frac{1}{8}$.\footnote{The value of $\alpha_3^{\varphi}$ comes from the value of ${\rm w}_\sigma$ in  \cite[Section~4.2.1]{25SG}, thanks to the definition of optimal inner smoothing in \cite[Definition~4.5]{25SG} and the distance between functions in \cite[Eq.~(2.2)]{25SG}.}
	Hence, we set $\alpha_3^{\varphi} = \frac{1}{8}$ in Algorithm~\ref{alg-sl-sm-esqm}.
    Moreover, we have that $\partial h(0) = \{\lambda \in \mathbb{R}_+^{m} : \sum_{i=1}^{m} \lambda_i = 1 \}$, and that,
    for every $\mu>0$ and $y \in \mathbb{R}^{m}$,
	\begin{align*}
		& {\rm prox}_{\mu h }(y) = y - \mu {\rm Proj}_{\partial h(0)}( y/ \mu ),
		\ \ \nabla h_{\mu}(y) = \frac{1}{\mu} ( y -  {\rm prox}_{\mu h }(y))
		= {\rm Proj}_{\partial h(0)}( y/ \mu ),
		\\ & h_{\mu}(y) = h ({\rm prox}_{\mu h }(y)) + \tfrac{\mu}{2}\| \nabla h_{\mu}(y) \|_2^{2} + \tfrac{\mu}{2}.
	\end{align*}
	
	Next, we discuss the setting of $\{\mu_k\}$.
	Specifically, we let $n_0 = 200$, $\nu_0=1/(n_0+1)$ and $K = 5000$.
	Let $\bar r \in (0.01, 1)$ and $\bar s\ge 0$.
	For each $k\in \mathbb{N}_{0}$, let $k_1$, $k_2\in \mathbb{N}_{0}$ satisfy $k = k_2 (n_0+1) + k_1$ and $k_1\le n_0$,
	and let $\bar k \coloneqq k_2 (n_0 +1) + \nu_{0} k_1$.
	We construct $\{\mu_{k}\}$
	as follows: set $\mu_0 = 10^6$ and
	$
		\mu_{k} = \mu_{0}\left( \bar k +
		1\right)^{-r_{\bar k}} ({\rm log}( \bar k +
		3 ))^{-s_{\bar k}}
    $ for all $k\in \mathbb{N}$,
	where $r_{j} \coloneqq 0.01 + \min\{1, \frac{j}{K}\} (\bar r - 0.01)$
	and $s_{j} \coloneqq \min\{1, \frac{j}{K}\} \bar s$
	for all $j\in \mathbb{N}_0$.
	Following the same line of reasoning as in \cite[Footnote~14]{25XPS},
	one can verify that the above sequence $\{\mu_k\}$ satisfies Assumption~\ref{ass-mu}(i).
	
	We set $\widehat{\gamma}_0 = 10$ and $\widehat{\gamma}_{t} = \widehat{\gamma}_0(1+t)^{-0.5}$ for $t\in \mathbb{N}$.
	We set $\overline{L} = 10^{11}$, $\underline{L} = 10^{-11}$, $c_1 = 0.001$ and $c_2 = 10^{-3}/\max\{1, \gamma_0\mu_0\}$.
	For $k = 0$, we let $L_{k,0} = 1$.
	For $k\ge 0$, if $\sqrt{|\langle s^{k}, z^{k} \rangle|}> 10^{-12}$ and $L_{\mbox{\footnotesize BB}} \coloneqq \mu_{k+1} \| z^{k} \|^2/|\langle s^{k}, z^{k} \rangle| \in [\underline{L}, \overline L]$, then let $L_{k+1,0} = L_{\mbox{\footnotesize BB}}$; otherwise, let $L_{k+1,0} = \min\{\overline L, \max\{\underline{L}, \frac{1}{2} L_{k,0}\} \}$;
	here $s^{k} \coloneqq x^{k+1} - x^{k}$ and
	$$
	z^{k} \coloneqq \gamma_{k}( \nabla f(x^{k+1}) - \nabla f(x^{k})) + \nabla g_{\mu_{k+1}}(x^{k+1})
	- \nabla g_{\mu_{k}}(x^{k}).
	$$
	
	We terminate Algorithm~\ref{alg-sl-sm-esqm} when it holds that, for all integers $j \in [\max\{0, k-3\}, k]$,
	\begin{align}
		&  \frac{\|x^{j+1} \!-\! x^{j}\|}{ \min\{1, \mu_{j} \gamma_j\} a_j} \! \le\!  \epsilon, \
		\frac{\max\{g(x^{j+1}), | \lambda_{j+1} g(x^{j+1}) | \} }{a_j} \!\le\!   \epsilon, \
		\frac{g(x^{j+1})\! -\! \langle c(x^{j+1}),  u^{j+1} \rangle}{a_j} \!\le\! \epsilon, 
        \notag
	\end{align}
	where $\epsilon=10^{-5}$, and for all $j\in \mathbb{N}_0$, $a_j \coloneqq \max\{1,\|x^{j+1}\|\}$, and $\lambda_{j+1}$ and $u^{j+1}$ are defined by \eqref{eq-lambda-def} and \eqref{eq-nabla-h-mu-part}, respectively;
	the first criterion above is motivated by \eqref{eq-delta-upbd-omega},
	and the third criterion above implies that $u^{j+1} \in \partial_{a_j  \epsilon} h( c(x^{j+1}) )$.
	Moreover, the algorithm is also terminated if $k=K$, or if {\bf Step 2} in Algorithm~\ref{alg-sl-sm-esqm} is invoked more than $40$ times for some $k$.
	
	We investigate several choices of the parameters $(\bar r, \bar s)$ governing the rate of decrease of $\{\mu_k\}$, specifically, $\bar r \in \{0.3, 0.6, 0.9\}$ and $\bar s \in \{3, 6\}$.
    Fig.~\ref{fig-1} summarizes the results for the convex and nonconvex instances of \eqref{opt-group-ell-one-homo}, where Algorithm~\ref{alg-sl-sm-esqm} is compared with CVX (version 2.2) using SDPT3 (version 4.0) as the solver for the {\bf convex setting}.
    In the {\bf convex setting}, among the variants of Algorithm~\ref{alg-sl-sm-esqm}, we see that the iterates become feasible within 1000 iterations and larger $\bar r$ and $\bar s$ lead to faster decrease in objective values. Moreover, Algorithm~\ref{alg-sl-sm-esqm} is faster than CVX and does not compromise too much in terms of the terminating objective value. On the other hand, in the {\bf nonconvex setting}, while we also observe that the iterates become feasible within 1000 iterations, it is interesting to note that smaller $\bar r$ and $\bar s$ lead to faster decrease in objective values.

    	\begin{figure}[h]
		\caption{
			Performance of {\em s}ESQM under different choices of $(\bar r,\bar s)$ for convex and nonconvex instances of \eqref{opt-group-ell-one-homo}.
			In the subfigures, $\Delta_k\coloneqq |\psi(x^{k}) - \psi^{\rm cvx}|/\max\{1,|\psi^{\rm cvx}|\}$ with $\psi^{\rm cvx}$ being the optimal value returned by CVX.
			For the tables, `obj.' and `constr.' denote, respectively, the optimal value $\psi(x^{k})$ and the feasibility violation $g^{+}(x^{k})$ of the corresponding method,
			and `time' denotes the CPU time in seconds.
			The reported time for {\em s}ESQM does not include the time for computing the initial point $x^0$, and the time for CVX includes the time for automatic problem reformulation.
		}
        \label{fig-1}
		\centering
		\begin{minipage}[t]{0.45\textwidth}
			\vspace{0pt}
			\centering
			\includegraphics[width=\linewidth]{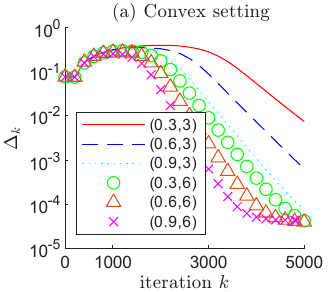}
		\end{minipage}
        \hfill
        \begin{minipage}[t]{0.45\textwidth}
			\vspace{0pt}
			\centering
			\includegraphics[width=\linewidth]{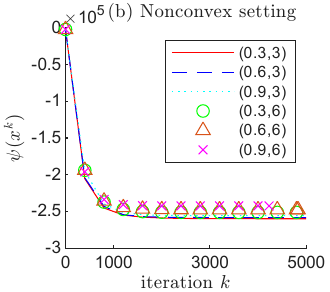}
		\end{minipage}
		\hfill
		\begin{minipage}[t]{0.45\textwidth}
			\vspace{0pt}
			\centering
			\includegraphics[width=\linewidth]{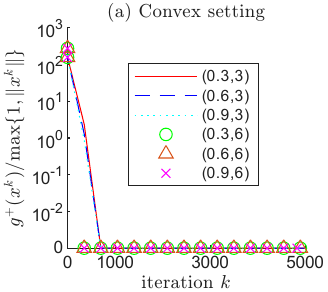}
		\end{minipage}
		\hfill
		\begin{minipage}[t]{0.45\textwidth}
			\vspace{0pt}
			\centering
			\includegraphics[width=\linewidth]{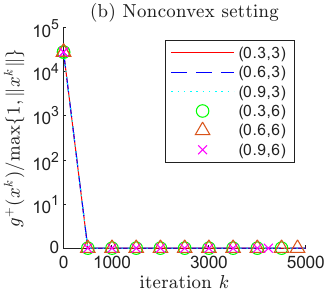}
		\end{minipage}

        \begin{minipage}[t]{0.45\textwidth}
			\vspace{5pt}
			\centering
			
			\scalebox{.9}{
				\begin{tabular}{cccc} 
\multicolumn{4}{c}{{(a)} Convex setting}\\ 
\hline 
$(\bar r, \bar s)$ & obj. & constr. & \text{time} \\
\hline 
$( 0.3, 3)$ &    -3980.401 & 0 &   740.8\\
$( 0.6, 3)$ &    -4007.801 & 0 &   737.9\\
$( 0.9, 3)$ &    -4010.014 & 0 &   723.6\\
$( 0.3, 6)$ &    -4010.149 & 0 &   720.5\\
$( 0.6, 6)$ &    -4010.160 & 0 &   722.4\\
$( 0.9, 6)$ &    -4010.145 & 0 &   702.7\\
       CVX  &    -4010.317 &  6.50e-07 &  2733.0\\ 
\hline 
\end{tabular}		
			}
		\end{minipage}
		\hfill
		\begin{minipage}[t]{0.45
				\textwidth}
			\vspace{5pt}
			\centering
			\scalebox{.9}{
				\begin{tabular}{cccc} 
\multicolumn{4}{c}{{(b)} Nonconvex setting}\\ 
\hline 
$(\bar r, \bar s)$ & obj. & constr. & \text{time} \\
\hline 
$( 0.3, 3)$ &    -259941.9 & 0 &  1026.2\\
$( 0.6, 3)$ &    -258337.1 & 0 &  1081.6\\
$( 0.9, 3)$ &    -255761.3 & 0 &  1169.3\\
$( 0.3, 6)$ &    -251312.5 & 0 &  1109.8\\
$( 0.6, 6)$ &    -247519.8 & 0 &  1044.8\\
$( 0.9, 6)$ &    -242500.7 & 0 &   931.1\\
\hline 
\end{tabular}	
			}
		\end{minipage}
	\end{figure}

    \ \\

\noindent {\bf Acknowledgements.} The work of the second author was supported in part by the Hong Kong Research Grants Council PolyU 15300423.

\noindent {\bf Data availability.}
The code used to generate our numerical results can be found at the following link: \href{https://github.com/XuJiefeng-CN/sESQM-GL}{https://github.com/XuJiefeng-CN/sESQM-GL}.


\end{document}